\documentclass[11pt, letterpaper]{amsart}
\usepackage{graphicx, float, color} 
\usepackage{thmtools, lscape, rotating}
\usepackage{thm-restate}
\usepackage{amsmath, comment, amsthm, enumerate, amssymb}
\usepackage{mathrsfs}
\usepackage{hyperref} 
\usepackage[latin1]{inputenc}
\theoremstyle{definition}
\newtheorem{defn}{Definition}[section]
\newtheorem{rem}[defn]{Remark}
\newtheorem{rmk}[defn]{Remark}

\theoremstyle{plain}
\newtheorem{thm}{Theorem}

\theoremstyle{plain}
\newtheorem{tthm}[defn]{Theorem}

\newtheorem{prop}[defn]{Proposition}
\newtheorem{lem}[defn]{Lemma}

\newcommand{\sgn}{\mathrm{\text{sgn}}}
 
\newcommand{\Z}{\mathbb{Z}} 
 
\newcommand{\hpc}{\vcenter{\hbox{\includegraphics[scale=.5]{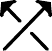}}}}
\newcommand{\hnc}{\vcenter{\hbox{\includegraphics[scale=.5]{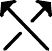}}}}
\newcommand{\hco}{\vcenter{\hbox{\includegraphics[scale=.5]{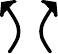}}}}
\UseRawInputEncoding

\begin{document}
\phantom{A}
\vspace{-25mm}
\title{A topological model for the HOMFLY-PT polynomial}
\author[C. Anghel]{Cristina Ana-Maria Anghel}
\author[C. Lee]{Christine Ruey Shan Lee}
\address[]{Universit\'e Clermont Auvergne-LMBP, Campus des C\'ezeaux 3, place Vasarely, 63178 Aubi\`ere, France; Institute of Mathematics “Simion Stoilow” of the Romanian Academy, 21 Calea Grivitei Street, 010702 Bucharest, Romania.}
\email[] {\tt cristina.anghel@uca.fr, cranghel@imar.ro}

\address[]{Texas State University, Department of Mathematics, San Marcos, TX}

\email[]{vne11@txstate.edu}

\vspace{-5mm}
\begin{abstract}
We give the first known topological model for the HOMFLY-PT polynomial constructed directly from link diagrams. More precisely, we prove that this invariant is given by graded intersections between explicit Lagrangian submanifolds in a {\em fixed configuration space} on a {\em Heegaard surface} for the link exterior. The submanifolds are supported on a collection of arcs and ovals on the Heegaard surface. We also obtain two topological models for the Jones polynomial via a Heegaard surface associated {\em to a link diagram}. 
 This opens up new avenues for constructing categorifications for the Jones and the HOMFLY-PT polynomials of a geometric nature. 

\end{abstract}

\maketitle
\vspace{-10mm}

\section{Introduction}
 Toward the goal of understanding the geometric information encoded in quantum link invariants for a link $L\subset S^3$, an approach is to realize these invariants as graded intersection pairings in configuration spaces. We call this a \emph{topological model}. Bigelow \cite{Bigelow2002} and Lawrence \cite{Lawrence1993} provided the first topological model for the Jones polynomial. Recently, the first author defined topological models for quantum generalizations of these invariants  \cite{Anghel2022Globalisation,Anghel2023TAMS,Anghel2023QT,Anghel2024Adv}. 

In \cite{Bigelow2007} Bigelow gave a homological model for the $A_N$ polynomials, which are 1-variable specializations of the HOMFLY-PT polynomial $P_L(a, z)$.
The number of particles in his configuration space for $A_N$ depends on $N$. This leads to a description of the HOMFLY-PT polynomial via infinite sequences of intersections in different configuration spaces. The question of finding a  topological model for the HOMFLY-PT polynomial in the same configuration space with a fixed number of particles has remained open. Moreover, all previous topological models rely on a choice of a braid representative for the link. This relates to the open problem of constructing a categorification of the HOMFLY-PT polynomial directly from link diagrams \cite{Abel, khovanov-rozansky} as the celebrated Khovanov-Rozansky homology  \cite{khovanov-rozansky} also depends on choosing a braid representative. 

Motivated by these questions, the purpose of this paper is to
\begin{enumerate}
\item {  construct a  topological model for the HOMFLY-PT polynomial in a fixed configuration space} (Theorem \ref{t.main}), and
\item {   give a new method for constructing topological models for quantum invariants directly from a Heegaard suface of $S
^3\setminus L$ that does not rely on braid representatives} (Theorem \ref{t.main2}). 
\end{enumerate}

\paragraph{\textbf{Main Results.}} \ 


\begin{thm}[Topological model for the HOMFLY-PT polynomial]
\label{t.main} For a link $L\subset S^3$ let $D$ be a diagram for $L$. 
Let $\Theta_H(D) \in \mathbb{Z}[a^{\pm 1}, z^{\pm 1}]$ be the state sum of graded intersections between explicit Lagrangian submanifolds in a fixed configuration space on the Heegaard surface $\Sigma$ defined from $D$, as in Definition \ref{intHintr}. Then this topological model recovers the HOMFLY-PT invariant: 
 \begin{equation*}
   \Theta_H(D) = P_L(a, z).  
 \end{equation*}
\end{thm}

As far as the authors know, this is the first topological model for the HOMFLY-PT polynomial that does not rely on the specialization to the $A_N$ polynomials \footnote{The $A_N$ polynomials are also known as the $\mathfrak{sl}(N)$ polynomials \cite{khovanov-rozansky1}.} and is supported on a fixed configuration space.  In contrast to previous topological models including that of the first author's work \cite{Lawrence1993, Bigelow2002, Bigelow2007, Anghel2022Globalisation, Anghel2023QT, Anghel2023TAMS, Anghel2024Adv}, our topological model does not rely on a choice of a braid representative and the braid action on a punctured disk. Therefore, our construction is a new approach to creating topological models. 

With this new method, we construct as well a new topological model of the (un-normalized) Jones polynomial $J_L(q)$ that recovers the Kauffman bracket. 

\begin{thm}[New topological model for the Jones polynomial]
\label{t.main2}
\hspace{0pt}\\
Let $\Theta_J(D) \in \mathbb{Z}[q^{\pm 1}]$ be the state sum of graded intersections in the same configuration space on our fixed Heegaard surface, as in Definition \ref{intJintr}. Then it provides a topological model for the Jones polynomial, relying just on the choice of the diagram $D$:
\begin{equation*}
\Theta_J(D) = J_L(q). 
\end{equation*}
\end{thm}

 Theorem \ref{t.main2} provides the first topological model for the Jones polynomial from graded intersections in a configuration space on the Heegaard surface $\Sigma$.  Compared to the geometric models via Floer theory for Khovanov homology by Seidel-Smith \cite{seidel-smith} and Manolescu \cite{manolescu} in terms of nilpotent slices of certain Hilbert schemes, our topological model only depends on the link diagram and does not require a choice of braid representative.

We give a summary of the steps of the  construction of our topological model and the sections in which the steps are done in detail.

\subsection{Configuration space of a Heegaard surface (Section \ref{ss.Heegaardsurface}-\ref{ss.configspace})}
We start with a diagram $D$ for our link and consider a Heegaard surface $\Sigma$ associated to it, which is in the exterior of our link. Our main topological tools are built up using graded intersections between Lagrangian submanifolds in the configuration space on a punctured version of $\Sigma$. We show that we can obtain the Jones polynomial and HOMFLY-PT polynomial from an explicit set of graded Lagrangian intersections in a fixed configuration space on the surface. Let us make this precise.

Let $n$ be the number of crossings of the link diagram $D$. First we fix a set of $8$ punctures on $\Sigma$ associated to each crossing (as in Figure \ref{f.paths}) and an additional fixed point and consider $\Sigma'$ to be the resulting Heegaard surface with $8n+1$ punctures. We will be working in the configuration space of $2n$ particles on the punctured surface $\Sigma'$: 
$$C_{2n}(\Sigma') := Conf_{2n}(\Sigma').$$

\subsection{Generic local system and homology groups (Section \ref{ss.localsystem})} 
We define a local system on this configuration space:
\begin{equation}
\begin{aligned}
&\Phi: \pi_1(C_{2n}(\Sigma')) \rightarrow \mathbb Z^{8n}\\
&\hspace{28mm} \langle \{x_j\}_{j \in \{1,...,8n\}}\rangle. \ \end{aligned}
\end{equation}
After that, we consider the covering of the configuration space $C_{2n}(\Sigma')$ associated to this local system. Our tools are the middle dimension homology groups of this covering: the Borel-Moore homology $\mathscr H^{lf}_{2n}$ and homology $\mathscr H_{2n}$, which are $\Z[x_1^{\pm 1},...,x_{8n}^{\pm 1}]$-modules (as defined in Definition \ref{Hom}).
 These homologies are related by a Poincar{\'e}-type duality:  
$$\ll ~,~ \gg: \mathscr H^{lf}_{2n} \otimes \mathscr H_{2n} \rightarrow\Z[x_1^{\pm 1},...,x_{8n}^{\pm 1}] 
 \ \ \ (\text{Proposition } \ref{P:3'''}).$$
We use homology classes that are given by lifts of Lagrangian submanifolds from the base space. The recipe for constructing such submanifolds is to fix a collection of $2n$ arcs (or circles) on the punctured Heegaard surface $\Sigma'$ and consider their products and quotient to the unordered configuration space. We call such a set of curves the ``geometric support'' for our submanifolds. Then our pairing can be read off from the base space: it is parametrized by intersection points between the geometric supports and graded by the local system. 

Both topological models are constructed based on the same underlying topological data:
\begin{itemize}
    \item the configuration space of $2n$ particles on the fixed punctured Heegaard surface $\Sigma'$:  $C_{2n}(\Sigma')$  
    \item the intersection pairing between homologies of its fixed $\Z^{8n}$ covering space from above.
   \end{itemize}
The difference is reflected in the choice of homology classes and specialization of coefficients, as follows.

\subsection{Homology classes associated to a state}
\subsubsection{Classes for Jones polynomial: indexed by Kauffman states} For the Jones polynomial, our idea is to fix a resolution into state circles associated to a Kauffman state $\sigma$ and to ``push the set of state circles'' onto the surface $\Sigma'$. Then, we transform the state circles into two collections of arcs and ovals. 
We define two sets of curves on $\Sigma'$: $F(\sigma)$ (given by arcs between the punctures) and $L(\sigma)$ (given by ovals around punctures):
$$\mathscr F(\sigma)\in \mathscr H^{lf}_{2n} \ \ \ \text { and } \ \ \ \mathscr L(\sigma)\in \mathscr H_{2n} \ (\text{Figure } \ref{f.arcsovalso}).$$

\subsubsection{Classes for HOMFLY-PT polynomial: indexed by renormalized Kaufmann states}
For the HOMFLY-PT invariant there are two additional subtle ideas. On the combinatorial side, we develop a state sum formula using descending diagrams and renormalized Kauffman states $\sigma^K_H$. These states are associated to Jaeger states $\sigma_H$ in the Jaeger state sum formula (for simplicity of notation we use $\sigma_H = \sigma_{H, P}$ where $\sigma_{H, P}$ is the notation for a Jaeger state in Section \ref{S:H}). 

The second idea is to encode geometrically a renormalization procedure that appears in the state sum formula by cutting our geometric support for the homology classes. We do this by fixing a cutting point on the diagram and once we consider a state, we define the set of arcs and ovals by the same technique as for the Jones polynomial, but we shrink the oval associated to this cutting point. For this model we fix a point on our diagram and call it a ``cutting point.'' Then, once we have a renormalized Kauffman state $\sigma^K_H$, we associate another two sets of curves $F(\sigma^K_H)$ (given by arcs between the punctures) and $L(\sigma^K_H)$ (given by ovals around punctures where we shrink the oval around the cutting point), as in Figure \ref{f.arcsovalsoH}. They lead to the classes:
$$\mathscr F^H(\sigma^K_H)\in \mathscr H^{lf}_{2n} \ \ \ \text { and } \ \ \ \mathscr L^H(\sigma^K_H)\in \mathscr H_{2n} \ (\text{Figure } \ref{f.arcsovalsoH}).$$

\subsection{Specialization associated to a state}

We remark that the above homology classes depend on the choice of a (renormalized) Kauffman state, but they belong to the generic homology groups of the covering, which are intrinsic and do not depend on the state. Now, we define a specialization of coefficients associated to a fixed state, as in Definition \ref{specQ}.
\begin{defn}[Change of coefficients for a state]
\hspace{0pt}\\
Consider $\alpha^{\sigma}_Q: \Z[x_1^{\pm 1},\ldots,x_{8n}^{\pm 1}] \rightarrow \Z[Q^{\pm1}]$, which is constructed in a geometric manner such that it satisfies a certain {\em Monodromy Requirement} related to evaluations around the punctures from the surface (Definition \ref{reqspec}).
\end{defn}

Both models are then given by the generic intersection pairing $\ll ~,~ \gg$, between the classes associated to the state $\sigma$ or $\sigma_H^K$, then specialized through $\alpha^{\sigma}_Q$ or $\alpha^{\sigma_H^K}_Q$, for two different choices of $Q$. 
\begin{defn}[Specialization of coefficients for the Jones polynomial]  For a particular choice of $Q=Q_J$, we have the induced specialization $\alpha^{\sigma}_J: \Z[x_1^{\pm 1},...,x_{8n}^{\pm 1}] \rightarrow \Z[q^{\pm1}](1-q-q^{-1})$ and the specialized intersection pairing as below (see Definition \ref{specJ}):
 \begin{equation*}
\ll ~,~ \gg_{\alpha^{\sigma}_J}: \mathscr H^{lf}_{2n}\mid_{J} \otimes \mathscr H_{2n}\mid_{J} \rightarrow \Z[q^{\pm 1}](1-q-q^{-1}).
\end{equation*}
\end{defn}

\begin{defn}[Specialization for the HOMFLY-PT polynomial] 
For the case of the HOMFLY-PT polynomial we will use the change of coefficients $\alpha^{\sigma^K_H}_H:  \Z[x_1^{\pm 1},...,x_{8n}^{\pm 1}]  \rightarrow \Z[a^{\pm1},z^{\pm1}](1-\frac{a-a^{-1}}{z})$ and the specialized groups (see Definition \ref{specH}):
 \begin{equation*}
\ll ~,~ \gg_{\alpha^{\sigma^K_H}_H}: \mathscr H^{lf}_{2n}\mid_{H} \otimes \mathscr H_{2n}\mid_{H} \rightarrow\Z[a^{\pm1},z^{\pm1}](1-\frac{a-a^{-1}}{z}).
\end{equation*}
\end{defn}

\subsection{Topological formula for the Jones polynomial}
Our topological model $\Theta_J$  interprets the Kauffman state sum definition of the Jones polynomial in terms of graded intersections $\ll ~,~ \gg_{\alpha^{\sigma}_J}$ between homology classes $\mathscr F(\sigma)$ and $\mathscr L(\sigma)$ over all choices of Kauffman states $\sigma$, and linking numbers  between certain curves $D_\sigma(\alpha)$ on the Heegaard surface $\Sigma'$ and $D$ ($i(D_\sigma(\alpha), D)$), as follows.
\begin{defn}[Intersection form for the Jones polynomial]\label{intJintr}
Define $\Theta_J(D)$ as follows (see Definition \ref{intformulaJ}):  
 \begin{align*}   
    \Theta_J(D)(q):=  \sum_{\sigma \text{ a Kauffman state}} &(-1)^{3w(D)} \cdot \\  
    & \cdot (-q)^{\frac{-i(D_\sigma(\alpha), D)+3w(D)}{2}} \ll \mathscr F(\sigma) , \mathscr L(\sigma)  \gg_{\alpha^{\sigma}_{J}}.
    \end{align*}
\end{defn}

Theorem \ref{t.main} shows that this is a topological model for the Jones polynomial.
\subsection{Topological formula for the HOMFLY-PT polynomial}
Our topological model defines the HOMFLY-PT polynomial via a state sum of intersection pairings between classes $\mathscr F(\sigma^K_H)$ and $\mathscr L(\sigma^K_H)$ indexed by {\em \bf renormalized Kauffman states} $\sigma^K_H$. These states correspond to {\em \bf Jaeger states} $\sigma_H$ (Lemma \ref{renkaf}) in the state sum formulation of the HOMFLY-PT polynomial due to Jaeger \cite{Jaeger}.  
\begin{defn} [Intersection form for the HOMFLY-PT polynomial]\label{intHintr}
Consider the function $\Theta_H(D)$ defined as follows (see Definition \ref{intformulaH}): 
\begin{align*}
\Theta_H(D)(a,z) &:= \sum_{\substack{\sigma^K_H \text{ a renormalized} \\ \text{Kauffman state} \\ \text{associated to a} \\ \text{ Jaeger state $\sigma_H$}}} \sgn(\sigma^K_H) \cdot a^{i^a(\sigma^K_H)} \cdot z^{i^z(\sigma^K_H)}  \cdot \\
& \hspace{3.2cm} \cdot \ll \mathscr F^H(\sigma^K_H), \mathscr L^H(\sigma^K_H) \gg_{\alpha^{\sigma^K_H}_{H}}.
\end{align*}
\end{defn} 
By Theorem \ref{t.main}, $\Theta_H(D)(a,z)$ provides a topological model for the HOMFLY-PT polynomial of $L$. \begin{rmk}[Geometric meaning of renormalization and Jaegaer's coefficients] An extra subtlety in the construction comes from modifying the homology classes by ``cutting the strand", which encodes the renormalization procedure that appears in the HOMFLY-PT invariant. 

Secondly, we  provide a new geometric interpretation (Lemma \ref{e.giHOMFLY}) of Jaeger's coefficient polynomial $ \sgn(\sigma^K_H) \cdot a^{i^a(\sigma^K_H)} \cdot z^{i^z(\sigma^K_H)}$ from \cite{Jaeger} that appears for each state $\sigma_H$. We do this in terms of linking numbers of certain twisted curves on the punctured Heegaard surface $\Sigma'$, see Definition \ref{dtwist}. The method is inspired by the relationship between the Jones polynomial and boundary slopes of essential surfaces \cite{FKP-adequate} (Lemma \ref{l.intslope} and Remark \ref{r.tbdslope}). 
\end{rmk}

\begin{rmk}[Geometric perspective for the HOMFLY-PT polynomial]
This shows that the HOMFLY-PT polynomial of a link is given by the set of graded intersections between the homology classes associated to the set of states:
\begin{align*}
P_L(a,z) \  \longleftrightarrow \ & \{ \text{graded intersections } \ll \mathscr F^H(\sigma^K_H), \mathscr L^H(\sigma^K_H) \gg\\
& \text{   in the configuration space of } 2n \text{ points}\\
& \text{   on the punctured Heegaard surface } \Sigma' \ \}_{\sigma^K_H \text{ state}. }
\end{align*}

\end{rmk}

\begin{rmk}[Model for Jones polynomial vs. HOMFLY-PT polynomial]

Our first topological model $\Theta_J(D)(q)$ leads to the Jones polynomial and the second model $\Theta_H(D)(a,z)$ gives the HOMFLY-PT polynomial. This, in turn, leads to the Alexander polynomial and normalized Jones polynomial through variable specializations. Comparing the two models for the case of the Jones polynomial, we see that in one case we obtain the un-normalized version (where we do not cut the strand) and in the other we obtain the normalized version, where we cut the strand. 
\end{rmk}
\subsection{Connection to representation theory}

Philosophically we expect topological models to have a deeper meaning coming from representation theory.  The fact that the HOMFLY-PT invariant contains the Alexander polynomial as a specialization means, from a quantum point of view, that it contains a non-semisimple invariant. Geometrically this non-semisimplicity is encoded in the procedure of cutting of the strand and shrinking one oval from our geometric support. 

It would be interesting to study this phenomenon further to clarify the connections between our topological model for the HOMFLY-PT polynomial and the representation-theoretic information encoded by this invariant.

\subsection{Connection to Floer-theoretic invariants and categorifications}
From the topological model of the HOMFLY-PT polynomial, we can obtain a topological model of the Alexander polynomial by specializing the variables $a \mapsto 1$ and $z \mapsto t^{1/2} - t^{-1/2}$.  Since the definition of the Heegaard surface $\Sigma$ used in our construction closely follows \cite{OS}, where Ozsv{\'a}th-Szab{\'o} determine the knot Floer homology of alternating knots with their setup, it is natural to ask about the relationship between two topological models for the Alexander polynomial: our model, and the one from knot Floer homology. We plan to investigate this question in the future. 

\vspace{-2mm}
\section*{Acknowledgements}
We would like to thank Jake Rasmussen, Emmanuel Wagner, Mikhail Khovanov, Matt Hedden, Ciprian Manolescu and Roland van der Veen for their comments and suggestions on the drafts of this paper. The first author acknoledges the support offered by the EPSRC Programme Grant EP/W007509/1. She also acknowledges the support of the ANR grant ANR-24-CPJ1-0026-01 at Universit\'e Clermont Auvergne - LMBP and partial support by grants of the Ministry of Research, Innovation and Digitization, CNCS - UEFISCDI, project numbers  PN-IV-P2-2.1-TE-2023-2040 and PN-IV-P1-PCE-2023-2001, within PNCDI IV.
The second author acknowledges the support from NSF Grant No. 2244923 and the University of Leeds for providing excellent working conditions. 
\tableofcontents

\section{A local system} \label{s.localsystem}

\subsection{Punctures and base points on a Heegaard surface of a link diagram} \label{ss.Heegaardsurface}

As defined in \cite{Moriah}, a \textit{compression body} is a 3-manifold $V$ obtained from a surface $\Sigma$ cross an interval $[0, 1]$ by attaching a finite number of 2-handles and 3-handles to $\Sigma \times \{0 \}$. The component $\Sigma \times \{1\}$ of the boundary is denoted by $\partial_+V$, and $\partial V \setminus \partial_+V$ is denoted by $\partial_-V$. 

Let $L$ be a link in $S^3$. A (generalized) \textit{Heegaard splitting} for a link exterior $E(L)$ is a decomposition $E(L) = V \cup_\Sigma W$, where $V$ is a compression body with $\partial_+V \simeq \Sigma$ and $\partial_-V = \partial N(L)$, and $W$ is a handlebody. The surface $\Sigma$ is called a \textit{Heegaard surface} for the link $L$. 

Let $D$ be a diagram of a link $L$. We follow Osz\'ath-Szab\'o \cite{OS} in the construction of a Heegaard surface $\Sigma$ of the link $L$ from $D$.
To each crossing $\chi$ of the diagram we associate a $4$-punctured sphere as a local piece of the Heegaard surface:  
\begin{figure}[H]
\begin{center}
\def \svgwidth{.7\columnwidth}
\begingroup%
  \makeatletter%
  \providecommand\color[2][]{%
    \errmessage{(Inkscape) Color is used for the text in Inkscape, but the package 'color.sty' is not loaded}%
    \renewcommand\color[2][]{}%
  }%
  \providecommand\transparent[1]{%
    \errmessage{(Inkscape) Transparency is used (non-zero) for the text in Inkscape, but the package 'transparent.sty' is not loaded}%
    \renewcommand\transparent[1]{}%
  }%
  \providecommand\rotatebox[2]{#2}%
  \newcommand*\fsize{\dimexpr\f@size pt\relax}%
  \newcommand*\lineheight[1]{\fontsize{\fsize}{#1\fsize}\selectfont}%
  \ifx\svgwidth\undefined%
    \setlength{\unitlength}{694.01629591bp}%
    \ifx\svgscale\undefined%
      \relax%
    \else%
      \setlength{\unitlength}{\unitlength * \real{\svgscale}}%
    \fi%
  \else%
    \setlength{\unitlength}{\svgwidth}%
  \fi%
  \global\let\svgwidth\undefined%
  \global\let\svgscale\undefined%
  \makeatother%
  \begin{picture}(1,0.22173506)%
    \lineheight{1}%
    \setlength\tabcolsep{0pt}%
    \put(0,0){\includegraphics[width=\unitlength,page=1]{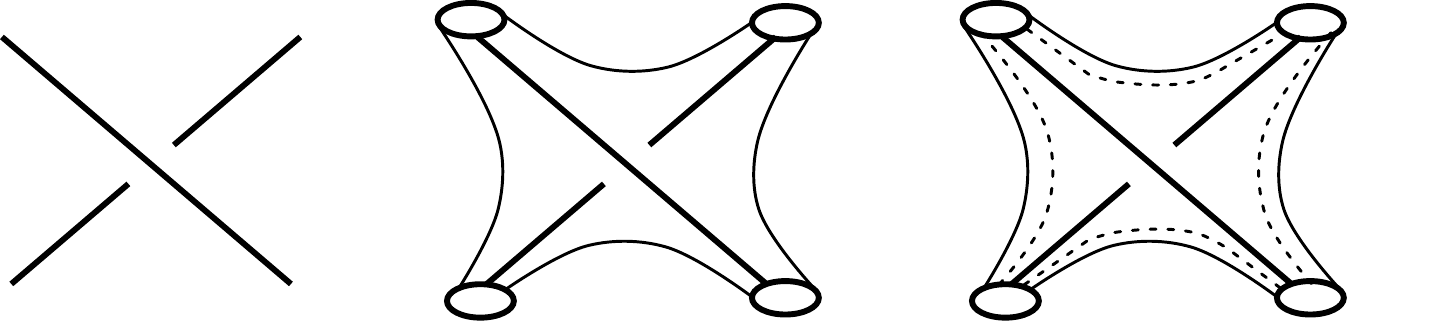}}%
    \put(0.78182503,0.19060168){\color[rgb]{0,0,0}\makebox(0,0)[lt]{\lineheight{1.25}\smash{\begin{tabular}[t]{l}$\alpha_1$\end{tabular}}}}%
    \put(0.90214916,0.10707468){\color[rgb]{0,0,0}\makebox(0,0)[lt]{\lineheight{1.25}\smash{\begin{tabular}[t]{l}$\alpha_2$\end{tabular}}}}%
    \put(0.78118254,0.01291594){\color[rgb]{0,0,0}\makebox(0,0)[lt]{\lineheight{1.25}\smash{\begin{tabular}[t]{l}$\alpha_3$\end{tabular}}}}%
    \put(0.6578783,0.10351061){\color[rgb]{0,0,0}\makebox(0,0)[lt]{\lineheight{1.25}\smash{\begin{tabular}[t]{l}$\alpha_4$\end{tabular}}}}%
  \end{picture}%
\endgroup%

\end{center}
\caption{\label{f.curves} Left to right: a crossing, a 4-punctured sphere at a crossing, $\alpha$ curves.}
\end{figure} 
To a crossing $\chi$ we associate four $\alpha$ arcs $\alpha_1(\chi), \alpha_2(\chi), \alpha_3(\chi), \alpha_4(\chi)$, corresponding to each quadrant of the crossing as also shown in Figure \ref{f.curves}. Formally, each boundary circle is transverse to the projection plane and centered on the knot strand.

The surface $\Sigma$ is built by identifying every pair of boundary circles on two 4-punctured spheres sharing the same knot strand, so that the $\alpha$ arcs are extended through each 4-punctured sphere, see Figure \ref{f.glue}.

\begin{figure}[H]
\begin{center}
\def \svgwidth{.5\columnwidth}
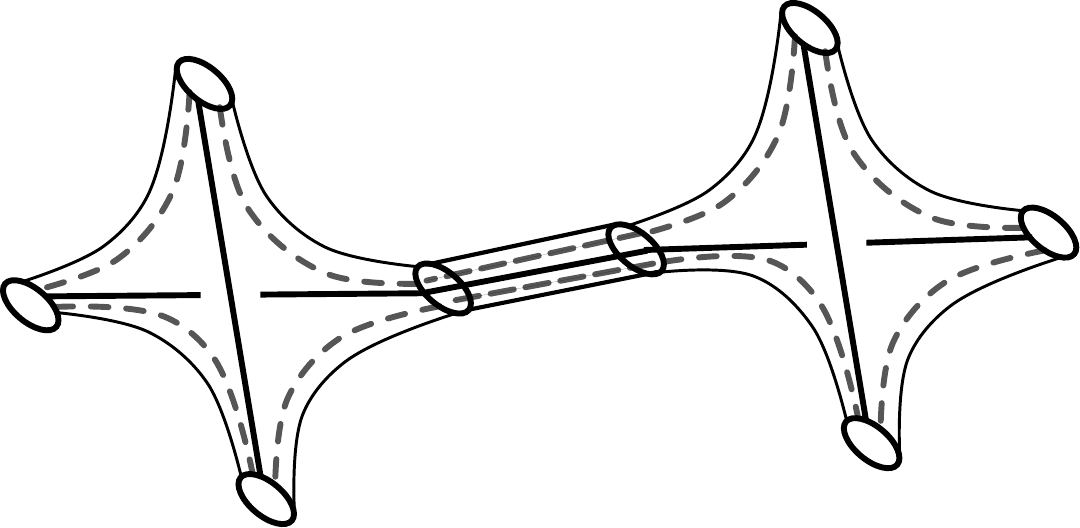
\end{center}
\caption{\label{f.glue} Gluing pieces of the Heegaard surface $\Sigma$ along a knot strand.}
\end{figure}

We see the surface $\Sigma$ is a Heegaard surface for $E(L)$  by  thickening the surface to $\Sigma\times [0, 1]$, so that $\Sigma \times \{0\}$ contains the link $L$ in its interior. At each crossing, add $2$-handles to $\Sigma \times \{0\}$ to fill out the space between two strands of a crossing and $\Sigma \times \{0\}$.  Finally, cap off the resulting sphere components with 3-handles. 

By definition, the space $\Sigma \times [0, 1]$ with the 2-and 3-handles attached is a compression body $V$ with $\partial_-V = \partial N(L)$. The complement of $V$ in $S^3$ is bound by $\Sigma$ and therefore a handlebody, and $V \cup_{\Sigma} (S^3\setminus V)$ defines a Heegaard splitting of $E(L)$. 

\subsection{Intersection points} \label{ss.intersectionpts}

For a link $L \subset S^3$, we will define the intersection pairing as arising from certain intersection points on the Heegaard surface $\Sigma$ associated to $L$ discussed in the previous section.

Isotope the link diagram so that 
\begin{itemize}
\item every crossing is oriented with the North, South, East, and West quadrant between a pair of strands as on the left in Figure \ref{f.intersectionpts}, and 
\item the link diagram lies on the surface, as in the following figure. 
\begin{figure}[H]
\begin{center}
\def \svgwidth{.5\columnwidth}
\begingroup%
  \makeatletter%
  \providecommand\color[2][]{%
    \errmessage{(Inkscape) Color is used for the text in Inkscape, but the package 'color.sty' is not loaded}%
    \renewcommand\color[2][]{}%
  }%
  \providecommand\transparent[1]{%
    \errmessage{(Inkscape) Transparency is used (non-zero) for the text in Inkscape, but the package 'transparent.sty' is not loaded}%
    \renewcommand\transparent[1]{}%
  }%
  \providecommand\rotatebox[2]{#2}%
  \newcommand*\fsize{\dimexpr\f@size pt\relax}%
  \newcommand*\lineheight[1]{\fontsize{\fsize}{#1\fsize}\selectfont}%
  \ifx\svgwidth\undefined%
    \setlength{\unitlength}{468.69808167bp}%
    \ifx\svgscale\undefined%
      \relax%
    \else%
      \setlength{\unitlength}{\unitlength * \real{\svgscale}}%
    \fi%
  \else%
    \setlength{\unitlength}{\svgwidth}%
  \fi%
  \global\let\svgwidth\undefined%
  \global\let\svgscale\undefined%
  \makeatother%
  \begin{picture}(1,0.32833025)%
    \lineheight{1}%
    \setlength\tabcolsep{0pt}%
    \put(0,0){\includegraphics[width=\unitlength,page=1]{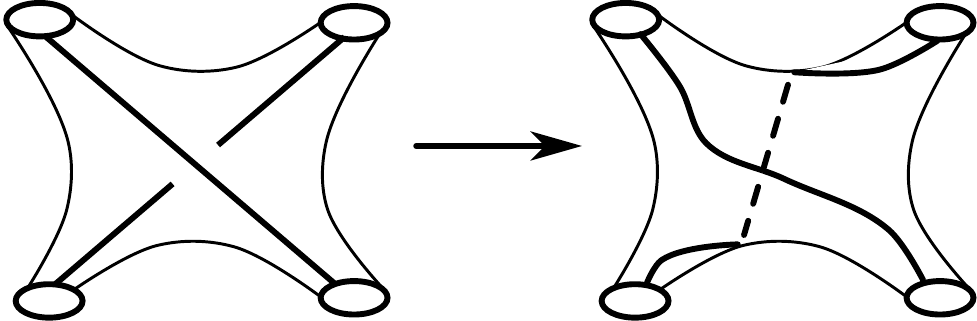}}%
    \put(0.42490404,0.10214688){\color[rgb]{0,0,0}\makebox(0,0)[lt]{\lineheight{1.25}\smash{\begin{tabular}[t]{l}Isotopy\end{tabular}}}}%
    \put(0.17213781,0.28490958){\color[rgb]{0,0,0}\makebox(0,0)[lt]{\lineheight{1.25}\smash{\begin{tabular}[t]{l}$N$\end{tabular}}}}%
    \put(0.17348326,0.02137959){\color[rgb]{0,0,0}\makebox(0,0)[lt]{\lineheight{1.25}\smash{\begin{tabular}[t]{l}$S$\end{tabular}}}}%
    \put(0.34435833,0.14938231){\color[rgb]{0,0,0}\makebox(0,0)[lt]{\lineheight{1.25}\smash{\begin{tabular}[t]{l}$E$\end{tabular}}}}%
    \put(-0.00193771,0.14938231){\color[rgb]{0,0,0}\makebox(0,0)[lt]{\lineheight{1.25}\smash{\begin{tabular}[t]{l}$W$\end{tabular}}}}%
  \end{picture}%
\endgroup%

\end{center}
\caption{ \label{f.disotope} At a crossing, push the overstrand to the front of the surface on the page, and push the understrand to the back of the surface into the page. }
\end{figure}
\end{itemize}

With the understanding that the link diagram $D$ is isotoped to lie on the surface $\Sigma$, first we define four points $x_l, x_r, y_l, y_r$ on the surface at each crossing of $D$ as indicated in the local picture, with $x_l, y_r$ on the over arc and $x_r, y_l$ on the under arc. Then, we define $\alpha'$ curves (deformed from $\alpha$ arcs) that join $x_l$ to $x_r$, $x_r$ to $y_r$, $y_l$ to $y_r$, and lastly $x_l$ to $y_l$ and color them blue as shown in Figure \ref{f.intersectionpts}. 
\begin{figure}[H]
    \centering
    \def \svgwidth{.7\columnwidth}
    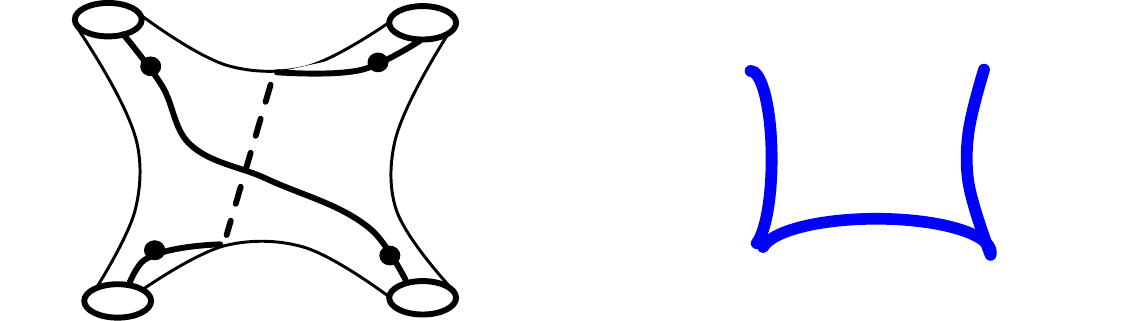
    \caption{Left: four intersection points on $\Sigma$. Right: blue colorings of the $\alpha'$ curves from resolutions. A dashed arc indicates the under arc.}
    \label{f.intersectionpts}
\end{figure}

Let $n$ be the number of crossings in $D$.
We mark additional points that avoid the $\alpha'$ curves and denote them by:
\begin{align}
\begin{cases} 
p_l \text{ and } \bar{p}_l, \\
p_r \text{ and } \bar{p}_r,\\
p'_l \text{ and } \bar{p}'_l,\\
p'_r \text{ and } \bar{p}'_r, \text{ for } i \in \{1,...,n\} \text{ as in Figure } \ref{f.paths}.
\end{cases}
\end{align}

\begin{figure}[H]
\centering
\def \svgwidth{.5\columnwidth}
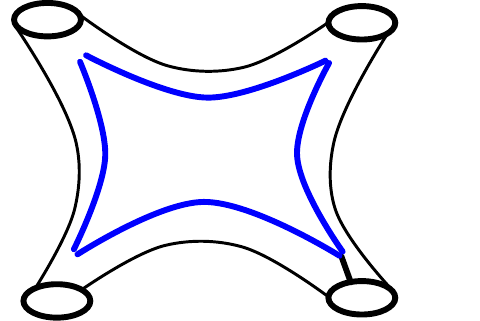
 \caption{\label{f.paths} Additional purple punctures placed with blue arcs.}
\end{figure}
We choose an extra fixed point denoted by $s$ on our surface that does not coincide with the above points. 

\subsection{Configuration space} \label{ss.configspace}
We will be working in the configuration space of $2n$ particles on the surface 
$\Sigma'$ obtained from $\Sigma$ by removing the special puncture $s$ and the sets of punctures and marked points:
$$\{p_l,\bar{p}_l,p_r,\bar{p}_r,p'_l,\bar{p}'_l,p'_r,\bar{p}'_r\} \text{ and }\{ x_l, x_r, y_l, y_r\}$$ that are associated to all the crossings.

Next, for the grading procedure and the definition of the intersection pairing, we define a local system on this configuration space. We designate one puncture for each arc of an $\alpha'$ curve at a crossing. We denote these points as in Figure \ref{f.intersectionpts}:
\begin{equation}
\{ x_l, y_r\}.
\end{equation}

For a computational purpose, we choose two base points near the punctures $x_l, x_r$ and we denote these points as: 
\begin{equation}
\{ b_l, b_r\}.
\end{equation}
The basepoints $\{ b_l, b_r\}$ will be chosen using an alternating link projection $D'$ constructioned from the original diagram $D$. 

Consider the 4-valent graph obtained by collapsing the double points of every crossing in the link diagram to a single point as in Figure \ref{f.collapse}.
 \begin{figure}[H]
 \def \svgwidth{.5\columnwidth}
 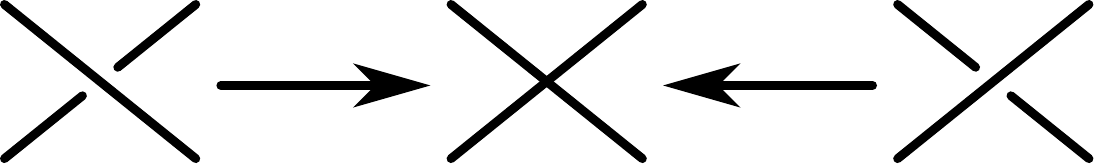

 \caption{\label{f.collapse} Collapse the double points of each crossing of $D$ to a single point.}
 \end{figure}
 By \cite[Theorem 2.7]{ADFK}, the graph has an alternating link projection, say $D'$ through crossings changes of the original diagram $D$. Embed $ D'$ onto the surface $\Sigma$ as in Figure \ref{f.disotope}. Next we  use $D'$ to choose punctures and base points on $\Sigma$.  

\begin{defn}(Choice of base points) \label{d.basepoints}
  For each crossing in $D'$, choose a pair of base points $\{b_l, b_r\}$ on the over arc of $D'$ avoiding $x_l, x_r$, the $\alpha'$ arcs and all the rest of the punctures as shown in Figure \ref{f.pbpts}. 
\begin{figure}[H]
\def \svgwdith{.1\columnwidth}
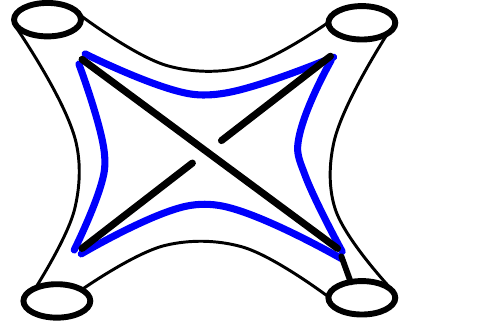
\caption{\label{f.pbpts} Choice of base points at a 4-punctured sphere.}
\end{figure}

\end{defn} 
Recall we color  the $\alpha'$ arcs at a crossing blue. Extensions of $\alpha'$ arcs between crossings will be colored black.

\begin{defn} All of this discussion is for a fixed crossing. We have $n$ crossings in our diagram, denoted $c_1,\ldots ,c_n$ and for crossing $c_i$ we add $i$ as an index to the chosen points, as below:

\begin{align}
c_i: \ \ \  
& \text{Marked points} \ \  \{ x^i_l, x^i_r, y^i_l, y^i_r\}\\
& \text{Punctures} \ \ \{p^i_l,\bar{p}^i_l,p^i_r,\bar{p}^i_r,(p')^i_l,(\bar{p}')^i_l,(p')^i_r,(\bar{p}')^i_r\}\\
& \text{Base points} \ \ \{ b^i_l, b^i_r\}
\end{align}  
\end{defn}
\begin{defn}[Base points and intersection points]
Let us define the set of all base points, associated as above to the set of crossings by:
$$\{ b_1, \ldots ,b_{2n} \}.$$
Also we consider the set of all punctures on the surface and denote it as below:
$$\{ p_1,\ldots,p_{8n},x_1,\ldots,x_{4n} \}.$$
This means that our punctured surface $\Sigma'$ is: \
 $$\Sigma'=\Sigma \setminus \{ p_1,\ldots,p_{8n},x_1,\ldots,x_{4n},s\}.$$
\end{defn}
Next, we define a local system on a specific configuration space on this punctured surface $\Sigma'$. For this, we will use a collection of loops that are based at $\{ b_1, \ldots ,b_{2n} \}$ and go around the punctures $\{ p_1,\ldots ,p_{8n} \}$, as follows.
\begin{defn}[Loops around punctures]
For a fixed $i \in \{1, \ldots ,n\}$ we look at the crossing $c_i$ and  consider a collection of eight curves, each of which goes around a puncture, distributed as below:
\begin{itemize}
    \item $4$ loops based in $b^i_l$ which go around each of the punctures 
$p^i_l$, $\bar{p}^i_l$, $(p')^i_l$ and $(\bar{p}')^i_l$
\item $4$ loops based in $b^i_r$ which go around each of  the punctures $p^i_r$, $(\bar{p})^i_r$, $(p')^i_r$ and $(\bar{p}')^i_r$, as in the last figure on the right of Figure \ref{f.looppaths}.
\end{itemize} 
\begin{figure}[H]
\def \svgwidth{.5\columnwidth}
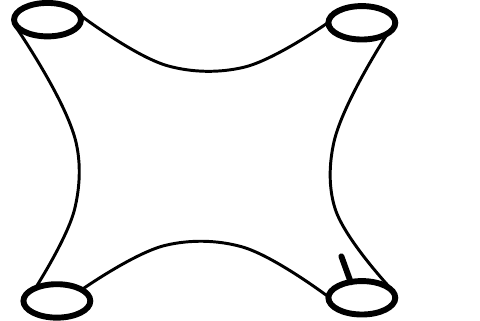
\caption{Loops $\{ \gamma^i_1, \ldots,\gamma^i_8\}$. }
\label{f.looppaths}
\end{figure}
We denote this set of curves by $\{ \gamma^i_1, \ldots,\gamma^i_8\}$.
\end{defn}
\begin{defn}[Set of loops]\label{setlop}
Further, we denote the set of all loops around punctures from above as:
$$\{ \gamma_1,\ldots,\gamma_{8n} \}.$$
For $j \in \{1,\ldots,2n\}$  define $\tau(j)$ to be the index of the base point within $\{b_1, b_2, \ldots, b_{2n}\}$ that belongs to $\gamma_j$:
\begin{equation}
    b_{\tau(j)}\in \gamma_j.
\end{equation}
\end{defn}
\subsection{
Local system} \label{ss.localsystem}

In this part we introduce a local system on our surface with punctures $\Sigma'$.
Let $C_{2n}(\Sigma')$ be the unordered configuration space of $2n$ points on $\Sigma'$. More specifically, this is given by:
$$C_{2n}(\Sigma'):=\{ (z_1,\ldots,z_{2n})\in (\Sigma')^{\times 2n} \mid z_i \neq z_j, \forall \ 1 \leq i < j\leq 2n\} /Sym_{2n} $$
where $Sym_{2n}$ is the symmetric group of order $2n$.

We fix a base point in the configuration space $C_{2n}(\Sigma')$, given by the collection of base points associated to all the crossings, denoted by:
\begin{equation}
{\bf d}:=(b_1,\ldots,b_{2n}) \in C_{2n}(\Sigma').
\end{equation}
\begin{defn}(Local system on $C_{2n}(\Sigma')$) Let us start with the abelianization map by $$ab: \pi_1(C_{2n}(\Sigma')) \rightarrow H_1(C_{2n}(\Sigma')).$$
\end{defn}
We remark that in this homology group we have $8n$ special homology classes, given by loops in the configuration space where we keep all particles fixed except one, which moves on a loop $\gamma_i(t)$. 

More precisely, consider the set of loops from below.
\begin{defn}[Loops in the configuration space] For a index $j\in \{1,\dots, 8n\}$ let us consider the following loops in $C_{2n}(\Sigma')$: 
$$\gamma'_i(t):=\{ \left(b_1,b_2,\dots,b_{\tau(j)-1},\gamma_j(t),b_{\tau(j)+1},\dots,b_{2n}\right) \}, \ t \in [0,1].$$
\end{defn}

\begin{lem}\label{iso}
    The homology classes of these curves $$\{[\gamma'_1],\dots, [\gamma'_{8n}]\}$$ form a linearly independent set in the homology of the configuration spaces $H_1(C_{2n}(\Sigma'))$.
\end{lem}
\begin{proof}
    This property follows from \cite[Proposition 6.32]{darné2022lower} which shows the structure of the abelianization of the braid group of any surface.  
    
    We recall the property that the fundamental group of the configuration space of a surface is precisely the braid group of that surface, which means that we have:
    $$\pi_1(C_{2n}(\Sigma'))\simeq B_{2n}(\Sigma').$$
    This implies that we have:
    \begin{equation}\label{propab}
        H_1(C_{2n}(\Sigma'))\simeq B_{2n}^{\text{ab}}(\Sigma').
            \end{equation}
Now, following Proposition 6.32 from \cite{darné2022lower} we have an isomorphism:

$$g:B_{2n}^{\text{ab}}((\Sigma')) \rightarrow H_1(\Sigma') \times \langle \sigma \rangle$$
    where the order of $\sigma$ is $2$.

  The homology group $H_1(\Sigma')$ is generated by classes of loops which encircle one puncture, for all punctures but one and classes given by meridians around handles. So, we have a linearly independent set given by classes of loops that encircle the punctures 
  $$ \{ p_1,\ldots,p_{8n},x_1,\ldots,x_{4n}\}.$$ (We exclude the loop around the special puncture $s$.) In particular, the classes provided by loops that go around the punctures 
$$ \{ p_1,\ldots,p_{8n}\} $$ are linearly independent in $H_1(\Sigma')$. Following the isomorphism $g$, these classes correspond precisely to $$\{[\gamma'_1],\dots,[\gamma'_{8n}]\}$$
so we conclude that these classes are independent in $B_{2n}^{\text{ab}}((\Sigma'))$. Using \eqref{propab} we obtain that $$\{[\gamma'_1],\dots,[\gamma'_{8n}]\}$$ are linearly independent in $H_1(C_{2n}(\Sigma'))$, which concludes the proof. 

\end{proof}
\begin{defn}[Subgroup in the homology group] Consider the subgroup generated by the classes of these elements:
$$ H:=\langle[\gamma'_1],\dots,[\gamma'_{8n}]\rangle \simeq \mathbb Z^{8n} \subset H_1(C_{2n}(\Sigma')).$$

Also, following Lemma \ref{iso} we have the projection map:
\begin{align*}
\nu: \ & H \ \ \simeq \ \ \mathbb Z^{8n} \\ 
&\langle[\gamma'_{j}] \rangle  \mapsto   \langle x_j \rangle
\end{align*}
defined by the formula:
$\nu([\gamma'_j])=x_j$, $\forall j \in \{1,\dots,8n\}$.
\end{defn}

\begin{defn}[Projection map] 
We have now the subgroup $H$ in the homology group of the configuration space. Consider the associated projection map which we denote as below:
\begin{equation}
    p: H_1(C_{2n}(\Sigma')) \rightarrow H.
\end{equation}

\end{defn}

\begin{defn}(Local system)\label{localsystem}
We consider the local system given by the composition of the three morphisms defined above: 

\begin{equation}
\begin{aligned}
&\Phi: \pi_1(C_{2n}(\Sigma')) \rightarrow \mathbb Z^{8n}\\
&\hspace{28mm} \langle \{x_j\}_{j \in \{1,...,8n\}}\rangle \ \ \\
&\Phi= \nu \circ p \circ ab. \ \ \ \ \ \ \ \ \ \ \ \ \ \ \ \ \ \ \ \ 
\end{aligned}
\end{equation}
\end{defn}

\begin{defn}[Covering space]\label{localsystemc}
Let $\tilde{C}_{2n}(\Sigma')$ be the covering of the configuration space 
associated to the local system $\Phi$.
\end{defn}
\begin{defn}[Base point in the covering space]
\label{d.dlift}

We also fix a base point in this covering space, denoted by $\tilde{{\bf d}}$, which belongs to the fiber over ${\bf d}$ in $\tilde{C}_{2n}(\Sigma')$.
\end{defn}

\section{Twisted homology groups} \label{s.twistedhomology}

We consider certain versions of the relative homology of the covering space $\tilde{C}_{2n}(\Sigma')$. More specifically, we will make use of a subspace of the middle dimensional Borel-Moore homology of $\tilde{C}_{2n}(\Sigma')$ coming from the twisted Borel-Moore homology of the base space. We cite the results we need from the general setup of \cite{CrM}.  

\begin{defn}\label{T2}
Let us consider two homology groups, as below:
\begin{equation*}
\begin{aligned}
& \bullet H^{\text{lf}}_{2n}(\tilde{C}_{2n}(\Sigma'), \Z) \text{ the Borel-Moore homology of $\tilde{C}_{2n}(\Sigma')$}.\\
&\bullet  H_{2n}(\tilde{C}_{2n}(\Sigma'), \Z) \text{ the homology of $\tilde{C}_{2n}(\Sigma')$}.
\end{aligned}
\end{equation*}
\end{defn}

\begin{rmk}[\cite{CrM}]
 There is a subtle difference between the Borel-Moore homology of a covering space and the twisted Borel-Moore homology of the base space. We will work with the homology of the covering space rather than the twisted homology of the base space. 

\end{rmk}
\begin{prop}[\cite{CrM}, Theorem E]\label{P:5}
There are natural injective maps:
\begin{equation}
\begin{aligned}
& \iota: H^{\text{lf}}_{2n}(C_{2n}(\Sigma'); \mathscr L_{\Phi})\rightarrow H^{\text{lf}}_{2n}(\tilde{C}_{2n}(\Sigma'), \Z)\\
& \iota^{\partial}:H_{2n}(C_{2n}(\Sigma'); \mathscr L_{\Phi})\rightarrow H_{2n}(\tilde{C}_{2n}(\Sigma'), \Z).
\end{aligned}
\end{equation}

In the above equations $\mathscr L_{\Phi}$ is the rank $1$ local system associated to $\Phi$ (\cite[Definition 2.7]{CrM}).
\end{prop}

\begin{defn}[Homology groups]\label{Hom}
We denote the above submodules in the homologies of the covering space (which are modules over $\Z[x_1^{\pm 1},...,x_{8n}^{\pm 1}]$):
\begin{enumerate}
 \item[$\bullet$]  $\mathscr H^{lf}_{2n}\subseteq H^{\text{lf}}_{2n}(\tilde{C}_{2n}(\Sigma'), \Z)$ \text{ and } 
 \item[$\bullet$]  $\mathscr H_{2n} \subseteq H_{2n}(\tilde{C}_{2n}(\Sigma'), \Z).$ 
\end{enumerate}
The submodules $\mathscr H^{lf}_{2n}$ and $\mathscr H_{2n}$ are  given by the images of the inclusions $\iota$ and $\iota^{\partial}$ of  the homology with twisted coefficients in the homology of the covering space.

\end{defn}

Using this notation, the homology groups of the covering become modules over the group ring $\Z[x_1^{\pm 1},...,x_{8n}^{\pm 1}]$.

The description of these homology groups has certain subtleties, but for our purpose we will use precise classes given by submanifolds in the configuration space on the punctured surface $\Sigma'$.

In what follows we describe a geometric intersection pairing between these homologies of the covering space  with respect to different parts of its boundary. The existence of this pairing is a consequence of a Poincar\'{e}-Lefschetz duality for twisted homology \cite[Proposition 3.2]{CrM} together with a relative cap product for twisted homology as described in \cite[Lemma 3.3]{CrM}. These two maps induce an intersection pairing between homology groups, as in \cite{CrM}.
\begin{prop}(\cite[Proposition 7.6]{CrM})\label{P:3'''}
There exists a topological intersection pairing:
$$\ll ~,~ \gg: \mathscr H^{lf}_{2n} \otimes \mathscr H_{2n} \rightarrow\Z[x_1^{\pm 1},...,x_{8n}^{\pm 1}].$$
\end{prop}

\subsection{Computing the intersection pairing}
We will need the precise form of this intersection pairing in the next part so we sketch the key components for its formula, which is presented in \cite[Section 7]{CrM}. Recall we are working with coefficients that belong to the group ring $\Z[x_1^{\pm 1},...,x_{8n}^{\pm 1}]$. For this, we introduce the following notation.

\begin{defn}\label{tildelocsyst}
Let $\Phi$ be the morphism induced by the local system $\Phi$, which takes values in the group ring of $\Z^{8n}$ :
\begin{equation}
\Phi: \pi_1(C_{2n}(\Sigma')) \rightarrow \Z[x_1^{\pm 1}, \ldots ,x_{8n}^{\pm 1}].
\end{equation}
\end{defn}
We use this morphism for the computation of the intersection pairing as follows. 
Fix two homology classes $$H_1 \in \mathscr H^{lf}_{2n} \text{ and } H_2 \in \mathscr H_{2n}. $$ Suppose $H_1$ and $H_2$ are given by the lifts of two submanifolds  $X_1,X_2 \subseteq C_{2n}(\Sigma')$, which we denote $\tilde{X}_1, \tilde{X}_2 \subseteq \tilde{C}_{2n}(\Sigma')$. We further assume that $X_1$ and $X_2$ intersect transversely in a finite number of points.

For an intersection point $x \in X_1 \cap X_2$ we construct a loop $l_x \subseteq C_{2n}(\Sigma')$.

\begin{defn} \label{d.associatedloop} {\bf Construction of the associated loop $l_x$}\\
 Suppose we have two paths $\gamma_{X_1}, \gamma_{X_2}$ starting in $\mathbf{d} = (b_1, \ldots, b_{2n})$ and ending on $X_1$, $X_2$ respectively. Also, we assume that:   
$\tilde{\gamma}_{X_1}(1) \in \tilde{X}_1$ and $ \tilde{\gamma}_{X_2}(1) \in \tilde{X}_2$.
Choose also two paths $\delta_{X_1}, \delta_{X_2}:[0,1]\rightarrow C_{2n}(\Sigma')$ such that:
\begin{equation}
\begin{cases}
Im(\delta_{X_1})\subseteq X_1; \delta_{X_1}(0)=\gamma_{X_1}(1);  \delta_{X_1}(1)=x\\
Im(\delta_{X_2})\subseteq X_2; \delta_{X_2}(0)=\gamma_{X_2}(1);  \delta_{X_2}(1)=x.
\end{cases}
\end{equation}
The composition of these paths gives a loop, as follows:
$$l_x=\gamma_{X_2}^{-1} \circ \delta_{X_2}^{-1}\circ \delta_{X_1} \circ \gamma_{X_1}.$$
Denote by $\alpha_x$ the sign of the geometric intersection at the point $x$ between $X_1$ and $X_2$ (in the base configuration space). See Figure \ref{f.loopx} for an example of $l_x$. 
\end{defn}

\begin{prop}[Intersection pairing from intersections in the base space]\label{formint}  
The intersection pairing can be described from the set of loops $l_x$ and the local system as below:
\begin{equation}\label{eq:1}  
\ll H_1,H_2 \gg=  \sum_{x \in X_1 \cap X_2}  \alpha_x \cdot \Phi(l_x) \in \Z[x_1^{\pm 1},...,x_{8n}^{\pm 1}].
\end{equation}
\end{prop}

\subsection{Specialization of coefficients associated to a state} \label{SS:spec} 

In the next few sections, we set up the specializations of coefficients for the intersection pairing $\ll ~,~ \gg$ so that they correctly compute key quantities in the definition of the Jones polynomial through the state sum from the Kauffman bracket, and in the definition of the HOMFLY-PT polynomial through a similar state sum. We conclude Section \ref{s.twistedhomology} with explicit specializations of the intersection pairing for both polynomials in Section \ref{ss.specialization}. 

\

\paragraph{\textbf{Coloring of Kauffman state circles.}}
Fix a link diagram $D$. 
Our first task for constructing the specialization is to define a coloring on the state circles resulting from applying a Kauffman state $D$. 

\begin{defn} \label{d.Kauffmanstate} A \textit{Kauffman state} $\sigma: c(D) \rightarrow \{ \pm \}$ is a function on the set of crossings of $D$ that assigns to a crossing $\vcenter{\hbox{\includegraphics[scale=.15]{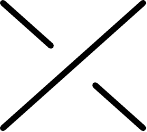}}}$ the $+$-resolution $\vcenter{\hbox{\includegraphics[scale=.2]{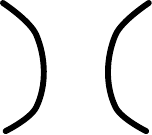}}}$ or the $-$-resolution $\vcenter{\hbox{\includegraphics[scale=.2]{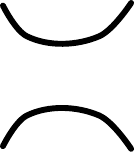}}}$.  
\end{defn}

Let $\sigma$ be a Kauffman state on a link diagram $D$. Replacing each crossing of the diagram with the choice of the $+$- or $-$- resolution by the Kauffman state results in a set of disjoint circles $S_{\sigma}$. A circle in $S_{\sigma}$ is called a \textit{state circle}. We define a coloring on the set of state circles such that every state circle is bicolored with an even number of blue segments and the same number of black segments, and the coloring alternates between blue and black as we traverse each state circle: 

First isotope the link diagram $D$ so that it lies on the surface $\Sigma$ as shown in Figure \ref{f.disotope}. The local replacements from applying a Kauffman state $\sigma$ to the link diagram can then be done on the surface $\Sigma$ by replacing the pair of link strands from $x_l$ to $y_r$ and from $x_r$ to $y_l$ as follows:\\ 

\hspace*{15mm}Replace by\\
\[  \begin{cases} 
&\text{the pair of $\alpha'$ arcs from $x_l$ to $y_l$ and then from $x_r$ to $y_r$,}\\
&  \text{ if $\sigma$ chooses $\vcenter{\hbox{\includegraphics[scale=.2]{pres.pdf}}}$.} \\
&\text{the pair of $\alpha'$ arcs from $x_l$ to $x_r$ and then from $y_l$ to $y_r$,}\\ 
& \text{ if $\sigma$ chooses $\vcenter{\hbox{\includegraphics[scale=.2]{nres.pdf}}}$.} \\ 

\end{cases} \]

At every crossing, color both types of arcs resulting from the choice of the resolution blue. This is shown in Figure \ref{f.statecolorarcs}. 

\begin{figure}[H]
\begin{center}
\def \svgwidth{.5\columnwidth}
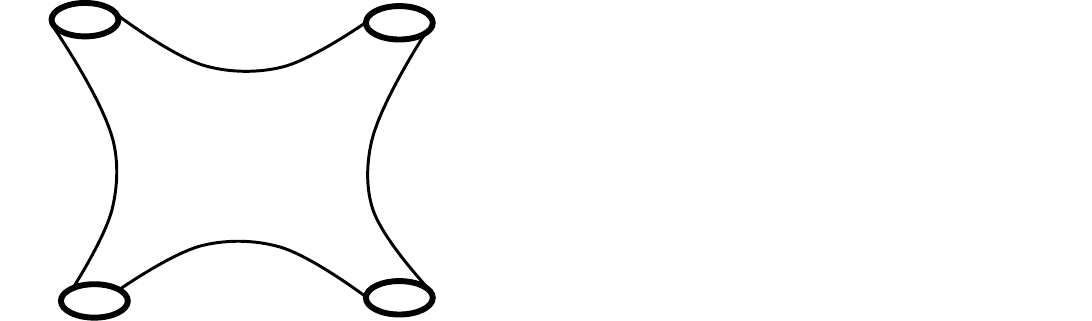
\end{center}
\caption{\label{f.statecolorarcs} Colorings on $\alpha'$ arcs of each choice of the resolution.}
\end{figure}

Keep the rest of the link strands black. Given a Kauffman state $\sigma$, replacing each crossing by the choice of the resolution by the state now results in a set of disjoint bi-colored circles. 
\begin{rmk} \label{r.bicolored}
By construction, for each state $\sigma$ there are twice as many blue-colored arcs as the number of crossings in the link diagram. Also, there are twice as many black-colored arcs as the number of crossings.    
\end{rmk}
This property comes from the fact that between a pair of blue-colored arcs on the same strand there is a black arc, and there is a pair of blue arcs for every crossing. 

\begin{figure}[H]
\begin{center}
\def \svgwidth{.5\columnwidth}
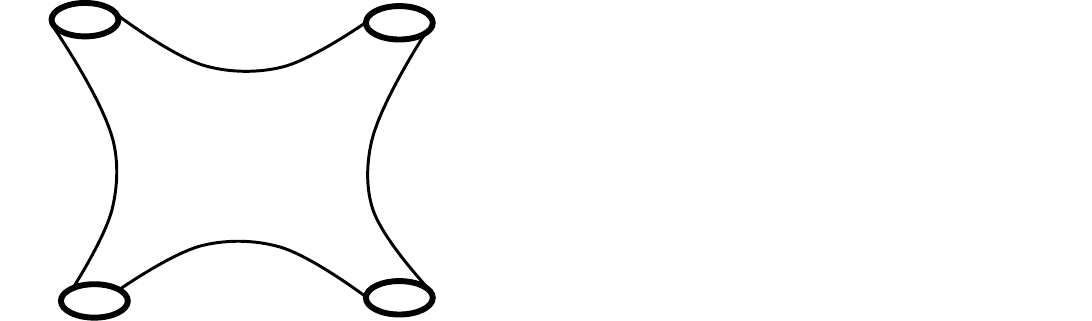
\end{center}
\caption{\label{f.statecolor} Arcs shown with punctures at a crossing.}
\end{figure}

\paragraph{\textbf{Special punctures associated to a state}}
Next we fix a Kauffman state $\sigma$. As we have seen, there is an associated set of disjoint circles on the surface, denoted by $S_{\sigma}$.
Since our surface $\Sigma$ is constructed from the knot diagram $D$ in a plane $\pi$, we will fix a projection of the set of punctures onto $\pi$ and denote this set as: $\{ \tilde{p}_1,\ldots,\tilde{p}_{8n},\tilde{x}_1,\ldots,\tilde{x}_{4n},\tilde{s}\}.$
\begin{figure}[H]
\def \svgwidth{.7\columnwidth}
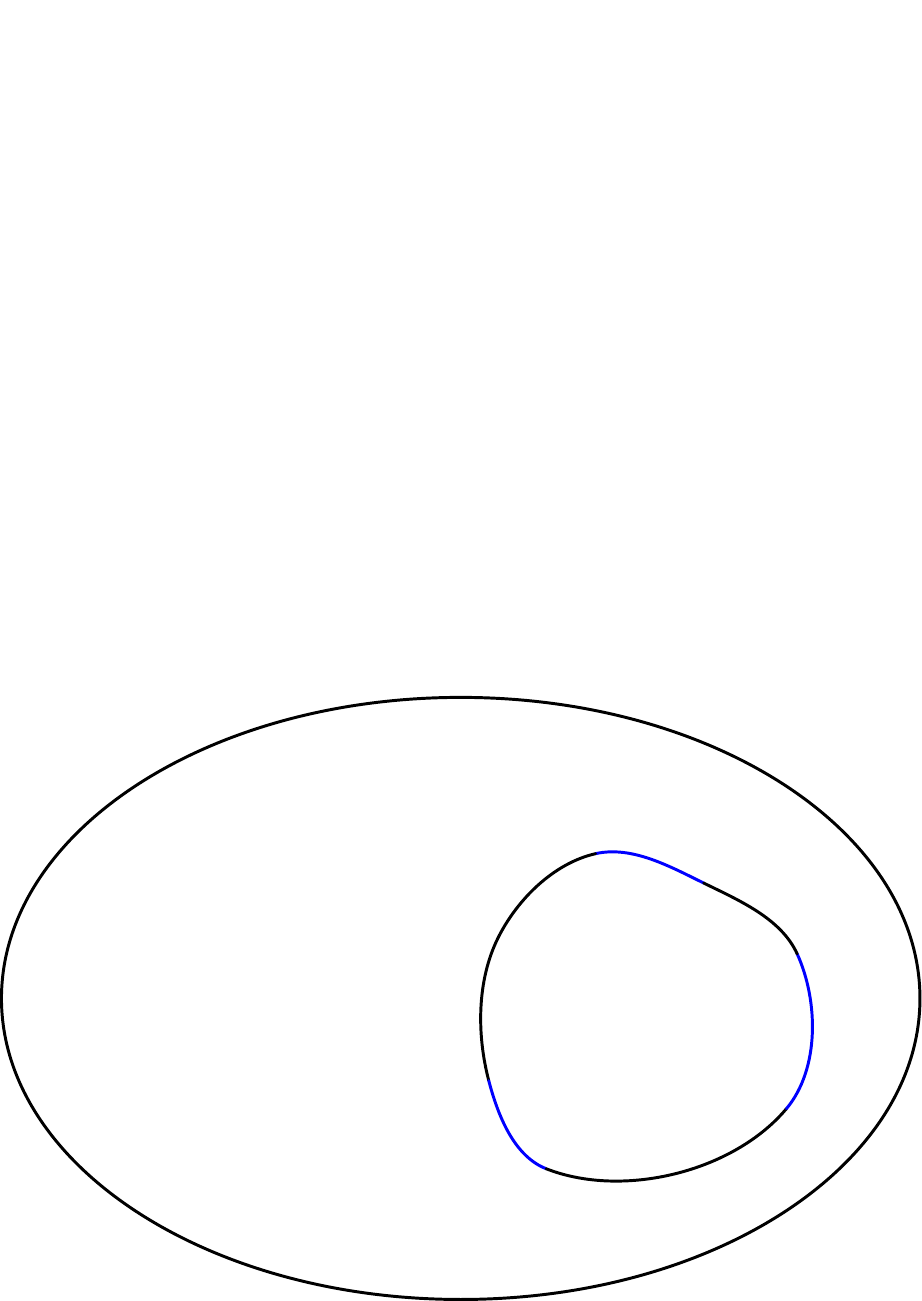
\caption{The punctured surface $\Sigma'$ projected onto the plane of the diagram  and associated images of the punctures.}
\end{figure}
\begin{rem}
    On the projection plane, the state circles may contain other circles.
\end{rem}
Up to isotopy, we can assume that the state circles do not project onto the punctures. Denote by $S'_{\sigma}$ the set of projections of the circles from $S_{\sigma}$. Without loss of generality, we also suppose that the puncture $\tilde{s}$ ``at infinity''  does not belong to the disk bounded by any circle from $S'_{\sigma}$ in $\pi$.

We  define the specialization of coefficients associated to the local system $\Phi$ using the above projections.

\begin{defn}[Disks and punctures associated to state circles] Let $C$ be a state circle from $S_{\sigma}$ and $C'\in S'_{\sigma}$ its projection onto the plane of the knot diagram. We denote by $D^C$ the disk bounded by $C'$ in this plane, and called it the \textit{state disk associated to $C$}.
\end{defn}

\begin{defn}[Punctures associated to state circles] Let $C$ be a state circle and consider its associated state disk $D^C \subseteq \pi$. The disk $D^C$ potentially contains some other disks associated to other state circles. We consider the region in the plane obtained from $D^C$ by removing all these disks, and denote this region by $R^C$.
\end{defn}

\begin{lem} \label{l.associatedpuncture} For any state circle $C$, the associated region $R^C$ contains at least one puncture from the set $\{\tilde{p}_1,\ldots,\tilde{p}_{8n}\}$.
    
\end{lem}
\begin{proof}
The state circle $C$ is constructed from the blue-colored arcs associated to the resolution $\sigma$. Recall there are two punctures associated to each such arc, say $p^j$ and $\bar{p}^j$, see Figure \ref{f.paths}.  Exactly one of these punctures projects inside the disc $D^C$. Since these punctures are chosen in a tubular neighborhood of the arc, they are not contained in any other associated discs that lie inside $D^C$. So, the projection of the puncture lies in the region $R^C$.
\end{proof}
\begin{defn}[$\sigma$-punctures] \label{d.projectionf}
    For a state circle $C_i \in S_{\sigma}$ choose one of the punctures guaranteed by the proof of Lemma \ref{l.associatedpuncture}  that projects inside $R^{C_i}$  and denote it as $P^i$. 

    Also denote the evaluation of the local system $\Phi$ around this puncture by $x_{f(i)}$, where $f(i)\in \{1,...,8n\}$.
    
\end{defn}

Following the procedure described above, each circle in the collection of state circles $S_{\sigma}=\{C_1,...,C_{|\sigma|}\}$ has an associated puncture. We call this collection of ``special punctures'' associated to the state $\sigma$:
$$\{P^1,\ldots,P^{|\sigma|}\} \subseteq \{\tilde{p}_1,\ldots,\tilde{p}_{8n}\}$$  
 the set of \textit{$\sigma$-punctures}.

\begin{defn}[$\sigma$-punctures inside a fixed disk]\label{sigmpunctures}
    For a circle $C_i \in S_{\sigma}$ consider all the $\sigma$-punctures associated to other state circles $C_{i'} \subset C_i$ that belong to the associated disk $D^{C_i}$ that do not include the chosen puncture $P^i$. 
    Denote this set as below:
    \begin{equation}
        N^i := \{N^i_1,\ldots,N^i_{\tau(i)}\} \subset \{P^1,\ldots,P^{|\sigma|}\}.
    \end{equation}

    Also denote the evaluation of the local system $\Phi$ around each of these punctures by $x_{g_i(j)}$, where $j\in \{1,\ldots,\tau(i)\}$ and $g_i(j)\in \{1,\ldots,8n\}$.
\end{defn}

\paragraph{\textbf{Generic Specialization of coefficients}}
We define a specialization of coefficients for our local system, such that the monodromies that are ``seen'' through this specialization are just the ones around the set of $\sigma$-punctures.
Recall that we have the local system introduced in Definition \ref{tildelocsyst}:
\begin{equation}
\Phi: \pi_1(C_{2n}(\Sigma')) \rightarrow \Z[x_1^{\pm 1},\ldots,x_{8n}^{\pm 1}].
\end{equation}
Our aim is to introduce a change of coefficients that will lead to the desired intersection of homology classes, once specialized. We do this in two steps. First for each Kauffman state $\sigma$ we project onto $\Z^{|\sigma|}$, which means that we count only the monodromies around $\sigma$-punctures. After that, we evaluate each component of $\Z^{|\sigma|}$ towards polynomials in one or two variables, associated to the case of the Jones polyomial or HOMFLY-PT polynomial, respectively. Let us make this precise.
\\

\noindent {\bf Step 1-Mark the $\sigma$-punctures}

\

\begin{defn}[Projection map which keeps just the $\sigma$-punctures]
 Let us denote the following projection map   
\begin{equation}    
\begin{aligned}
&p_{\sigma}: \ \ \ \ \ \ \Z[x_1^{\pm 1}, \ldots ,x_{8n}^{\pm 1}] \ \ \ \ \ \rightarrow \ \ \ \ \ \Z[x_{f(1)}^{\pm 1},\ldots ,x_{f(|\sigma|)}^{\pm 1}], 
\end{aligned}
\end{equation}
with $f(i)\in \{1, \ldots, 8n\}$ as defined by Definition \ref{d.projectionf}. 
\end{defn}

\

\noindent{\bf Step 2- Evaluate monodromies around $\sigma$-punctures}

\

The second step prescribes the evaluation of monodromies around each $\sigma$-puncture such that a monodromy requirement is satisfied.
The main idea is that when we count the contributions of the above evaluations around punctures belonging to any state disk we always obtain  a fixed monomial. 

This condition that we get a fixed monomial for each state circle will play a key role that ensures our intersection of homology classes gives the Kauffman bracket when evaluated on the set of state circles. We proceed as follows.

\begin{defn}[Requirement for specialization of coefficients]\label{reqspec}
   Fix $Q$ to be a polynomial or an indeterminate.
 We search for a change of coefficients:
 \begin{equation}
      \begin{aligned}
&\alpha^{\sigma}_Q: \ \ \ \ \ \  \Z[x_1^{\pm 1},\ldots,x_{8n}^{\pm 1}] \ \ \ \ \ \rightarrow \ \ \ \ \ \Z[Q^{\pm1}],
\end{aligned}
 \end{equation}

which is a ring homomorphism that satisfies the following {\em Monodromy Requirement}.

\

\noindent {\bf Monodromy Requirement} Consider a state circle $C_i$ and the 

monodromies around punctures associated to the $\sigma$-punctures inside $D^{C_i}$. The requirement is that the product of all these monodromies gets evaluated through $\alpha^{\sigma}_Q$ to our fixed indeterminate $Q$.

\end{defn}

\

\begin{lem}
    There exists a change of coefficients $\alpha^{\sigma}_Q$, which we construct explicitly and which satisfies the Monodromy Requirement. 
\end{lem}
\begin{proof}

Since the Monodromy Requirement contains obstructions for evaluation of monodromies just around $\sigma$-punctures, it means that such $\alpha^{\sigma}_Q$ should factor through the projection $p_{\sigma}$, so it is enough to define 
 \begin{equation}
      \begin{aligned}
&\alpha^{\sigma}_Q: \ \ \ \ \ \  \Z[x_{f(1)}^{\pm 1},\ldots,x_{f(|\sigma|)}^{\pm 1}] \ \ \ \ \ \rightarrow \ \ \ \ \ \Z[Q^{\pm1}].
\end{aligned}
 \end{equation}

We will proceed by an inductive argument. First, we look at state circles whose associated disks do not contain any other state disk. If $C'_i$ is such a state circle, then $D^{C_i}$ contains just a $\sigma$-puncture, namely $P^i$. This is evaluated to $x_{f(i)}$ by the local system. In this case, we define:
\begin{equation}
    \alpha^{\sigma}_Q(x_{f(i)})=Q.
    \end{equation}

This means that the unique puncture associated to such a disk has the monodromy specialized to $Q$.

We proceed to define $\alpha^{\sigma}_Q$ on every $x_{f(i)}$ for which the corresponding special puncture $P^i$ is contained in a state circle whose associated disk does not contain any other disk. 

Inductively, consider a state circle $C'_i$ which contains other state circles,  and suppose that we have defined the function $\alpha^{\sigma}_Q$ around all punctures associated to the region $$D^{C_i} \setminus R^{C_i},$$ which are:
 \begin{equation}
        N^i = \{N^i_1, \ldots,N^i_{\tau(i)}\} \subset \{P^1, \ldots,P^{|\sigma|}\}, 
    \end{equation}
 and that $a^{\sigma}_Q$ satisfies the Monodromy requirement for all the state circles $C'\subset D^{C_i}$. 
Recall that the local system $\Phi$ evaluates monodromies around these punctures in $N^i$, by: $$x_{g_i(j)},$$ where $j\in \{1,\ldots,\tau(i)\}$ and $g_i(j)\in \{1,\ldots, 8n\}$, see Definition \ref{sigmpunctures}.
This means that we have already defined :
$$\alpha^{\sigma}_Q(x_{g_i(j)}), \text { for } j\in \{1,\ldots,\tau(i)\}.$$
\begin{figure}[H]
\def \svgwidth{.5\columnwidth}
\begingroup%
  \makeatletter%
  \providecommand\color[2][]{%
    \errmessage{(Inkscape) Color is used for the text in Inkscape, but the package 'color.sty' is not loaded}%
    \renewcommand\color[2][]{}%
  }%
  \providecommand\transparent[1]{%
    \errmessage{(Inkscape) Transparency is used (non-zero) for the text in Inkscape, but the package 'transparent.sty' is not loaded}%
    \renewcommand\transparent[1]{}%
  }%
  \providecommand\rotatebox[2]{#2}%
  \newcommand*\fsize{\dimexpr\f@size pt\relax}%
  \newcommand*\lineheight[1]{\fontsize{\fsize}{#1\fsize}\selectfont}%
  \ifx\svgwidth\undefined%
    \setlength{\unitlength}{173.09411958bp}%
    \ifx\svgscale\undefined%
      \relax%
    \else%
      \setlength{\unitlength}{\unitlength * \real{\svgscale}}%
    \fi%
  \else%
    \setlength{\unitlength}{\svgwidth}%
  \fi%
  \global\let\svgwidth\undefined%
  \global\let\svgscale\undefined%
  \makeatother%
  \begin{picture}(1,0.68651878)%
    \lineheight{1}%
    \setlength\tabcolsep{0pt}%
    \put(0,0){\includegraphics[width=\unitlength,page=1]{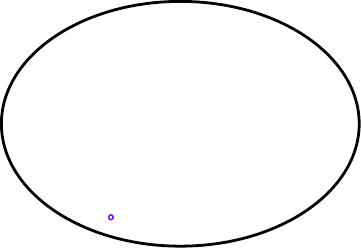}}%
    \put(0.33091768,0.07050596){\color[rgb]{0,0,0}\makebox(0,0)[lt]{\lineheight{1.25}\smash{\begin{tabular}[t]{l}$\tilde{s}$\end{tabular}}}}%
    \put(0,0){\includegraphics[width=\unitlength,page=2]{circleevaluation.pdf}}%
    \put(0.71594793,0.36388073){\color[rgb]{0,0,0}\makebox(0,0)[lt]{\lineheight{1.25}\smash{\begin{tabular}[t]{l}$D^{C^i}$\end{tabular}}}}%
    \put(0.55027132,0.27453523){\color[rgb]{0,0,0}\makebox(0,0)[lt]{\lineheight{1.25}\smash{\begin{tabular}[t]{l}$P^i$\end{tabular}}}}%
    \put(0.48690039,0.52078833){\color[rgb]{0,0,0}\makebox(0,0)[lt]{\lineheight{1.25}\smash{\begin{tabular}[t]{l}$C_i$\end{tabular}}}}%
    \put(0,0){\includegraphics[width=\unitlength,page=3]{circleevaluation.pdf}}%
  \end{picture}%
\endgroup%

\caption{A fixed state circle $C'_i$ and its associated disk $D^{C_i}$ with chosen puncture $P^i$ and other punctures associated to other disks contained in $D^{C_i}$.}
\end{figure}

For the inductive case, we need only to figure out how $\alpha^{\sigma}_Q$ evaluates around the puncture $P^i$. 
    We set the change of coefficients associated to the $\sigma$-puncture $P^{i}$ as below:
    \begin{equation}
    \alpha^{\sigma}_Q(x_{f(i)})=Q \cdot \prod_{j=1}^{\tau(i)}\alpha_Q(x_{g_i(j)})^{-1}.
    \end{equation}
With this definition, we see that the evaluation of $\alpha^{\sigma}_Q$ on variables on all $\sigma$-punctures inside the state disc $D^{C_i}$ including $P^i$ is precisely $Q$, which satisfies the Monodromy Requirement for the state circle $C_i$. Following this procedure to define $\alpha^{\sigma}_Q$ for all the state circles in $S_{\sigma}$, we obtain the change-of-coefficients map $\alpha^{\sigma}_Q$ that satisfies the Monodromy Requirement for all state circles. Extending it by linearity to $\mathbb{
Z}[x_1^{\pm 1}, \ldots, x^{\pm 1}_{8n}]$, we obtain the desired function. 
\end{proof}
\begin{defn}[Generic change of coefficients]\label{specQ}
Fix such a change of coefficients:
 \begin{equation}
      \begin{aligned}
&\alpha^{\sigma}_Q: \ \ \ \ \ \  \Z[x_1^{\pm 1},\ldots,x_{8n}^{\pm 1}] \ \ \ \ \ \rightarrow \ \ \ \ \ \Z[Q^{\pm1}]
\end{aligned}
 \end{equation}
satisfying the conditions and the {\em Monodromy Requirement} from Definition \ref{reqspec}.
\end{defn}

We recall that our homology groups of the covering space are modules over the group ring $\Z[x_1^{\pm 1},...,x_{8n}^{\pm 1}]$ and we have the associated intersection pairing
$$\ll ~,~ \gg: \mathscr H^{lf}_{2n} \otimes \mathscr H_{2n} \rightarrow\Z[x_1^{\pm 1},...,x_{8n}^{\pm 1}] \qquad \text{(Proposition \ref{P:3'''}).}$$
\begin{defn}[$Q$-specialization of coefficients] \label{d.qspecial}
We consider the specializations of these modules using the specialization of coefficients $\alpha^{\sigma}_Q$:
$$\mathscr H^{lf}_{2n}\mid_{\alpha^{\sigma}_Q} \text{ and   }  \ \ \mathscr H_{2n}\mid_{\alpha^{\sigma}_Q}.$$ Then, we have an associated specialized intersection pairing which we denote as below:
\begin{equation}
\ll ~,~ \gg_{\alpha^{\sigma}_Q}: \mathscr H^{lf}_{2n}\mid_{\alpha^{\sigma}_Q} \otimes \mathscr H_{2n}\mid_{\alpha^{\sigma}_Q} \rightarrow \Z[Q^{\pm 1}].
\end{equation}
\end{defn}

\subsection{Specializations for the Jones and HOMFLY-PT polynomials}
\label{ss.specialization}

We have seen that once we fix an indeterminate $Q$, we have a recipe to define a specialization of coefficients that satisfies the Monodromy Requirement, see Definition \ref{d.qspecial}. Now we turn our attention to the invariants that we want to achieve, namely the Jones polynomial and the HOMFLY-PT polynomial, respectively, through state sums. 
In this section, we introduce the specialization of coefficients that we will use for each of these cases. They will both be constructed as $\alpha_Q$ but the only difference is that we will use different indeterminates for the Jones polynomial case and for the HOMFLY-PT polynomial case, as follows.

\begin{defn}[Fixing the specializations]
    Let us define two indeterminates as below:
    \begin{itemize}
        \item $Q_J:=1-q-q^{-1}$ for the Jones polynomial 
        \item $Q_H:=1-\frac{a-a^{-1}}{z}$ associated to the HOMFLY-PT polynomial.
\end{itemize}
\end{defn}
After this, following the procedure described in Section \ref{SS:spec}, we obtain the specializations of coefficients presented below. 
We remark that following the setting used from the previous construction, we have as image for the specialization map the ring of Laurent polynomials in $Q$. In this situation, it means that in the above two cases, which will be associated to $\alpha^{\sigma}_{J}$ and $\alpha^{\sigma}_{H}$, we would have polynomials in the variables $1-q-q^{-1}$ (and $1-\frac{a-a^{-1}}{z}$ respectively) and their inverses. 

However, we will enlarge slightly the ring of coefficients and look at the following specializations. 
\begin{defn}[Specialization of coefficients for the Jones polynomial] \label{specJ}
We have the change of coefficients:
 \begin{equation}
      \begin{aligned}
&\alpha^{\sigma}_J: \Z[x_1^{\pm 1},...,x_{8n}^{\pm 1}] \rightarrow \Z[q^{\pm1}](1-q-q^{-1})
\end{aligned}
 \end{equation}
and the associated specialized intersection pairing:
\begin{equation}
\ll ~,~ \gg_{\alpha^{\sigma}_J}: \mathscr H^{lf}_{2n}\mid_{\alpha^{\sigma}_J} \otimes \mathscr H_{2n}\mid_{\alpha^{\sigma}_J} \rightarrow \Z[q^{\pm 1}](1-q-q^{-1}).
\end{equation}
\end{defn}

\begin{defn}[Specialization of coefficients for the HOMFLY-PT polynomial] \label{specH}
For the case of the HOMFLY-PT polynomial we will use the change of coefficients:
 \begin{equation}
      \begin{aligned}
&\alpha^{\sigma}_H:  \Z[x_1^{\pm 1},...,x_{8n}^{\pm 1}]  \rightarrow \Z[a^{\pm1},z^{\pm1}](1-\frac{a-a^{-1}}{z})
\end{aligned}
 \end{equation}
and the associated specialized intersection pairing:
\begin{equation}
\ll ~,~ \gg_{\alpha^{\sigma}_H}: \mathscr H^{lf}_{2n}\mid_{\alpha^{\sigma}_H} \otimes \mathscr H_{2n}\mid_{\alpha^{\sigma}_H} \rightarrow\Z[a^{\pm1},z^{\pm1}](1-\frac{a-a^{-1}}{z}).
\end{equation}
\end{defn}
The specialization $\alpha^{\sigma}_H$ will be used in Section \ref{S:H} for the topological model of the HOMFLY-PT polynomial.

\section{Recovering the Jones polynomial} \label{s.RecoverJones}

In this part we describe how the intersection pairing with the  specialization of coefficients defined in the previous section recovers the Jones polynomial.  

\subsection{The Kauffman bracket and the Jones polynomial} \label{ss.Joneskb}
Let $L$ be a link in $S^3$ with diagram $D$. For a plane $S^2$ of projection for a diagram of $L$, the Kauffman bracket skein module \cite{Przytycki} $K(S^2)$ is the complex vector space generated by unoriented link diagrams, considered up to planar isotopy and modulo the Kauffman bracket skein relations\footnote{We use the variable substitution $q = -A^{-2}$.}: 

\begin{itemize}
\item the simple unknot $ \vcenter{\hbox{\includegraphics[scale=.125]{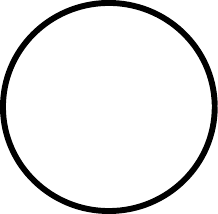}}}= (q + q^{-1})$; 
\item let $D$ be an unoriented link diagram and let $\vcenter{\hbox{\includegraphics[scale=.15]{crossing.pdf}}}$ represent a crossing in $D$, then 
\[ D = D_{\vcenter{\hbox{\includegraphics[scale=.1]{crossing.pdf}}}} =  \sqrt{-q^{-1}}D_+  + \sqrt{-q} D_{-}. \]
 
Here $D_+$ and $D_-$ are link diagrams that are identical to $D=D_{\vcenter{\hbox{\includegraphics[scale=.1]{crossing.pdf}}}}$ everywhere outside of a small neighborhood of the crossing $\vcenter{\hbox{\includegraphics[scale=.15]{crossing.pdf}}}$, and they differ from $D$ in the small neighborhood as indicated in Figure \ref{f.kb}. 
\begin{figure}[H]
\centering
\def \svgwidth{.2\columnwidth}
\begingroup%
  \makeatletter%
  \providecommand\color[2][]{%
    \errmessage{(Inkscape) Color is used for the text in Inkscape, but the package 'color.sty' is not loaded}%
    \renewcommand\color[2][]{}%
  }%
  \providecommand\transparent[1]{%
    \errmessage{(Inkscape) Transparency is used (non-zero) for the text in Inkscape, but the package 'transparent.sty' is not loaded}%
    \renewcommand\transparent[1]{}%
  }%
  \providecommand\rotatebox[2]{#2}%
  \newcommand*\fsize{\dimexpr\f@size pt\relax}%
  \newcommand*\lineheight[1]{\fontsize{\fsize}{#1\fsize}\selectfont}%
  \ifx\svgwidth\undefined%
    \setlength{\unitlength}{331.96241484bp}%
    \ifx\svgscale\undefined%
      \relax%
    \else%
      \setlength{\unitlength}{\unitlength * \real{\svgscale}}%
    \fi%
  \else%
    \setlength{\unitlength}{\svgwidth}%
  \fi%
  \global\let\svgwidth\undefined%
  \global\let\svgscale\undefined%
  \makeatother%
  \begin{picture}(1,0.32719076)%
    \lineheight{1}%
    \setlength\tabcolsep{0pt}%
    \put(0,0){\includegraphics[width=\unitlength,page=1]{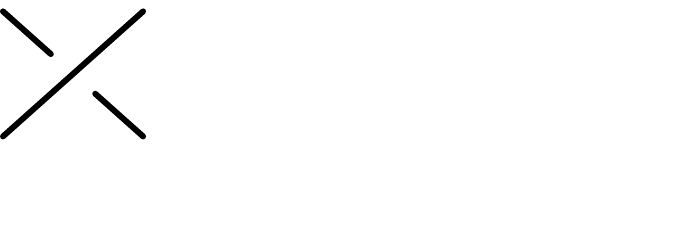}}%
    \put(0.0448068,0.00691907){\color[rgb]{0,0,0}\makebox(0,0)[lt]{\lineheight{1.25}\smash{\begin{tabular}[t]{l}$D_{\vcenter{\hbox{\includegraphics[scale=.1]{crossing.pdf}}}}$\end{tabular}}}}%
    \put(0,0){\includegraphics[width=\unitlength,page=2]{Kauffman_bracket.pdf}}%
    \put(0.42293414,0.0107798){\color[rgb]{0,0,0}\makebox(0,0)[lt]{\lineheight{1.25}\smash{\begin{tabular}[t]{l}$D_+$\end{tabular}}}}%
    \put(0.83864406,0.0107798){\color[rgb]{0,0,0}\makebox(0,0)[lt]{\lineheight{1.25}\smash{\begin{tabular}[t]{l}$D_-$\end{tabular}}}}%
  \end{picture}%
\endgroup%

\caption{\label{f.kb} The resolutions at each crossing.}
\end{figure}
\end{itemize}
\begin{rem}
We remark that we are not using the standard normalization for the skein relation. 
\end{rem}

The \textit{Kauffman bracket} $\langle D \rangle$ is the Laurent polynomial in $q$ multiplying the empty diagram obtained after removing every crossing of $D$ and closed circle components using the Kauffman bracket skein relations. The writhe $w(D)$ of an oriented link diagram $D$ is the sum of the signs of the positive $\vcenter{\hbox{\includegraphics[scale=.15]{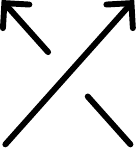}}}$ and negative $\vcenter{\hbox{\includegraphics[scale=.15]{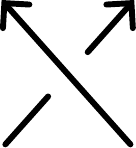}}}$ crossings of $D$.  We give a definition of the (unreduced) Jones polynomial \cite{Lickorish} using the Kauffman bracket. 
\begin{defn} \label{d.jones}
    Let $D$ be the diagram of an oriented link $L$ with orientation forgotten. The Jones polynomial $J_L(q)$ of the link $L$ is given by 
    \begin{equation} J_L(q) =  (-1)^{-3w(D)}(-q)^{\frac{3w(D)}{2}} \langle D \rangle.  
    \end{equation}
    \end{defn}

Given a Kauffman state $\sigma$, let $n_+(\sigma)$ be the number of $+$-resolutions chosen by $\sigma$ and $n_-(\sigma)$ be the number of $-$-resolutions chosen by $\sigma$. Define $\sgn(\sigma):= n_+(\sigma) - n_-(\sigma)$. Recall that a Kauffman state chooses the $+$- or the $-$-resolution at each crossing of a link diagram $D$ (Definition \ref{d.Kauffmanstate}). Applying a Kauffman state $\sigma$ to a link diagram $D$ means that we replace every crossing of $D$ by the $\alpha'$ arcs of the resolution chosen by $\sigma$, that results in a disjoint collection of circles $S_{\sigma}$. Let $|\sigma|$ be the number of circles in the collection. The Kauffman bracket $\langle D\rangle$ can be defined in terms of a \emph{state sum}: 
\begin{equation} \langle D \rangle = \sum_{\sigma \text{ a Kauffman state on $D$}} (-q)^{\frac{-\sgn(\sigma)}{2}}(q+ q^{-1})^{|\sigma|} . \end{equation}
Therefore, 
\begin{equation} \label{e.bracketJones} J_L(q) =  (-1)^{-3w(D)} \sum_{\sigma \text{ a Kauffman state on $D$}} (-q)^{\frac{-\sgn(\sigma)+3w(D)}{2}}(q+ q^{-1})^{|\sigma|}. 
    \end{equation}
We remark that the writhe $w(D)$ already has a geometric meaning as the self-intersection number of the link diagram.

\begin{defn}[Collections of arcs and ovals] \label{d.FandL}
Let us suppose that we have a fixed Kauffman state $\sigma$. Recall the bi-coloring on the state circles as defined in Section \ref{ss.intersectionpts} and illustrated in Figure \ref{f.statecolor}. 
We recall that we have constructed in Definition \ref{d.basepoints} our collection of base points, which comes from a choice of an alternating diagram $D'$ from $D$ and the condition that at every crossing we choose the base points diagonally on the over strand of $D'$.

For our topological model, we keep the set of blue arcs, but we replace the black arcs by green ovals, which are constructed in a tubular neighborhood of these arcs, as in Figure \ref{f.constarsovals}.

We construct the ovals such that if they are near a base point then they pass through that base point, as presented in Figure \ref{f.constarsovals}. We recall that locally, at each crossing, our base points are chosen near the punctures. Since the ovals are obtained from the $\alpha'$ arcs also by a deformation, we can choose them such that they pass through the base points. For the collection of $\alpha'$ arcs, we notice that they also contain the set of base points after a small isotopy. 
\begin{figure}[H]
\def\svgwidth{.7\columnwidth}
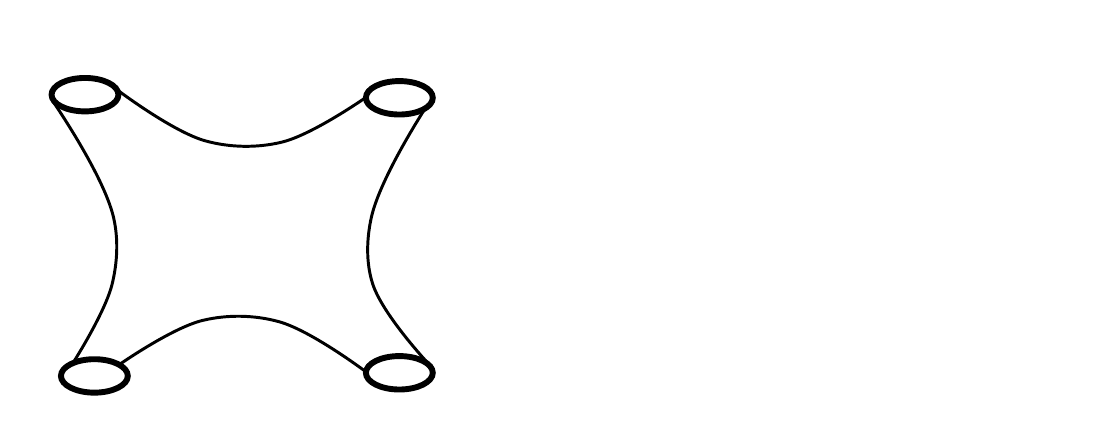
\caption{\label{f.constarsovals} Construction of arcs and ovals.}
\end{figure}

With this data, we denote:
\begin{enumerate}
\item[$\bullet$] $F(\sigma)$ to be the set of blue arcs associated to the state, and 
\item[$\bullet$] $L(\sigma)$ to be the set of green ovals constructed from the black arcs associated to our state.
\end{enumerate}
\end{defn}

With respect to our choice of  base points, we show the following:

\begin{lem}[Base points for our arcs and ovals]\label{baseptscol} Each blue arc from the collection $F(\sigma)$ contains exactly one base point and each green oval from  $L(\sigma)$ contains exactly one base point from the set $$\{b_1,...,b_{2n}\}.$$ 
\end{lem}
\begin{proof}
    First of all, at each crossing, we have chosen the two base points diagonally on the over strand at each crossing of the alternating projection $D'$ from $D$. Thus, each blue $\alpha'$ arc from the choice of a $+$-or $-$-resolution contains exactly one point from the set $\{b_1,...,b_{2n}\}$.  So each arc from the collection $F(\sigma)$ contains exactly one base point.

    For the ovals, we recall that they come from the link strands between crossings of $D$, that are colored black. Since in Definition \ref{d.basepoints} we took care that when we travel the link ignoring the type of crossing we alternate with the choice of base points (utilizing the property that $D'$ is alternating, so the over strand alternates), this means that each black arc is adjacent to exactly one of our chosen base points. Correspondingly, each oval (that is given by a tubular neighborhood of a black arc) passes through exactly one of the base points $$\{b_1,...,b_{2n}\}.$$
\end{proof}

\begin{figure}[H]
\def\svgwidth{.7\columnwidth}
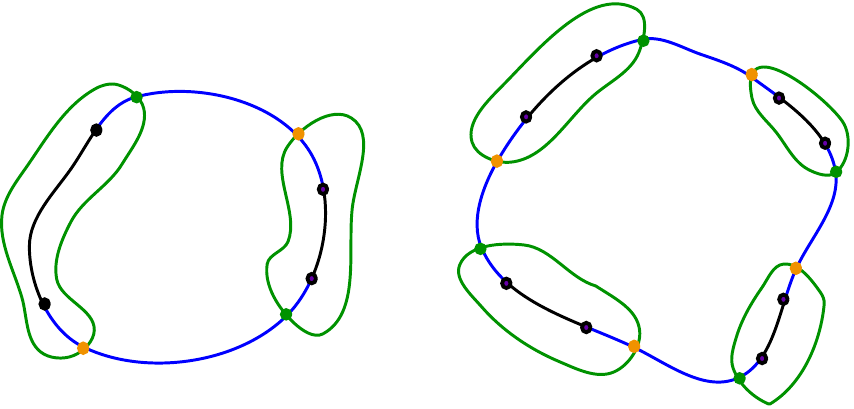
\caption{\label{f.arcsovals} Green ovals replacing black arcs.}
\end{figure} 

\begin{defn}[Orientation] We  fix an orientation for our arcs and ovals. First, we orient all the ovals such that their orientations are consistent with the counter-clockwise orientation. For the arcs, we orient them such that each pair that comes from a resolution has the orientation as in Figure \ref{f.arcsovalso}.    

\begin{figure}[H]
\def\svgwidth{.7\columnwidth}
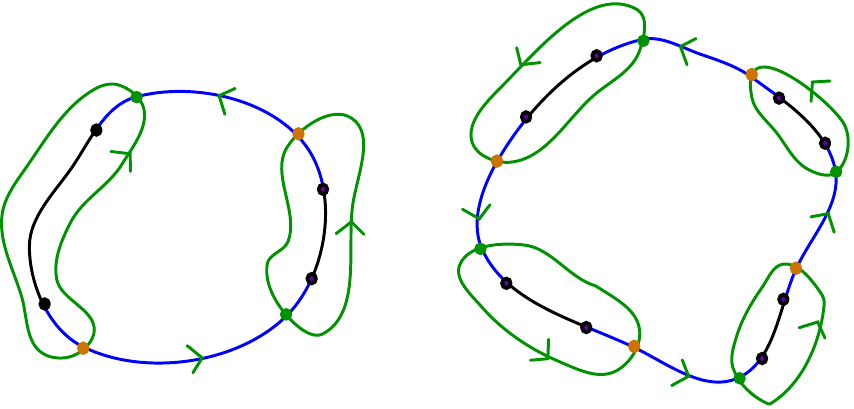
\caption{\label{f.arcsovalso} Orientations on $F(\sigma), L(\sigma)$.}
\end{figure}

Now we fix an orientation for our surface $\Sigma'$ such that the base points from Figure \ref{f.arcsovalso} have all positive orientation.
\end{defn}

\subsection{Homology classes associated to a state}
\label{ss.homologyfromstate}

The next part of our construction is devoted to the definition of two homology classes $\mathscr{F}(\sigma) \in \mathscr H^{lf}_{2n}, \mathscr{L}(\sigma) \in \mathscr H_{2n}$ coming from the two collections of arcs and ovals $F(\sigma), L(\sigma)$ associated to a fixed Kauffman state $\sigma$.  
By Lemma \ref{baseptscol}, to $2n$ base points there are exactly $2n$ arcs in $F(\sigma)$ and $2n$ arcs in $L(\sigma)$.

\noindent {\bf General method for obtaining homology classes.}

One can obtain homology classes from the following data:
\begin{itemize}
\item 

A {\em geometric support}, given by a {\em collection of arcs or ovals in the punctured disk}. Then, the image of the $2n$ product of these arcs or ovals in the configuration space $C_{2n}(\Sigma')$ leads to a submanifold (which has half of the dimension of the configuration space, so dimension $2n$). 
\item A collection of {\em paths to the base point}, which start in the base points $\{b_1,...,b_{2n}\}$ from the punctured surface  and end on these curves. The set of all these paths leads to a path in the configuration space, from the base point $\bf d$ to the submanifold. 
\end{itemize}

Then, the procedure is to lift the path to a path in the covering of the configuration space, that starts from $\tilde{\bf{d}}$ (which is our fixed lift of the base point $\bf d$ by Definition \ref{d.dlift}). After that, we lift the submanifold through the end point of this path. The detailed construction of such homology classes is presented in \cite{Anghel-cjp}.

\begin{rmk}
    
    In our situation, following Lemma \ref{baseptscol}, the collection of arcs $F(\sigma)$ and the collection of ovals $L(\sigma)$ already contain the set of base points $$\{b_1,...,b_{2n}\}.$$ This means that each collection of curves leads to a submanifold in the configuration space $C_{2n}(\Sigma')$ that passes through the base point:
$${\bf d}=\{b_1,...,b_{2n} \}.$$
So, in this case we lift directly these two submanifolds in the covering towards two submanifolds that pass through the base point $\tilde{\bf{d}}$. 
\end{rmk}
In the next part, we use the homology classes coming from these submanifolds, as below.

\begin{defn}[Homology classes for the Jones polynomial]\label{classesJ}
    
With this procedure the collection of curves $F(\sigma)$ leads to a homology class in the covering, which we denote as $$\mathscr F(\sigma)\in \mathscr H^{lf}_{2n}.$$

Dually the collection of ovals $L(\sigma)$ leads to a homology class in the covering, which we denote as $$\mathscr L(\sigma)\in \mathscr H_{2n} \ \ (\text{Figure} \ \ref{f.arcsovalso}).$$
\end{defn}

\subsection{Intersection of classes recovers the evaluation of state circles}

\

In this part we  put all the ingredients together and show that through the topological tools that we introduced above we can obtain the evaluation of the Kauffman bracket on a set of state circles.

\

\paragraph{\textbf{Intersection for generic parameter Q}}

We first show that for a given Kauffman state $\sigma$, the intersection between the associated homology classes recovers a polynomial raised to the power given by the number of circles $|\sigma|$, as below.  

\begin{lem}[Recovering the Kauffmann bracket in a general form]\label{intgeneral}
We work in the context where we have an indeterminate $Q$ and the associated specialization $\alpha^{\sigma}_Q$.  Fix a state $\sigma$. Then, the intersection between the classes $\mathscr 
F(\sigma)$ and $\mathscr 
L(\sigma)$ specialized through $\alpha^{\sigma}_Q$ gives a fixed polynomial in $Q$ to the number of circles, as below:
\begin{equation}
    \ll \mathscr F(\sigma), \mathscr L(\sigma) \gg_{\alpha^{\sigma}_Q}=(1-Q)^{|\sigma|}.
\end{equation}    
\end{lem}
\begin{proof}
We will use the formula for computing the intersection. 

\

\noindent{\bf Step 1- Reducing the problem to the case of a single state circle}

\

Consider the geometric supports of the homology classes $\mathscr F(\sigma)$ and $\mathscr L(\sigma)$: they are given by  lifting the collection of arcs $F(\sigma)$ and the collection of ovals $ L(\sigma)$ to the configuration space, as described in Section \ref{ss.homologyfromstate}. We remark that these two collections of arcs and ovals on the surface contain the set of state circles associated to $\sigma$. The intersection between $\mathscr F(\sigma)$ and $\mathscr L(\sigma)$ is parameterized by the set of intersection points between their geometric supports, and then graded by the local system. 

Consider the set of intersection points. Such a point is encoded by a $2n$-tuple of points $z=(z_1,...,z_{2n})$ on the surface $\Sigma'$ such that:
\begin{enumerate}
    \item each arc from $F(\sigma)$ contains exactly one component from the set $\{z_1,...,z_{2n}\}$;
    \item each oval from $L(\sigma)$ contains exactly one component from the set $\{z_1,...,z_{2n}\}$.
\end{enumerate}

Note the above requirement gives conditions on the intersection points between the collections of arcs and ovals that contain the same state circle. So, in order to have an intersection point $z=(z_1,...,z_{2n})$ we should choose for each state circle $C_i$ (which is bi-colored into let's say $2m_i$ components) $m_i$ points on the punctured surface $(z^i_1,...,z^i_{m_i})$ with the following requirement:
\begin{enumerate}
    \item each arc from $F(\sigma)$ bounding $C_i$ contains exactly one component from the set $\{z^i_1,...,z^i_{m_i}\}$;
    \item each oval from $L(\sigma)$ bounding $C_i$ contains exactly one component from the set $\{z^i_1,...,z^i_{m_i}\}$.
\end{enumerate}

Then our intersection point $z$ will be given by the collection of chosen points $z^i=(z^i_1,...,z^i_{2m_i})$ associated to each state circle: $$\{z^i_1,...,z^i_{2m_i}\mid 1\leq i \leq | \sigma| \  \}.$$
Now, for our intersection, by Proposition \ref{formint} we have to compute two parts:
\begin{enumerate}[i)]
\item the loop $l_{z}$ (Definition \ref{d.associatedloop}) and its evaluation through the local system $\Phi$;
\item the sign $\alpha_z$.
\end{enumerate}

The sign will be the product of signs associated to each component from the set of state circles. 

The main remark that is used for the next step is that given the property that the intersection point $z$ has its components separated into components associated to the state circles, the loop $l_{z}$ and then its evaluation through the local system are both obtained from separate components associated to each state circle separately.
For $\Phi(l_{z})$ we compute the contribution of each loop associated to a state circle $l_{z^i}$ and consider the product of all of these when we evaluate the monodromies around the punctures through the local system and then the specialization. 

\

\noindent {\bf Step 2- Intersection points associated to one state circle}

\

As explained in Step 1, the intersection between our two homology classes can be computed by investigating each individual state circle and computing its contribution to the monodromy of the local system, as below.

\begin{lem}[Intersection points along a closed component] Let us fix a state circle $C_i$. Then, there are exactly two choices for the intersection points associated to this component: ${B}^i=(b^i_1,...,b^i_{m_i})$ or $\bar{B}^i=(\bar{b}^i_1,...,\bar{b}^i_{m_i})$.
\end{lem}

\begin{proof}
First, we look at the components of the base point $\bf d$ that belong to this state circle and denote them as: 
$(b^i_1,...,b^i_{m_i}).$

Let us denote by:
\begin{itemize}
    \item $F^i(\sigma)$  the set of arcs from $F(\sigma)$ bounding $C_i$, and 
 \item $L^i(\sigma)$  the set of ovals from $L(\sigma)$ bounding $C_i$.
\end{itemize}

\

\noindent {\bf Requirement for intersection points}\\
We have seen that an intersection point associated to a state circle $C_i$ is given by a choice of $2m_i$ points on the punctured surface $(z^i_1,...,z^i_{m_i})$ with the following requirement:
\begin{enumerate}
    \item each arc from $F^i(\sigma)$ contains exactly one component from the set $\{z^i_1,...,z^i_{m_i}\}$;
    \item each oval from $L^i(\sigma)$ contains exactly one component from the set $\{z^i_1,...,z^i_{m_i}\}$.
\end{enumerate}
We remark that we have a set of special points that belong to our collection of arcs and ovals:
$$(b^i_1,...,b^i_{m_i}).$$ Following our construction, we remark that these base points, when following the state circle, are chosen at intervals at step of length $2$, as illustrated in Figure \ref{f.arcsovalso}. This means that $${B}^i=(b^i_1,...,b^i_{m_i})$$ provides a well defined intersection point that satisfies the above {\bf Requirement for intersection points}. 

Now, we investigate which other intersection points satisfy this requirement.
Consider the oval containing $b^i_1$. If for this oval we do not choose the intersection point $b^i_1$, we are obliged to choose the other point at the intersection with the next arc. Let us denote this point by $\bar{b}^i_1$. Then, we look at the next oval following the orientation, containing the base point $b^i_2$. We have to choose a component belonging to this oval as well. It cannot coincide with $b^i_2$, since $b^i_2$ and $b^i_1$ belong to the same arc. So we have to choose the next point, which we denote by $\bar{b}^i_2$. 
Following this inductive procedure, we see that all components are prescribed by the choice of the point $\bar{b}^i_1$, and so we have a unique choice for a second intersection point $$\bar{B}^i=(\bar{b}^i_1,...,\bar{b}^i_{m_i})$$ that satisfies the above requirement. 
Overall, we see that we have exactly two intersection points: one given by the components close to the base points $B^i=(b^i_1,...,b^i_{m_i})$, and a second one whose components are given by the other possible situation $\bar{B}^i=(\bar{b}^i_1,...,\bar{b}^i_{m_i})$.
\end{proof}

\begin{figure}[H]
\def\svgwidth{.7\columnwidth}
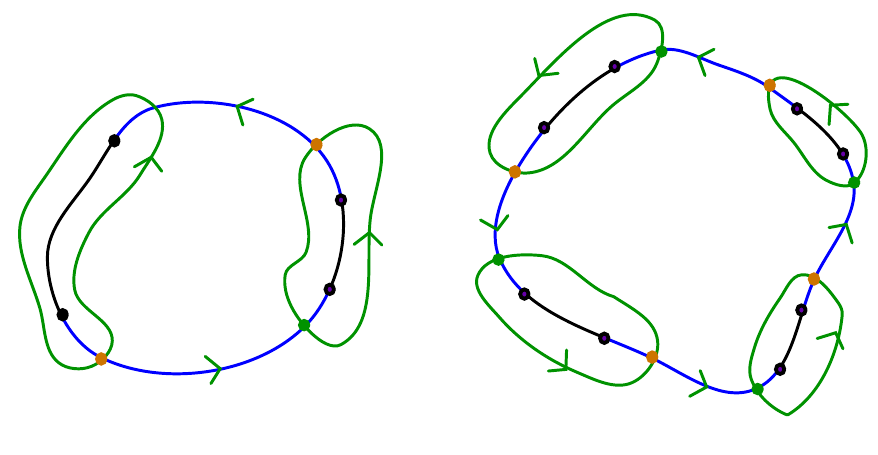
\caption{\label{f.loopx} A loop $l_x$ (in yellow) in the configuration space.}
\end{figure}

\noindent{\bf Step 3- Computing the contribution of the loop associated to one state circle}

\

For this step we aim to compute the contribution of the intersection points $B^i$ or $\bar{B}^i$ once we evaluate the local system and specialize it through $\alpha^{\sigma}_Q$ (Definition \ref{d.qspecial}). For this, we have to see which loops are associated to these intersection points. 

\noindent {\bf a) Intersection point close to the base point}

If we have chosen the components of $B
^i$, then the loop associated to it is a constant loop whose components are the base points $$b^i_1,...,b^i_{m_i}.$$ We remark that this loop does not wind around any puncture, so the contribution of the intersection point $B^i$ to the local system gets evaluated to $1$.

\noindent{\bf b) Intersection point opposite to the base point}

If we have chosen the components of $\bar{B}^i$, then the loop associated to it starts in the base points $$(b^i_1,...,b^i_{m_i-1},b^i_{m_i}),$$ continues on the ovals to the components $$\bar{b}^i_1,...,\bar{b}^i_{m_i}$$ and goes back to the ``next'' base points following the arcs, arriving in
$$(b^i_2,...,b^i_{m_i},b^i_1).$$
We remark that this is a loop in the configuration space whose image on the  
surface is, up to a little isotopy precisely the state circle $C_i$. So, its contribution to the local system gets evaluated to the product of monodromies of punctures associated to the state disc $D^{C_i}$. Following the definition of the change of coefficients $\alpha^{\sigma}_Q$, this product is precisely $Q$, since $\alpha^{\sigma}_Q$ was defined exactly such that it satisfies the {\bf Monodromy Requirement} in Section \ref{SS:spec}.  

We conclude that the contribution to the local system of the intersection point $\bar{B}^i$ gets evaluated to $Q$.

\

\noindent {\bf Step 4- Sign contribution of intersection points associated to a state circle}

\

For each intersection point $z$ in the configuration space, we have to compute the sign of the intersection between the submanifolds given by $F(\sigma)$ and $L(\sigma)$ at $z$. This sign is given by:
\begin{enumerate}
    \item the product of the sign of local orientations at the components of $z$ on the surface;
\item the sign of the permutation induced by this intersection point on the components of the geometric supports: $F(\sigma)$ and $L(\sigma)$.
\end{enumerate}

\noindent {\bf a) Sign contribution of the intersection point $B^i$}

For the point $B^i$ we remark that all intersections between arcs and ovals have positive local orientations and also there is no permutation induced at the level of the set of arcs and ovals. So this point has a sign contribution of $$+1.$$

\noindent {\bf b) Sign contribution of the intersection point  $\bar{B}^i$ }

Looking at the second intersection point $\bar{B}^i$, we remark that all intersections between arcs and ovals have negative local orientations. This gives the sign:
$$(-1)^{m_i}.$$
Also the permutation induced by its components can be seen by checking the permutation induced by the associated loop on the base points. This, as we have seen, is a cyclic permutation of order $m_i$, which has the signature $(-1)^{m_i-1}$. So, the sign coming from the permutation associated to $\bar{B}^i$ is:
$$(-1)^{m_i-1}.$$
Overall, the point $\bar{B}^i$ carries the sign:
$$(-1)^{m_i-1} \cdot (-1)^{m_i}=-1.$$

\noindent {\bf Step 5- Contribution of intersection points associated to one circle to the intersection form}

Overall for our two points $B^i$ and $\bar{B}^i$ we obtain the following contributions:
\begin{enumerate}
    \item 
$B^i$: contributes to the local system evaluation by $1$ and carries the sign
$1$.
    \item  $\bar{B}^i$: contributes to the local system evaluation by $-Q$ and carries the sign
$-1$.
\end{enumerate}
This shows that the component $C_i$ gives two possible choices for components of intersection points, and each of them gets weighted by the monomials presented above. This means that this component contributes to the intersection by the sum of these two weights, which is precisely:
\begin{equation}
    1+(-1) \cdot Q=1-Q.
\end{equation}
Doing this procedure for all state circles, and we have $| \sigma |$ of them, we obtain that our intersection is indeed: 
\begin{equation}
    \ll \mathscr{F}(\sigma), \mathscr{L}(\sigma) \gg_{\alpha^{\sigma}_Q}=(1-Q)^{|\sigma|}.
\end{equation}    
which concludes the proof of the Lemma.

\end{proof}

Now we are ready to put together all the previous set-up and show that we have defined the right topological model, provided by the set of graded intersections of homology classes which, for the appropriate choice of specialization of coefficients, lead to the Jones polynomial and to the HOMFLY-PT polynomial, respectively.

\subsection{Intersection form recovering the evaluation on closed components} \label{ss.intJoneseval}
Recall that in Definition \ref{specJ} we have introduced the change of coefficients  
 \begin{equation}
      \begin{aligned}
&\alpha^{\sigma}_{J}: \Z[x_1^{\pm 1},...,x_{8n}^{\pm 1}] \rightarrow \Z[q^{\pm1}](1-q-q^{-1}).
\end{aligned}
 \end{equation}
This gives the associated specialized intersection pairing:
\begin{equation}
\ll ~,~ \gg_{\alpha^{\sigma}_{J}}: \mathscr H^{lf}_{2n}\mid_{\alpha^{\sigma}_{J}} \otimes \mathscr H_{2n}\mid_{\alpha^{\sigma}_{J}} \rightarrow\Z[q^{\pm 1}](1-q-q^{-1}).
\end{equation}

\begin{lem}[Recovering the Kauffman bracket via intersections of arcs and ovals]\label{l.intJ}
Fix a state $\sigma$. Then, the intersection between the classes $\mathscr 
F(\sigma)$ and $\mathscr 
L(\sigma)$ specialized through $\alpha^{\sigma}_{J}$ gives precisely the Kauffman bracket evaluated on the state $\sigma$:
\begin{equation}
    \ll \mathscr F(\sigma), \mathscr L(\sigma) \gg_{\alpha^{\sigma}_{J}}=(q+q^{-1})^{|\sigma|}.
\end{equation}    
\end{lem}
\begin{proof}
Following Lemma \ref{intgeneral}, we have that:
\begin{equation}
    \ll \mathscr F(\sigma), \mathscr L(\sigma) \gg_{\alpha^{\sigma}_{Q}}=(1-Q)^{|\sigma|}.
\end{equation}
In our case we have $Q=Q_J=1-q-q^{-1}$ and so this leads to our desired formula for the intersection:
\begin{equation}
    \ll \mathscr F(\sigma), \mathscr L(\sigma) \gg_{\alpha^{\sigma}_{J}}=(q+q^{-1})^{|\sigma|}.
\end{equation}
\end{proof}

\subsection{A topological model for the Jones polynomial} \label{ss.topmodelJones}

In this section, we obtain a topological model of the Jones polynomial by defining $\Theta_J$ and proving Theorem \ref{t.main2}. To define $\Theta_J$, recall the Jones polynomial of an oriented link may be defined in terms of the Kauffman bracket (Definition \ref{d.jones}) of a diagram $D$ of the link with orientation forgotten: 
\begin{equation} 
J_L(q) =  (-1)^{3w(D)} \sum_{\sigma \text{ a Kauffman state on $D$}} (-q)^{\frac{-\sgn(\sigma)+3w(D)}{2}}(q+ q^{-1})^{|\sigma|}.
\label{e.Jonesstate}\end{equation}
Since Lemma \ref{intgeneral} recovers $(q+ q^{-1})^{|\sigma|}$ as the intersection pairing $\ll ~,~ \gg_{\alpha^{\sigma}_{J}}$ for each Kauffman state $\sigma$, it remains to give a geometric interpretation of the monomial $q^{\frac{-\sgn(\sigma)+3w(D)}{2}}$.

\begin{defn}[Twist of the $\alpha'$-curves] \label{d.twistedcurvesJ}
 At each crossing $\chi$, we distinguish a pair of the $\alpha'$ curves by whether they give the oriented resolution or the non-oriented resolution. They are denoted by $\alpha_o(\chi)$ for the oriented resolution, $\alpha_{\hat{o}}(\chi)$ for the non-oriented resolution, respectively. Let $\sigma$ be a Kauffman state on a diagram $D$ of $L$, we twist the $\alpha'$ arcs by $\sigma$ as follows.

\paragraph{\textbf{If $\chi$ is positive:}} If $\sigma$ chooses the $-$-resolution at the crossing, then replace by the twisted $\alpha_{o}$-arcs as indicated in Figure \ref{f.ct}. If $\sigma$ chooses the $+$-resolution, then replace by the $\alpha_{\hat{o}}$-arcs (without twisting). 
\paragraph{\textbf{If $\chi$ is negative:}}
If $\sigma$ chooses the $+$-resolution at the crossing, then replace by the twisted $\alpha_{o}$-arcs as indicated in Figure \ref{f.ct}. If $\sigma$ chooses the $-$-resolution, then replace by the  $\alpha_{\hat{o}}$-arcs (without twisting).

\begin{defn}[Twisted curves $D_\sigma(\alpha)$]
    \label{d.dtbystate}
We denote the set of curves resulting from this twisting by $D_\sigma(\alpha)$. See Figure \ref{f.ct}. In the figure, a dashed red line is in the back, and a solid red line is in the front. 
\end{defn}

\begin{figure}[H]
\def \svgwidth{.9\columnwidth}
\begingroup%
  \makeatletter%
  \providecommand\color[2][]{%
    \errmessage{(Inkscape) Color is used for the text in Inkscape, but the package 'color.sty' is not loaded}%
    \renewcommand\color[2][]{}%
  }%
  \providecommand\transparent[1]{%
    \errmessage{(Inkscape) Transparency is used (non-zero) for the text in Inkscape, but the package 'transparent.sty' is not loaded}%
    \renewcommand\transparent[1]{}%
  }%
  \providecommand\rotatebox[2]{#2}%
  \newcommand*\fsize{\dimexpr\f@size pt\relax}%
  \newcommand*\lineheight[1]{\fontsize{\fsize}{#1\fsize}\selectfont}%
  \ifx\svgwidth\undefined%
    \setlength{\unitlength}{925.13827587bp}%
    \ifx\svgscale\undefined%
      \relax%
    \else%
      \setlength{\unitlength}{\unitlength * \real{\svgscale}}%
    \fi%
  \else%
    \setlength{\unitlength}{\svgwidth}%
  \fi%
  \global\let\svgwidth\undefined%
  \global\let\svgscale\undefined%
  \makeatother%
  \begin{picture}(1,0.19483476)%
    \lineheight{1}%
    \setlength\tabcolsep{0pt}%
    \put(0,0){\includegraphics[width=\unitlength,page=1]{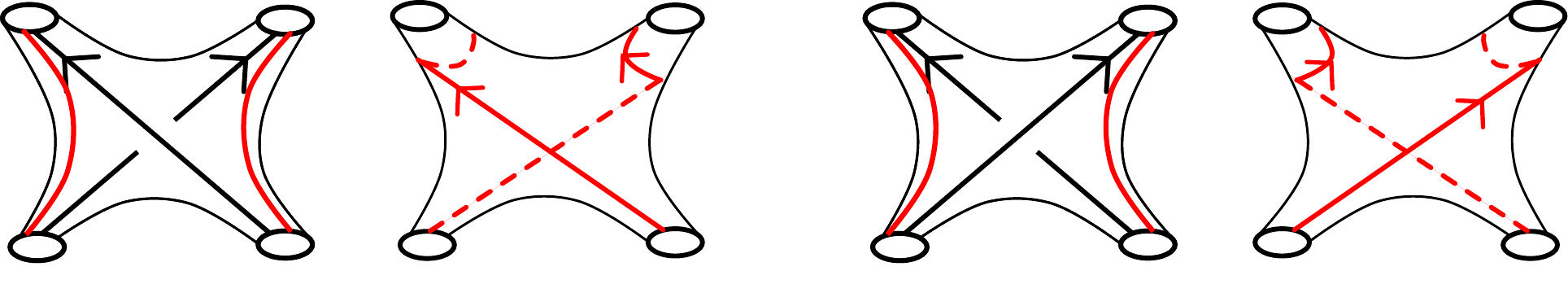}}%
    \put(0.3443454,0.00260941){\color[rgb]{0,0,0}\makebox(0,0)[lt]{\lineheight{1.25}\smash{\begin{tabular}[t]{l}$\alpha_o$\end{tabular}}}}%
    \put(0.08915036,0.00260941){\color[rgb]{0,0,0}\makebox(0,0)[lt]{\lineheight{1.25}\smash{\begin{tabular}[t]{l}$\alpha_{\hat{o}}$\end{tabular}}}}%
    \put(0.63401426,0.00260941){\color[rgb]{0,0,0}\makebox(0,0)[lt]{\lineheight{1.25}\smash{\begin{tabular}[t]{l}$\alpha_{\hat{o}}$\end{tabular}}}}%
    \put(0.88852787,0.00260941){\color[rgb]{0,0,0}\makebox(0,0)[lt]{\lineheight{1.25}\smash{\begin{tabular}[t]{l}$\alpha_{o}$\end{tabular}}}}%
  \end{picture}%
\endgroup%

\caption{\label{f.ct} Twisted curves $D_\sigma(\alpha)$ (in red).}
\end{figure}
\end{defn} 
We rewrite the exponent in the monomial $q^{\frac{-\sgn(\sigma)+3w(D)}{2}}$ in terms of the oriented intersection number $i(D_{\sigma}(\alpha), D)$ of $D_\sigma(\alpha)$ and the link diagram $D$. 
Figure \ref{f.cts} summarizes the local contributions to the intersection number. 
\begin{figure}[H]
\def \svgwidth{.5\columnwidth}
\begingroup%
  \makeatletter%
  \providecommand\color[2][]{%
    \errmessage{(Inkscape) Color is used for the text in Inkscape, but the package 'color.sty' is not loaded}%
    \renewcommand\color[2][]{}%
  }%
  \providecommand\transparent[1]{%
    \errmessage{(Inkscape) Transparency is used (non-zero) for the text in Inkscape, but the package 'transparent.sty' is not loaded}%
    \renewcommand\transparent[1]{}%
  }%
  \providecommand\rotatebox[2]{#2}%
  \newcommand*\fsize{\dimexpr\f@size pt\relax}%
  \newcommand*\lineheight[1]{\fontsize{\fsize}{#1\fsize}\selectfont}%
  \ifx\svgwidth\undefined%
    \setlength{\unitlength}{450.86420014bp}%
    \ifx\svgscale\undefined%
      \relax%
    \else%
      \setlength{\unitlength}{\unitlength * \real{\svgscale}}%
    \fi%
  \else%
    \setlength{\unitlength}{\svgwidth}%
  \fi%
  \global\let\svgwidth\undefined%
  \global\let\svgscale\undefined%
  \makeatother%
  \begin{picture}(1,0.38024067)%
    \lineheight{1}%
    \setlength\tabcolsep{0pt}%
    \put(0,0){\includegraphics[width=\unitlength,page=1]{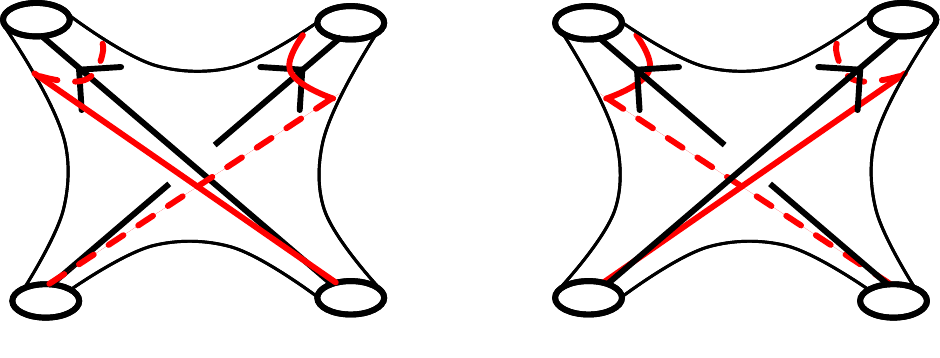}}%
    \put(0.77207861,0.00469152){\color[rgb]{0,0,0}\makebox(0,0)[lt]{\lineheight{1.25}\smash{\begin{tabular}[t]{l}$+2$\end{tabular}}}}%
    \put(0.15525484,0.00469152){\color[rgb]{0,0,0}\makebox(0,0)[lt]{\lineheight{1.25}\smash{\begin{tabular}[t]{l}$-2$\end{tabular}}}}%
    \put(0,0){\includegraphics[width=\unitlength,page=2]{curves_twisted_slope.pdf}}%
  \end{picture}%
\endgroup%

\caption{\label{f.cts}  Local contribution to $i(D_{\sigma}(\alpha), D)$.  }
\end{figure}
At a negative crossing, choosing the $\alpha_o$ arcs associated to the $+$-resolution contributes $-2$. At a positive crossing, choosing the $\alpha_{\hat{o}}$ arcs associated to the $-$-resolution contributes $+2$. 
The contribution in the remaining cases, of choosing the $+$-resolution, and therefore the $\alpha_o$ arcs,  at a positive crossing, and of choosing the $-$-resolution, and therefore the $\alpha_{\hat{o}}$ arcs at a negative crossing, is zero. 
\begin{lem}[Geometric interpretation state sum coefficients] \label{l.intslope}
The degree $\frac{-\sgn(\sigma)+3w(D)}{2}$ of the monomial of $q$ in \eqref{e.Jonesstate}
is given by the oriented intersection (linking number) $i(D_{\sigma}(\alpha), D)$ of the twisted curves $D_{\sigma}(\alpha)$ with the link diagram $D$: 
\[ \frac{-\sgn(\sigma)+3w(D)}{2} = -\frac{i(D_{\sigma}(\alpha), D)}{2} +w(D). \]
\end{lem}

\begin{rem}
     Note $D_{\sigma}(\alpha)$ may be written in terms of homology classes of loops around the punctures on the 4-punctured spheres, and therefore the meridians of the surface $\Sigma$, but it may not necessarily lift to the configuration space if it 
      does not evaluate to 0 by the local system $\Phi$, which is why it is not part of the intersection form $\ll ~,~ \gg_{\alpha^{\sigma}_J}$ defined in Section \ref{ss.intJoneseval}. 
  
\end{rem}

\begin{proof}
We adapt the proof in \cite{FKP-adequate} to our setting. 
Recall $\sgn(\sigma) = n_+(\sigma) - n_-(\sigma)$, where $n_+(\sigma)$ is the number of crossings of $D$ on which $\sigma$ chooses the $+$-resolution, and $n_-(\sigma)$ is the number of crossings of $D$ on which $\sigma$ chooses the $-$-resolution. Finally, $n(D)$ is the number of crossings in the set of crossings $c(D)$ of $D$. 

The writhe $w(D) = n^+(D) - n^-(D)$ is the difference of the number of positive crossings of $D$ with the number of negative crossings of $D$. 

Let $n_-^+(\sigma)$ be the number of positive crossings on which $\sigma$ chooses the $-$-resolution, and let $n_+^-(\sigma)$ be the 
number of negative crossings on which $\sigma$ chooses the $+$-resolution. 

Then,
\begin{align*}
\sgn(\sigma) - w(D) &= n_+(\sigma) - n_-(\sigma) - (n^+(D) - n^-(D)). \\ 
\intertext{Decomposing} 
n_+(\sigma)  &= n_+^-(\sigma) + n_+^+(\sigma), \qquad n_-(\sigma) = n_-^-(\sigma) + n_-^+(\sigma) \\ 
\intertext{ and }
 n^+(D) &= n_+^+(\sigma) + n_-^+(\sigma), \qquad  n^-(D) = n_+^-(\sigma) + n_-^-(\sigma), 
 \intertext{we obtain} 
 \sgn(\sigma) - w(D) &= n_+(\sigma) - n_-(\sigma) - (n^+(D) - n^-(D)) \\ 
 &= n_+^-(\sigma) + n_+^+(\sigma) - (n_-^-(\sigma) + n_-^+(\sigma))   \\ 
 & -(n_+^+(\sigma) + n_-^+(\sigma) - (n_+^-(\sigma) + n_-^-(\sigma))) \\
 &= 2n^-_+(\sigma) - 2n_-^+(\sigma). 
\end{align*}
Thus we can rewrite 
\[ \frac{\sgn(\sigma)-3w(D)}{2} = \frac{\sgn(\sigma)-w(D) - 2w(D)}{2} = \frac{2n^-_+(\sigma) - 2n_-^+(\sigma)}{2} - w(D). \]

We claim 
\[ i(D_{\sigma}(\alpha), D) = 2n^-_+(\sigma) - 2n_-^+(\sigma) = \sgn(\sigma) - w(D), \]
since we can compute the linking number by summing over the local contributions. On a positive crossing where $\sigma$ chooses the $+$-resolution, there is no intersection between $D_{\sigma}(\alpha)$ and $D$. Similarly for a negative crossing on which $\sigma$ chooses the $-$-resolution, there is no intersection between $D_{\sigma}(\alpha)$ and $D$. In the rest of the cases the intersection contributes $\pm 2$ to the linking number, as illustrated in Figure \ref{f.cts}.

\end{proof}
\begin{rem}[Twisting curves and boudary slopes] \label{r.tbdslope}
We remark that from \cite{FKP-adequate}, the curves $D_{\sigma}(\alpha)$ are isotopic to the boundary curves of the $\sigma$-state surface in $S^3\setminus L$, and the oriented intersection number $i(D_{\sigma}(\alpha), D)$ recovers the boundary slope of the surface if it is essential. 
\end{rem}

Using the geometric interpretation of the coefficients of the state sum via linking numbers of twisted curves, we may rewrite the state-sum formula for the Jones polynomial from \eqref{e.Jonesstate} as
\begin{align} \label{e.jfinalss}
J_L(q) & =  (-1)^{3w(D)} \sum_{\sigma \text{ a Kauffman state on $D$}} (-q)^{\frac{-\sgn(\sigma)+3w(D)}{2}}(q+ q^{-1})^{|\sigma|} \notag \\ 
&= (-1)^{3w(D)} \sum_{\sigma \text{ a Kauffman state on $D$}} (-q)^{-\frac{i(D_{\sigma}(\alpha), D)}{2} +w(D)}(q+ q^{-1})^{|\sigma|}.
\end{align}

We are now ready to prove Theorem \ref{t.main2}. 

\subsection{Proof of Theorem \ref{t.main2}} \label{ss.pmain2}
Let $i(D_\sigma(\alpha), D)$ be the geometric intersection between the curve $D_\sigma(\alpha)$ and the link diagram $D$ defined in the previous section, Section \ref{ss.topmodelJones}, $w(D)$ be the writhe of $D$, and $\ll , \gg_{\alpha^{\sigma}_J}$ be the intersection pairing defined in Section \ref{ss.intJoneseval}. 
\begin{defn}[Intersection form for the Jones polynomial]\label{intformulaJ}
Let $F(\sigma)$, $L(\sigma)$ be the two collections of curves defined (Definition \ref{d.FandL}). They give the homology classes $\mathscr F(\sigma)$, $\mathscr L(\sigma)$ defined in Definition \ref{classesJ}. We consider their intersection pairing specialized as in Section \ref{ss.intJoneseval}. 

    We define the function $\Theta_J(D)$ as 
    \begin{align*} \Theta_J(D)(q) &:= (-1)^{3w(D)} \cdot  \\
    & \cdot\sum_{\sigma \text{ a Kauffman state}} q^{\frac{-i(D_\sigma(\alpha), D)}{2} + w(D)} \ll \mathscr F(\sigma) , \mathscr L(\sigma)  \gg_{\alpha^{\sigma}_{J}} .
    \end{align*}
\end{defn}

We restate Theorem \ref{t.main2} for the convenience of the reader. 

\begin{tthm}[Theorem \ref{t.main2}) (Topological model for the Jones polynomial from link diagrams]
\hspace{0pt}\\
Let $\Theta_J(D) \in \mathbb{Z}[q^{\pm 1}]$ be the state sum of graded intersections in the same configuration space on our fixed Heegaard surface, as in Definition \ref{intJintr}. Then it provides a topological model for the Jones polynomial, relying just on the choice of the diagram $D$:
\begin{equation*}
\Theta_J(D) = J_L(q). 
\end{equation*}
\end{tthm}

\begin{proof}
By Lemma \ref{l.intslope}, $q^{\frac{-i(D_\sigma(\alpha), D)}{2}+w(D)}$ in $\Theta_J$ recovers $q^{\frac{-\sgn(\sigma)+3w(D)}{2}}$ in the Kauffman state sum  definition \eqref{e.Jonesstate} of the Jones polynomial. Lemma \ref{l.intJ} implies that the term $\ll \mathscr{F}(\sigma), \mathscr{L}(\sigma) \gg_{\alpha^\sigma_J}$  recovers $(q+q^{-1})^{|\sigma|}$ in the state sum definition. It follows that $\Theta_J$ recovers the Jones polynomial.  
\end{proof}

\section{Recovering the HOMFLY-PT polynomial}
\label{S:H}

The HOMFLY-PT polynomial satisfies the following skein relations  and recovers the Jones polynomial and the Alexander polynomial through specializations of its variables. 
\begin{figure}[H]
\begin{center} 
\def \svgwidth{.5\columnwidth}
\begin{equation} \label{e.homfly}
 aP_{(\hpc)} - a^{-1}P_{(\hnc)} = zP_{(\hco)} \qquad \qquad P_U(a, z) = 1.\end{equation}
\end{center}
\end{figure} 
Here $U$ denotes the standard projection of the unknot on the projection plane. Suppose $U^m$ is an unlink with $m$ components, then 
\[ P_{U^m}(a, z) = \left( \frac{a-a^{-1}}{z}\right)^{m-1}.  \]

We recover the Jones polynomial $J_L(q)$ and the symmetrized Alexander polynomial $\Delta_L(t)$ through the following specializations: 
\begin{align*}
J_L(q) &= P_L(a, z)\rvert_{a = q^{-2}, \ z = q - q^{-1}} \\ 
\Delta_L(t) &= P_L(a, z)\rvert_{a = 1, \ z = t^{1/2} - t^{-1/2}}. 
\end{align*}

In this section we describe a topological model for the HOMFLY-PT polynomial from intersections of curves on a Heegaard surface. The common ingredient for the construction is a state sum model, which is given by the Kauffman bracket for the Jones polynomial, and in the case of the HOMFLY-PT polynomial is given by the skein relation \eqref{e.homfly}.

\subsection{Descending diagrams}
We recall an algorithm for applying the skein relation $\eqref{e.homfly}$ to compute the HOMFLY-PT polynomial using a sum over descending link diagrams. For the original reference, see \cite{Hoste} for Hoste's construction of the HOMFLY-PT polynomial and the argument that it is a link invariant using descending diagrams. The definition of nondescending complexity $nd(D)$ used in this paper is due to Murasugi and Przytycki  \cite{MP}. 

\begin{defn}[Nugatory crossing]
Let a link $L\subset S^3$ be represented by a diagram $D$. A \textit{nugatory crossing} of $D$ is a crossing of $D$ for which there exists a circle in the projection plane meeting the diagram transversely at that crossing, but not meeting the diagram at any other point. 
\end{defn}

A nugatory crossing of a diagram can be removed by rotating the portion of the diagram contained in the circle without changing the isotopy type of the link. 
In the next part, we define a subset $nd(D)$ of the crossings to represent the non-descending complexity of the diagram $D$, as follows.

\begin{defn}[Non-descending crossings] 
Let $D = D_1, \ldots, D_{k}$ be an ordering of the components of an oriented link diagram $D$, and let $\{\ell_1, \ldots, \ell_k\}$ where $\ell_i \in D_i$ be a set of base points (note: different from the base points $b_i$ considered previously) that avoid double points of crossings in $D$.
Define $r(D)$ to be the set of crossings reached by traversing each component of the link in given order starting with $\ell_1$ on $D_1$, where we terminate the travel before a nonnugatory crossing is met for the first time, and the strand being traversed is under another link strand at the crossing. We travel from the base point $\ell_{i+1}$ of the next component $D_{i+1}$ only if we finish traversing $D_i$ without encountering such a crossing. 

A crossing of $D$ is in $nd(D)$ if it is not in $r(D)$, and $nd(D)$ is the set of \textit{non-descending} crossings of the diagram $D$. 
 \end{defn} 
\begin{defn}[Non-descending complexity]
    
 The number of crossings in $nd(D)$, denoted by $|nd(D)|$, represents the nondescending complexity of $D$. If $nd(D)$ is empty, so there is no nondescending complexity, then $D$ is called a \textit{descending diagram}. A descending link diagram represents an unlink where the component $D_j$ is below $D_i$ if $j > i$.  
\end{defn}
 \begin{figure}[H]
 \def \svgwidth{.25\columnwidth}
\begingroup%
  \makeatletter%
  \providecommand\color[2][]{%
    \errmessage{(Inkscape) Color is used for the text in Inkscape, but the package 'color.sty' is not loaded}%
    \renewcommand\color[2][]{}%
  }%
  \providecommand\transparent[1]{%
    \errmessage{(Inkscape) Transparency is used (non-zero) for the text in Inkscape, but the package 'transparent.sty' is not loaded}%
    \renewcommand\transparent[1]{}%
  }%
  \providecommand\rotatebox[2]{#2}%
  \newcommand*\fsize{\dimexpr\f@size pt\relax}%
  \newcommand*\lineheight[1]{\fontsize{\fsize}{#1\fsize}\selectfont}%
  \ifx\svgwidth\undefined%
    \setlength{\unitlength}{474.00668407bp}%
    \ifx\svgscale\undefined%
      \relax%
    \else%
      \setlength{\unitlength}{\unitlength * \real{\svgscale}}%
    \fi%
  \else%
    \setlength{\unitlength}{\svgwidth}%
  \fi%
  \global\let\svgwidth\undefined%
  \global\let\svgscale\undefined%
  \makeatother%
  \begin{picture}(1,1.0499527)%
    \lineheight{1}%
    \setlength\tabcolsep{0pt}%
    \put(0.395272,0.00509287){\color[rgb]{0,0,0}\makebox(0,0)[lt]{\lineheight{1.25}\smash{\begin{tabular}[t]{l}$n(D) = 3$\end{tabular}}}}%
    \put(0,0){\includegraphics[width=\unitlength,page=1]{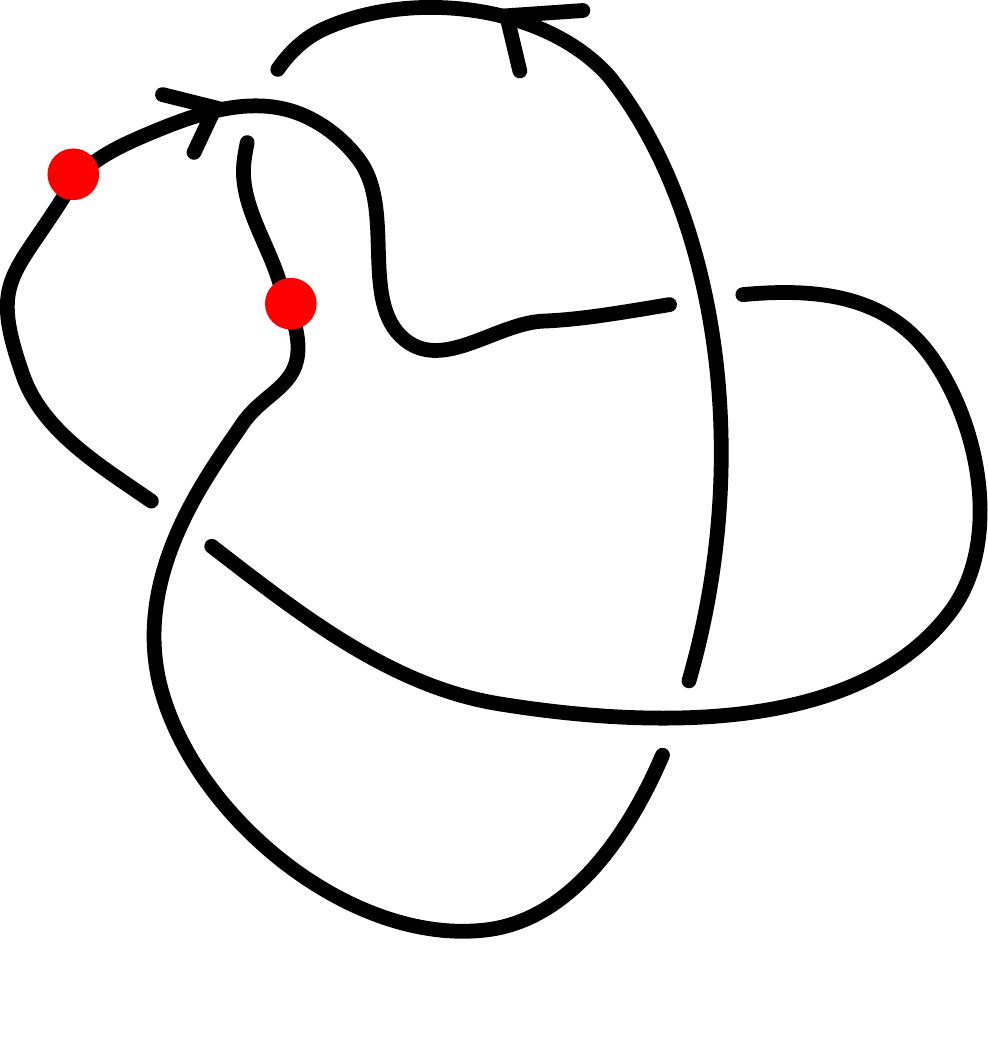}}%
    \put(0.02352113,0.98077175){\color[rgb]{0,0,0}\makebox(0,0)[lt]{\lineheight{1.25}\smash{\begin{tabular}[t]{l}$\ell_1$\end{tabular}}}}%
    \put(0.16181318,0.70218908){\color[rgb]{0,0,0}\makebox(0,0)[lt]{\lineheight{1.25}\smash{\begin{tabular}[t]{l}$\ell_2$\end{tabular}}}}%
    \put(0,0){\includegraphics[width=\unitlength,page=2]{ndcsingle.pdf}}%
    \put(0.38939092,0.56974541){\color[rgb]{1,0,0}\makebox(0,0)[lt]{\lineheight{1.25}\smash{\begin{tabular}[t]{l}$r(d)$\end{tabular}}}}%
  \end{picture}%
\endgroup%

 \caption{ \label{f.ndlink} A link diagram with nondescending complexity 3.}
 \end{figure}
 The nondescending complexity of the link with diagram $D$ shown in Figure \ref{f.ndlink} is $n(D) = 3$, which represents the number of crossings not included in the highlighted travelled portion $r(D)$.
 
\subsection{Algorithm for computing the HOMFLY-PT polynomial.}
Let $D$ be a link diagram with nonzero nondescending complexity. 
    With a suitable choice of base points and ordering of link components of the diagrams in the skein relation, the HOMFLY-PT polynomial of a link with diagram $D$ can be evaluated as a sum of unlinks through applying the skein relation $\eqref{e.homfly}$ to $D$, in a way that $|nd(D')|$ of intermediate diagrams $D'$ decreases at every step until it reaches 0. The sum of unlinks with polynomial coefficients in two variables from a skein relation of the HOMFLY-PT polynomial  then gives a state sum model \cite{Jaeger}. 

We describe how to apply the skein relation \eqref{e.homfly} to reduce the polynomial to a sum of the HOMFLY-PT polynomial of unlinks, following the description of \cite{Chmutov-Polyak}. The state sum model we describe based on the evaluation is indexed by the same states as in \cite{Jaeger}, which we will call \textit{Jaeger states}. 

\paragraph{\textbf{Construction of the binary tree $T_D$}}
From the skein relation \eqref{e.homfly} there is a choice of replacing a crossing $\chi$ by itself (which changes nothing), its mirror image $-\chi$ or the orientation preserving resolution $o$. We construct a directed binary tree $T_D$ to keep track of the choice of replacements at each crossing by $-\chi$ or $o$. 

For $T_D$, the vertices represent diagrams with $D$ as the root of the tree and a directed edge $e$ from $v_0$ to $v_1$ indicates that the link diagram $D_{v_1}$ is obtained from $v_0$ by replacing a crossing $\chi \in D_{v_0}$ by $-\chi$ or $o$.  

We will illustrate each step with a link diagram whose full binary tree $T_D$ is shown in Figure \ref{f.binarytree}. 

\begin{figure}[H]
    \centering
    \def\svgwidth{.2\columnwidth}
\begingroup%
  \makeatletter%
  \providecommand\color[2][]{%
    \errmessage{(Inkscape) Color is used for the text in Inkscape, but the package 'color.sty' is not loaded}%
    \renewcommand\color[2][]{}%
  }%
  \providecommand\transparent[1]{%
    \errmessage{(Inkscape) Transparency is used (non-zero) for the text in Inkscape, but the package 'transparent.sty' is not loaded}%
    \renewcommand\transparent[1]{}%
  }%
  \providecommand\rotatebox[2]{#2}%
  \newcommand*\fsize{\dimexpr\f@size pt\relax}%
  \newcommand*\lineheight[1]{\fontsize{\fsize}{#1\fsize}\selectfont}%
  \ifx\svgwidth\undefined%
    \setlength{\unitlength}{476.94786252bp}%
    \ifx\svgscale\undefined%
      \relax%
    \else%
      \setlength{\unitlength}{\unitlength * \real{\svgscale}}%
    \fi%
  \else%
    \setlength{\unitlength}{\svgwidth}%
  \fi%
  \global\let\svgwidth\undefined%
  \global\let\svgscale\undefined%
  \makeatother%
  \begin{picture}(1,1.37471648)%
    \lineheight{1}%
    \setlength\tabcolsep{0pt}%
    \put(0,0){\includegraphics[width=\unitlength,page=1]{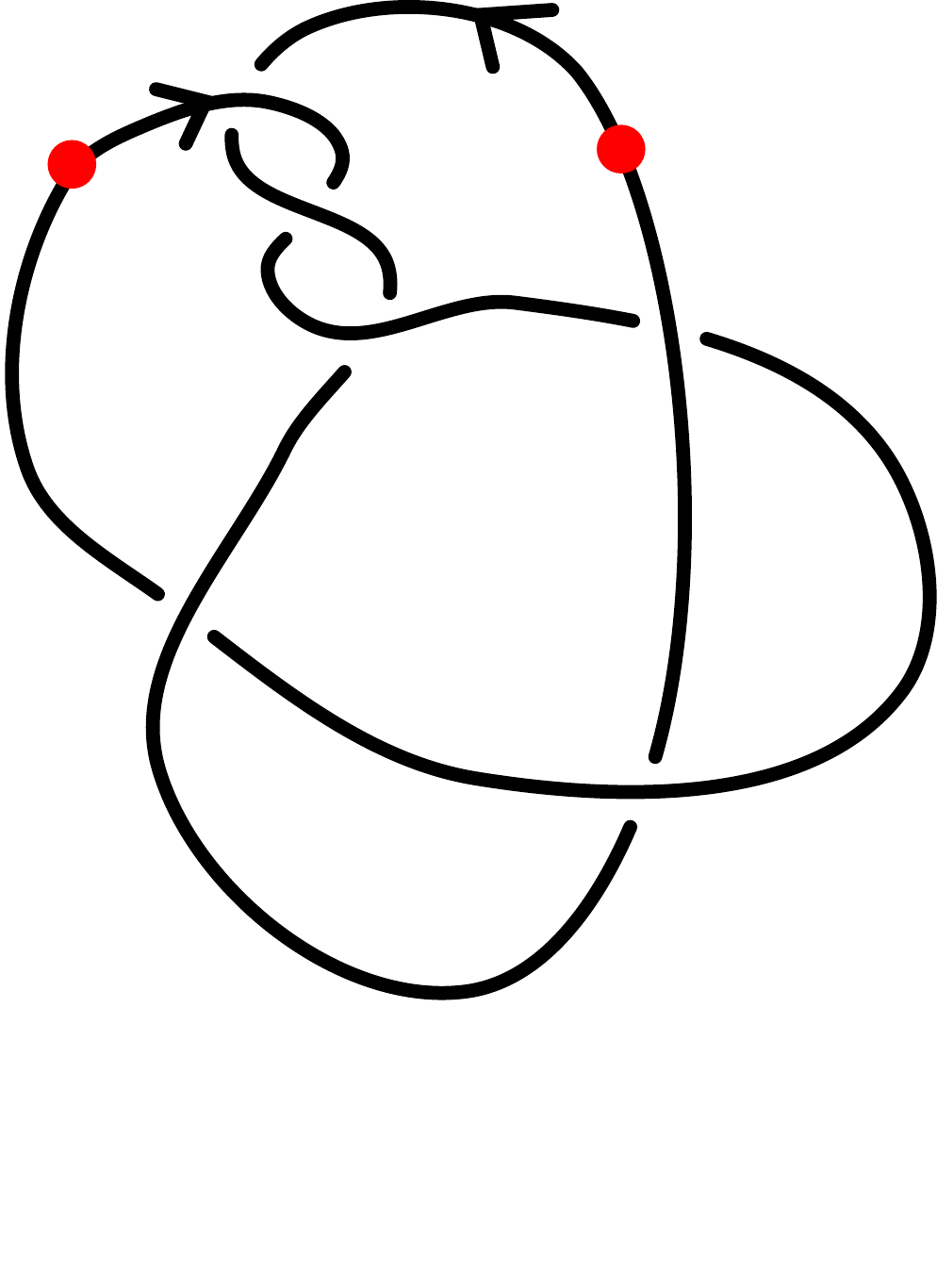}}%
    \put(-0.0019042,1.32797707){\color[rgb]{0,0,0}\makebox(0,0)[lt]{\lineheight{1.25}\smash{\begin{tabular}[t]{l}$\ell_1$\end{tabular}}}}%
    \put(0.7032952,1.22742737){\color[rgb]{0,0,0}\makebox(0,0)[lt]{\lineheight{1.25}\smash{\begin{tabular}[t]{l}$\ell_2$\end{tabular}}}}%
    \put(0.19482462,0.00506149){\color[rgb]{0,0,0}\makebox(0,0)[lt]{\lineheight{1.25}\smash{\begin{tabular}[t]{l}$n(D) = 5$\end{tabular}}}}%
    \put(0,0){\includegraphics[width=\unitlength,page=2]{ndcex1.pdf}}%
    \put(0.46012582,1.147298){\color[rgb]{0,0,0}\makebox(0,0)[lt]{\lineheight{1.25}\smash{\begin{tabular}[t]{l}$\chi$\end{tabular}}}}%
  \end{picture}%
\endgroup%

    \caption{A link diagran of nondescending complexity 5.}
    \label{f.ndcex1}
\end{figure}

\noindent \textbf{Step 1- Apply the HOMFLY-PT skein relation to a chosen crossing $\chi$ of $D$.} \\

Start with the graph $T_D$ consisting of a single vertex $v$ representing the link diagram $D$. Fix an ordering of the components of the diagram $D$, $D_1, \ldots, D_k$ and fix a set of base points $\{\ell_1, \ldots, \ell_k\}$, one on each component. Starting at the basepoint $\ell_1$, travel the strand of the component $D_1$ following the orientation on the link. Stop just before the first nonnugatory crossing $\chi$ that is undercrossed by the strand. If there is no such crossing on $\ell_1$, travel on $\ell_2$ starting from $b_2$, and so on for $\ell_3, \ell_4, \ldots, \ell_k$. If the travel through all the link components is finished without encountering such a crossing, then the diagram $D$ is descending with $nd(D)=0$, and the tree $T_D$ remains a single vertex. 

If such a crossing is encountered, say crossing $\chi$ on component $\ell_i$, then add two vertices $v_{-\chi}$ and $v_{o}$ representing the diagrams $D_{-\chi}$ and $D_{o}$, obtained by replacing the crossing $\chi$ with $-\chi$, $o$, respectively. Add an edge from $v$ to $v_{-\chi}$ and another edge from $v$ to $v_{o}$. 

Examine the diagram represented by the newly added vertices. If it is descending, then we do not add more edges or vertices. If it is  not descending, we fix an ordering on components and base points on the diagram depending on whether it is  $D_{-\chi}$ and $D_{o}$ as follows: 

\begin{figure}[H]
    \centering
    \def\svgwidth{.5\columnwidth}
    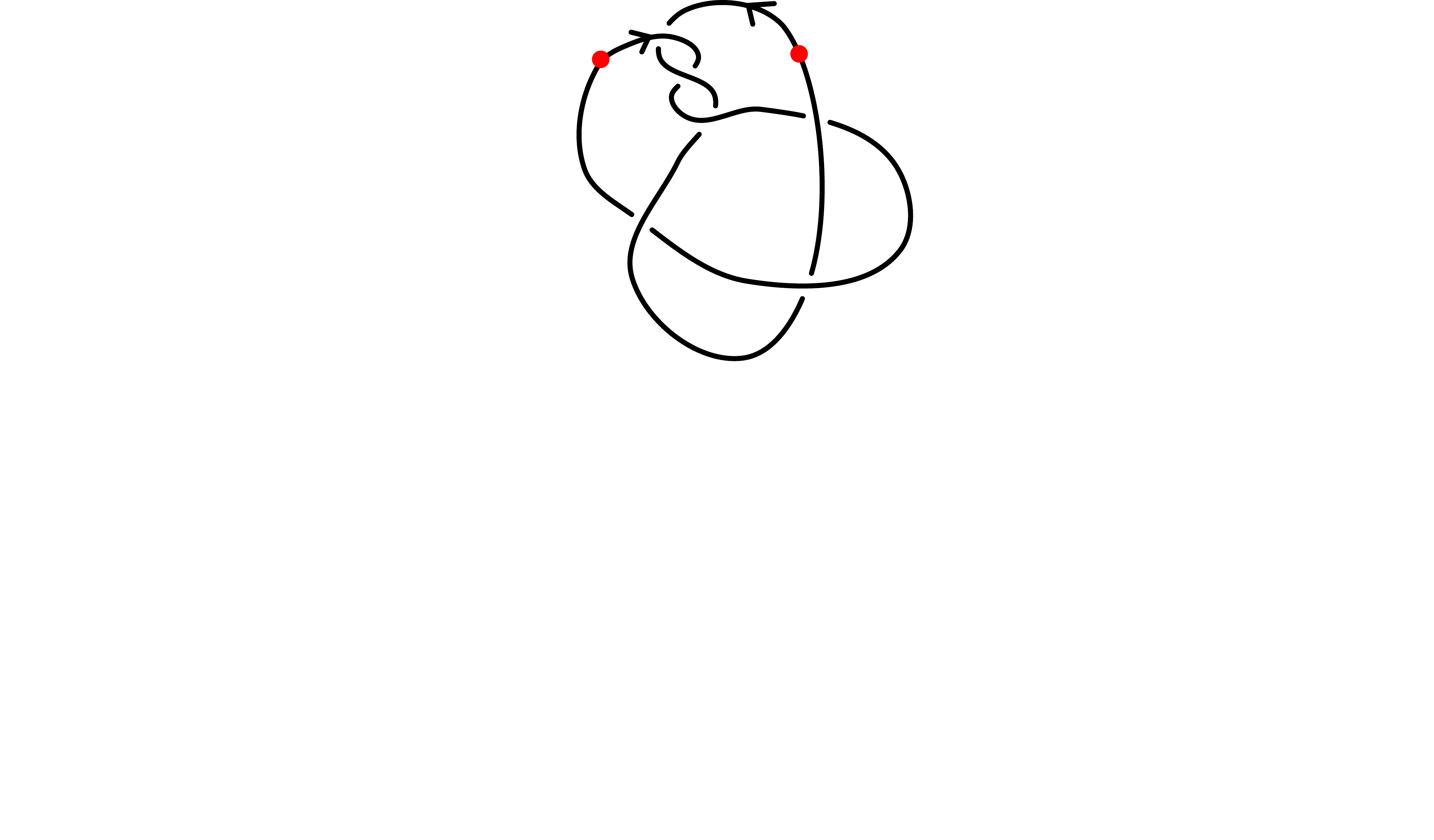
    \caption{First tier of $T_D$.}
    \label{f.ndcex1}
\end{figure}

\textbf{Step 2- Choose new ordering of components and base points for decreasing $|nd(D)|$.} \\

Let $D$ be an oriented link diagram with components $D_1, \ldots, D_{k}$ and base points $\ell_1, \ldots, \ell_{k}$ such that $\ell_i \in D_i$. Given a (positive or negative) crossing $\chi$ in $D$, let $-\chi$ be the crossing with opposite sign (negative or positive). Applying \eqref{e.homfly} expands the HOMFLY-PT polynomial of the link represented by $D = D_{\chi}$ in terms of the HOMFLY-PT polynomial of two link diagrams $D_{-\chi}$ and $D_o$, where $D_{-\chi}$ is  obtained by locally replacing $\chi$ by $\chi$, and $D_o$ is obtained by replacing $\chi$ by two arcs from the oriented resolution of $\chi$. 

Assume $nd(D)$ is not empty, and let $i$ be the smallest index for which $D_i$ has a crossing that is in $nd(D)$. Let $\chi \in nd(D)$ be the crossing in $D_i$ just after the traversal on $D_i$ terminates, so $\chi$ is the first crossing  included in $nd(D)$.  

We choose the following ordering of components and base points of $D_{-\chi}$ and $D_o$ so that the descending complexity $|nd|$ strictly decreases through each application of the skein relation. That is, $|nd(D_{-\chi})| < |nd(D)|$ and $|nd(D_o)| < |nd(D)|$. 

\begin{enumerate}[(1)]
\item The diagram $D_{-\chi}$ inherits base points and ordering of components from $D_\chi$: a component of $D_{-\chi}$ has the same base point and ordering as the corresponding component in $D_\chi$ that is identical to the component outside  of $\chi$. 
\item  There are two cases for $D_o$. Either 
\begin{enumerate}[(a)]
\item  \textbf{replacing by  the oriented resolution at $\chi$ splits a component $D_i$ of $D$ into two components $D_A, D_B$ in $D_o$}  \\
For $j<i$, let each component in $D_o$ corresponding to $D_j$ in $D$ inherit the same base point and order as $D_j$. Let the new component $D_A$ inherit the same base point as $D_i$ in $D$ and choose a base point on the new component $D_B$ just before the crossing $\chi$. Number component $D_A$ by $i$, component $D_B$ by $i+1$, and renumber each component in $D_o$ corresponding to $D_j$ for $j > i$ by shifting each index by $1$, sending $j \mapsto j+1$.  
\item \textbf{replacing by the oriented resolution at $\chi$ merges two components $D_i, D_j$, $i<j$, of $D$ into one component $D_A$ in $D_o$.} \\ 
Let the new component $D_A$ inherit the same base point as $D_i$ and number it $i$. 
For $k'\not= i$ such that $k' < j$, let each component in $D_o$ corresponding to $D_{k'}$ in $D$ inherit the same base point and order as $D_{k'}$. Renumber the components corresponding to $D_{k'}$ for $k' > j$ by sending $k' \mapsto k'-1$. 
\end{enumerate} 
\end{enumerate}

With the ordering and base points as described on the diagram $D_{v'}$ represented by a newly added vertex $v'$ in Step 1, compute $nd(D_{v'}
)$. If $nd(D_{v'}
) \not= 0$, proceed with Step 3. 

\begin{figure}[H]
    \centering
    \def\svgwidth{.9\columnwidth}
    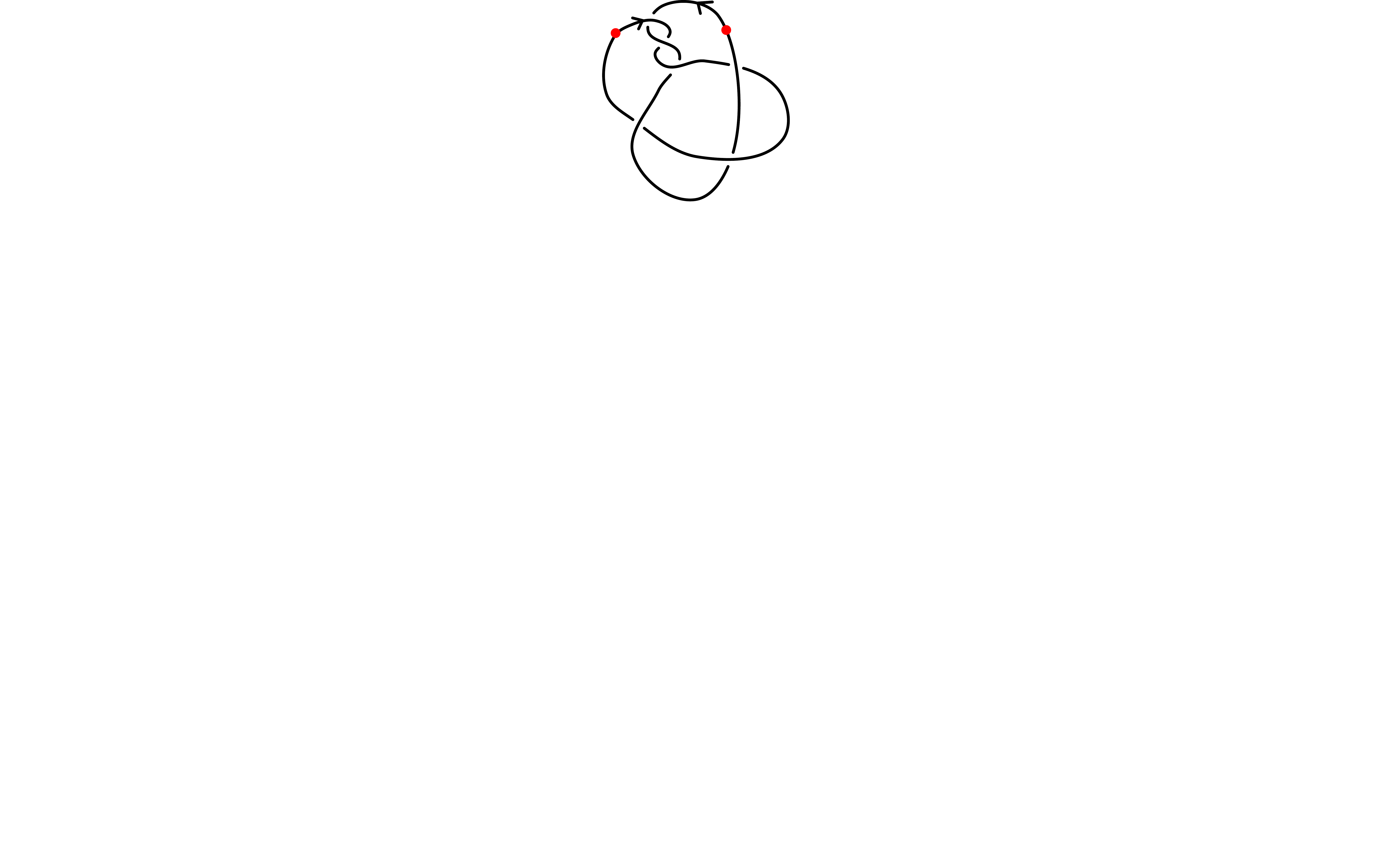
    \caption{Second tier of $T_D$.}
    \label{f.ndcex1}
\end{figure}

\noindent \textbf{Step 3- For each newly added vertex $v'$ with nonzero nondescending complexity, repeat Step 1 and Step 2.}  \\ 

Repeat \textbf{Step 1} and \textbf{Step 2}  with the vertex $v'$ and link diagram $D_{v'} \in \{D_{-\chi}, D_{o} \}$ taking the role of $v$ and $D_v$.  The following lemma proves that the algorithm terminates.

\begin{lem} \label{l.descend}
Let $D$ be a link diagram with a set of base points $\{b_1, \ldots, b_{\ell}\}$ and a fixed ordering of components. Let $\chi$ be the first crossing in $nd(D)$.  With the choice of base points and ordering as described above for $D, D_o$, and $D_{-\chi}$, we have 
\[ |nd(D_{-\chi})| < |nd(D)| \text{ and } |nd(D_o)| < |nd(D)|. \]  
\end{lem}
\begin{proof}
To show $|nd(D_{-\chi})| < |nd(D)|$, we see that $nd(D_{-\chi})$ has at least one fewer crossing compared to $nd(D)$ since we can now extend travel past the crossing $-\chi$ in $nd(D_{-\chi})$. 

For $|nd(D_o)| < |nd(D)|$  we again address two cases depending on whether $D_o$ merges a pair of components in $D$ or splits a component in $D$ into two.  \\ 
\textbf{Case 1} A single component $D_i$ in $D$ is split into two components numbered $i, i+1$ in $D_o$. \\ 
To compute $nd(D_o)$ we traverse $D_o$ in order of the newly chosen base points and compare crossings in the traveled portion $r(D_o)$ with that of $D$. For all components $D_j$ where $1\leq j < i$, $D$ and $D_o$ share the same crossings in the traveled portion of the link diagram. For the component numbered $i$, the traveled portion is now extended past the crossing $\chi$, where it may pick up more crossings in $D_o$, while $n(D_0) = n(D)-1$,  so $nd(D_o) \leq nd(D)-1$ and $nd(D_o) < nd(D)$. \\
\textbf{Case 2} Two components $D_i, D_j$ in $D$ are merged into a single component in $D_o$.  \\ 
Without loss of generality we assume $i < j$. For all components in $D_o$ numbered $k' < i$, the traveled portion remains the same. When traversing the $i$th component of $D_o$, the crossing $\chi$ is removed as we continue to travel on the portion coming from $D_j$, where additional crossings may be included in $r(D_0)$. Thus we have again $nd(D_o) \leq  nd(D)-1$ so $nd(D_o) < nd(D)$.  
\end{proof}

\begin{landscape}
\begin{figure}
\def \svgwidth{.9\columnwidth}
\input{ndcex.pdf_tex}
\caption{\label{f.binarytree} An example of the binary tree $T_D$.}
\end{figure}
\end{landscape}

\subsection{Decomposing into a state sum of unlinks.}
We use the binary tree $T_D$ to give a state sum definition of the HOMFLY-PT polynomial.

\begin{defn}[Jaeger state] \label{d.homflystate}
Let $P$ be a directed path in $T_D$ that starts at the root and ends at a terminal vertex whose associated link diagram has zero non-descending complexity.  

A \textit{Jaeger state} $\sigma_H := \sigma_{H, P}$ on the diagram $D$ associated to $P$ is a function 
\[ \sigma_{H, P}: c(D) \rightarrow \{o, -\chi, id \}   \]
where $\sigma_{H, P}$ chooses the oriented resolution ($o$), crossing-switch ($-\chi$), or to preserve the crossing at each crossing of $D$ ($id$) following the path $P$ as defined below. 

\[ \sigma_{H, P}(\chi) = \begin{cases} &o, \text{if $\chi$ is changed to $o$ by an edge in $T_D$} \\
& -\chi, 
\text{if $\chi$ is changed to $-\chi$ by an edge in $T_D$.} \\ 
&id, 
\text{if the crossing is not changed by any edge in $T_D$.}
\end{cases} \]

In the example illustrated by Figure \ref{f.homflystate}, the base point $\ell_1$ is marked in red as we travel from it through the diagram. Each nonnugatory crossing encountered that is undercrossed by the strand being traversed is enclosed by a dashed circle.
\begin{figure}[H]
\def \svgwidth{\columnwidth}
\begingroup%
  \makeatletter%
  \providecommand\color[2][]{%
    \errmessage{(Inkscape) Color is used for the text in Inkscape, but the package 'color.sty' is not loaded}%
    \renewcommand\color[2][]{}%
  }%
  \providecommand\transparent[1]{%
    \errmessage{(Inkscape) Transparency is used (non-zero) for the text in Inkscape, but the package 'transparent.sty' is not loaded}%
    \renewcommand\transparent[1]{}%
  }%
  \providecommand\rotatebox[2]{#2}%
  \newcommand*\fsize{\dimexpr\f@size pt\relax}%
  \newcommand*\lineheight[1]{\fontsize{\fsize}{#1\fsize}\selectfont}%
  \ifx\svgwidth\undefined%
    \setlength{\unitlength}{2670.15611543bp}%
    \ifx\svgscale\undefined%
      \relax%
    \else%
      \setlength{\unitlength}{\unitlength * \real{\svgscale}}%
    \fi%
  \else%
    \setlength{\unitlength}{\svgwidth}%
  \fi%
  \global\let\svgwidth\undefined%
  \global\let\svgscale\undefined%
  \makeatother%
  \begin{picture}(1,0.18725093)%
    \lineheight{1}%
    \setlength\tabcolsep{0pt}%
    \put(0,0){\includegraphics[width=\unitlength,page=1]{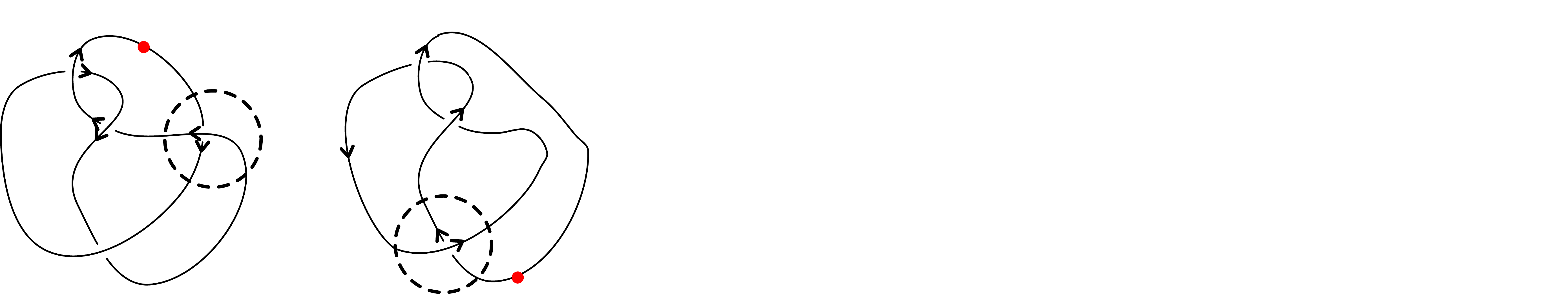}}%
    \put(0.10240997,0.15940629){\color[rgb]{1,0,0}\makebox(0,0)[lt]{\lineheight{1.25}\smash{\begin{tabular}[t]{l}$\ell_1$\end{tabular}}}}%
    \put(0,0){\includegraphics[width=\unitlength,page=2]{bicoloredHOMFLY.pdf}}%
  \end{picture}%
\endgroup%

\caption{\label{f.homflystate}  A Jaeger state $\sigma_{H, P}$ on the figure 8 knot.}
\end{figure} 
\end{defn}

We will let $S_{\sigma_{H, P}}$ denote the result of applying $\sigma_{H, P}$ to $D$, which locally replaces a crossing by the choice of the Jaeger state $\sigma_{H, P}$. This results in a descending link diagram of $|\sigma_{H, P}|$ components.

\begin{rem}[Number of states] We remark that the number of states $\{ \sigma_{H, P} \}$ as defined by Definition \ref{d.homflystate} on a diagram is the same as the number of directed paths from the root to a terminal vertex in the binary tree $T_D$, and therefore it  is generally less than $2^n$ , where $n$ is the number of crossings of $D$. 
\end{rem}
\begin{defn}[Coefficients]
    Now let $\sigma_{H, P}(a, z)$ be the product of coefficients in variables $a, z$ along the path defining $\sigma_{H, P}$, and let $D^{\sigma_{H, P}}$ represent the diagram of $|\sigma_{H, P}|$ component unlinks at the terminal vertex of the path $P$ defining $\sigma_{H, P}$. 
\end{defn}
 Then we obtain the following formula:
\begin{align} 
P_L(a, z) &= \sum_{\substack{\sigma_{H, P} \text{ a Jaeger state} \\ \text{on the diagram $D$}} } \sigma_{H, P}(a, z) P_{D^{\sigma_{H,P}}}(a, z). \notag \\ 
\intertext{We have:} 
P_{D^{\sigma_{H, P}}}(a, z)  &= P_{U^{|\sigma_{H, P}|}}(a, z), \intertext{where $U^{|\sigma_{H, P}|}$ is the collection of state circles obtained by applying $\sigma_{H, P}$ to every crossing in $D$, and $|\sigma_{H, P}|$ is the number of such circles.  Therefore} 
P_L(a, z) &= \sum_{\substack{\sigma_{H, P} \text{ a Jaeger state} \\ \text{on the diagram $D$}} } \sigma_{H, P}(a, z) \left(\frac{a-a^{-1}}{z} \right)^{|\sigma_{H, P}|-1}. \label{e.homflyeqjstate}
\end{align}
\begin{rem}[Relation to Jaeger's coefficients]
We remark that the coefficient polynomial $\sigma_{H, P}(a, z)$ counts the same quantity as in the Jaeger state sum for the HOMFLY-PT polynomial. 
\end{rem}

\subsection{ Strategy for the topological model of the  HOMFLY-PT polynomial}\label{strategy}
For the topological model, we need two geometric tools for interpreting the terms from formula \ref{e.homflyeqjstate}.
We will do this in two main steps, as follows:
\begin{enumerate}[1)]

\item  {\bf Construction of the Topological model (Section \ref{ss.khp})} We provide a model for the two-variable polynomial part via intersection of homology classes.
\item  {\bf Geometric interpretation of  Jaeger's coefficients $\sigma_{H, P}(a, z)$ (Section \ref{ss.Jcoeff})} We will give a geometric interpretation of  Jaeger's coefficients in terms of linking numbers of certain twisted curves on $\Sigma$.  
\item  {\bf Proof of the Main Theorem (Section \ref{ss.tmhomfly})}
\end{enumerate}
For Steps 1) and 2), we will show the following.

\

{\bf 1)  Construction of the homology classes}

Next in Sections \ref{ss.khp} and \ref{ss.tophomflypt}, we give a geometric interpretation to the term $\left(\frac{a-a^{-1}}{z} \right)^{|\sigma_{H, P}|-1}$ as the HOMFLY-PT specialization $\alpha^{\sigma}_H$ of the intersection pairing defined in Section \ref{ss.specialization} between two homology classes $\mathscr{L}^H(\sigma^K_{H, P})$, $\mathscr{F}^H(\sigma^K_{H, P})$, designed to count $|\sigma_{H, P}|-1$ in the degree of the HOMFLY-PT polynomal of $|\sigma_{H, P}|$-component unlinks.   We do this in two steps:

\begin{itemize}
    \item we define an associated Kauffman state $\sigma^K_{H, P}$ to our Jaeger state $\sigma_{H, P}$ (Subsection \ref{ss.khp})
    \item we construct the classes via the geometric support of the Kauffman state (Subsection \ref{ss.tophomflypt}).
\end{itemize}

\

{\bf 2) Understanding Jaeger's coefficients $\sigma_{H, P}(a, z)$}

Motivated by the geometric interpretation (Lemma \ref{l.intslope}) of the coefficient polynomial to each state of Kauffman state sum definition of the Jones polynomial, we introduce a family of twisted curves on our surface in Definition \ref{dtwist}. Then, we give a geometric characterization of the exponents of variables $a, z$ in $\sigma_{H, P}(a, z)$ in Lemma \ref{e.giHOMFLY} in terms of oriented intersections of these curves.


\subsection{Associating a Kauffman state $\sigma^K_{H, P}$ to a Jaeger state $\sigma_{H, P}$} \label{ss.khp}

\

We will refer to a Jaeger state for the HOMFLY-PT state sum as an $H$-state for the rest of this section to simplify notation. 

The descending diagram of unlinks at a terminal node of the binary tree $T_D$ may still have crossings, though the crossings are unimportant in terms of the evaluation of the HOMFLY-PT polynomial. We replace them by the set of state circles from a Kauffman state on the diagram $D$ via the following lemma. 

\begin{lem}[Renormalized Kauffman state with the same number of circles]\label{renkaf}

\

Given a link diagram $D$ with $m$ components, there exists a Kauffman state $\sigma$ on $D$ with $|\sigma|=m$ state circles in $S_{\sigma}$. 
\end{lem}

\begin{proof}
We first consider the case of a $2$-component link, and show that we can find a Kauffman state on $D$ whose application results in a diagram with two split components.  

A diagram of a link with two components has an even number of crossings $n_c$, where the strands of the crossing belong to different components. The number $n_c$ is the same as twice the linking number of the two components. Denote the set of all such crossings by $X_{n_c}$. Choosing the oriented resolution for all crossings in $X_{n_c}$ results in two components, since choosing the oriented resolution at a crossing merges the pair of link components, and then choosing the oriented resolution again for a self-crossing of a knot splits the component back into two again. See Figure \ref{f.splitmerge} for an illustration.

\begin{figure}[H]
    \centering
    \def \svgwidth{.5\columnwidth}
\begingroup%
  \makeatletter%
  \providecommand\color[2][]{%
    \errmessage{(Inkscape) Color is used for the text in Inkscape, but the package 'color.sty' is not loaded}%
    \renewcommand\color[2][]{}%
  }%
  \providecommand\transparent[1]{%
    \errmessage{(Inkscape) Transparency is used (non-zero) for the text in Inkscape, but the package 'transparent.sty' is not loaded}%
    \renewcommand\transparent[1]{}%
  }%
  \providecommand\rotatebox[2]{#2}%
  \newcommand*\fsize{\dimexpr\f@size pt\relax}%
  \newcommand*\lineheight[1]{\fontsize{\fsize}{#1\fsize}\selectfont}%
  \ifx\svgwidth\undefined%
    \setlength{\unitlength}{1484.80622539bp}%
    \ifx\svgscale\undefined%
      \relax%
    \else%
      \setlength{\unitlength}{\unitlength * \real{\svgscale}}%
    \fi%
  \else%
    \setlength{\unitlength}{\svgwidth}%
  \fi%
  \global\let\svgwidth\undefined%
  \global\let\svgscale\undefined%
  \makeatother%
  \begin{picture}(1,0.51729756)%
    \lineheight{1}%
    \setlength\tabcolsep{0pt}%
    \put(0,0){\includegraphics[width=\unitlength,page=1]{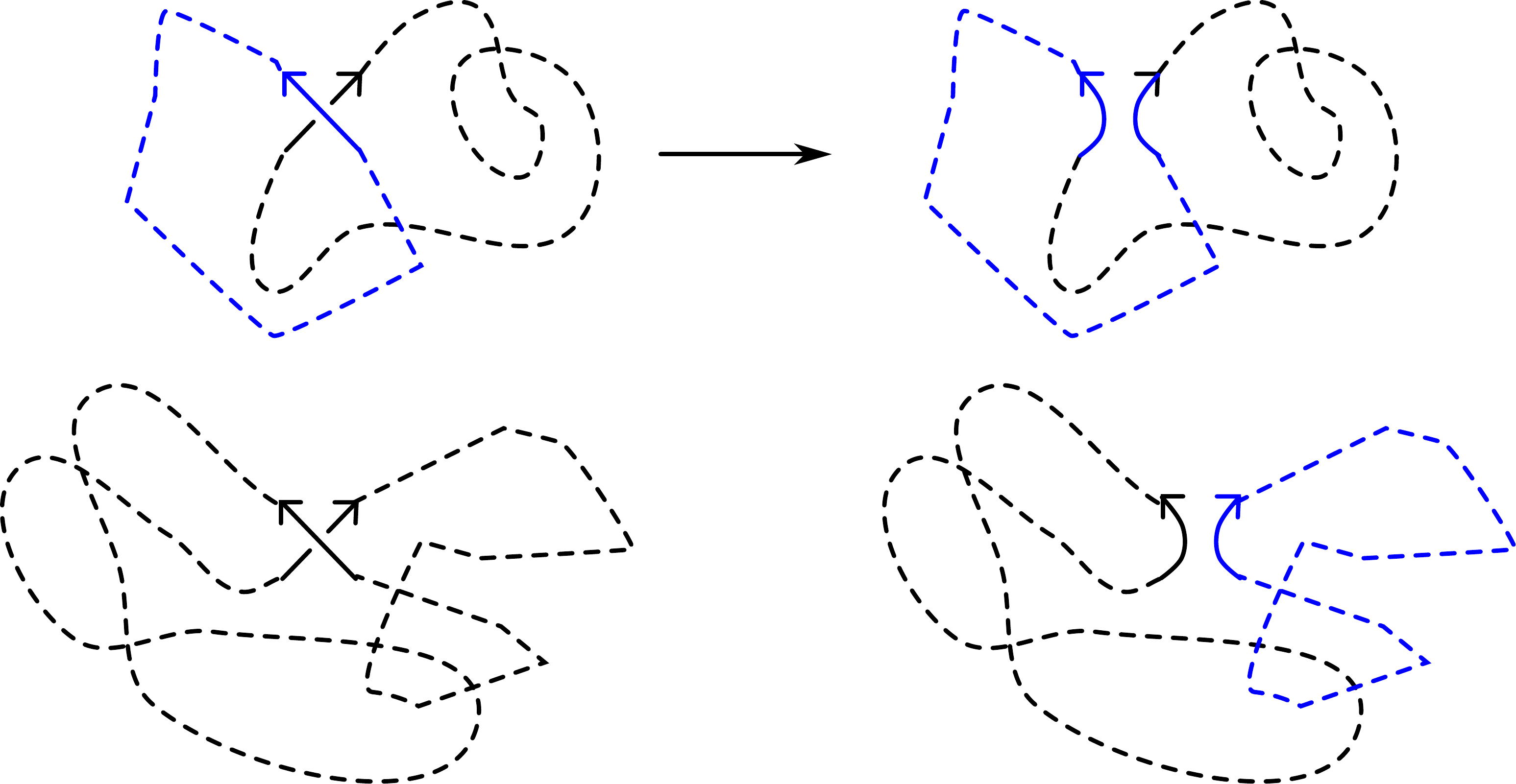}}%
    \put(0.43663124,0.43246582){\color[rgb]{0,0,0}\makebox(0,0)[lt]{\lineheight{1.25}\smash{\begin{tabular}[t]{l}Merge\end{tabular}}}}%
    \put(0,0){\includegraphics[width=\unitlength,page=2]{msplit.pdf}}%
    \put(0.43971684,0.14004775){\color[rgb]{0,0,0}\makebox(0,0)[lt]{\lineheight{1.25}\smash{\begin{tabular}[t]{l}Split\end{tabular}}}}%
  \end{picture}%
\endgroup%

    \caption{Merging and splitting.}
    \label{f.splitmerge}
\end{figure}

    This argument extends to the case of a diagram of a link with $m\geq 2$ components by inducting on $m$, since we can remove linking between any pair of components by applying the oriented resolution as defined above for all crossings between two different components, repeatedly as necessary, to obtain a split diagram $D_0 \sqcup D_1$, where $D_0$ is an unlinked component and $D_1$ is a link with $m-1$ components, then apply the inductive hypothesis to $D_1$.  

Given a link diagram $D'= D_1 \cup D_2 \cup \cdots \cup D_m$ from $D$ of $m$ split components from the above argument, we need to define the Kauffman state $\sigma$ on the remaining crossings that results in a single state circle. The existence of this Kauffman state is known, for example, in \cite{DFKLS}, but we can give a self-contained argument here. 

Pick a crossing in $D_1$ and name it $c_1$. Define $\sigma$ to choose the non-oriented resolution on $c_1$. This will necessarily preserve the single component of the diagram. Replace $c_1$ by the oriented resolution to obtain a new knot diagram $D_1^{s_1}$. Orient $D_1^{s_1}$, pick another crossing $c_2$, and define $\sigma$ to choose the non-oriented resolution with respect to this new orientation on $D_1^{s_1}$. Repeat this process for $D_1^{s_i}$ with $i = 2, \ldots, n(D_1)$ obtained by replacing the crossing $c_{i}$ by the nonoriented resolution in $D_1^{s_{i-1}}$, where $n(D_1)$ is the number of crossings in $D_1$.    We eventually exhaust all the crossings in $D_1$. Move on to the next component $D^2$,  then $D^3$, $D^4$, \ldots, $D^m$, after that, and record the choice of the resolution for $\sigma$. We have extended $\sigma$ to all the crossings of $D$.

\end{proof}

\begin{figure}[H]
\def \svgwidth{.5\columnwidth}
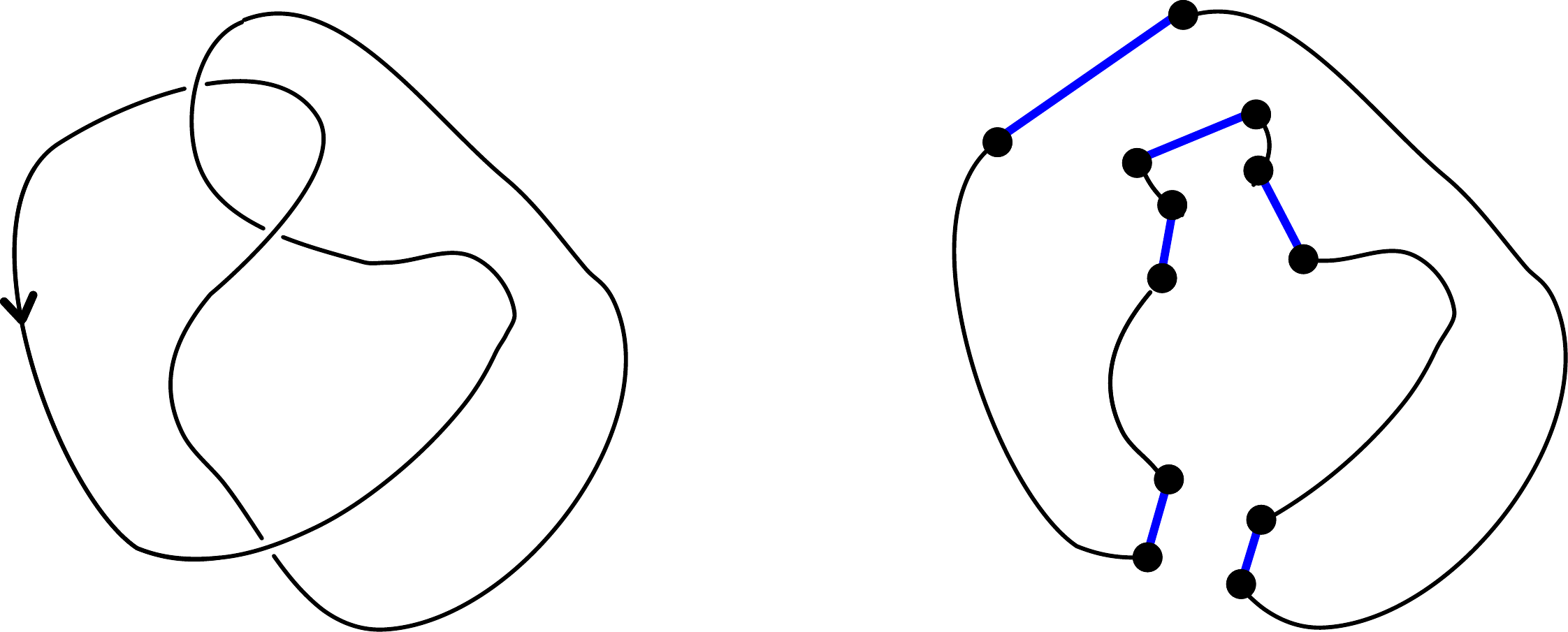
\caption{A renormalized Kauffman state $\sigma^K_{H, P}$.}
\end{figure}
\begin{rem}
We remark that a Kauffman state that preserves the number of components of a Jaeger state in Lemma \ref{renkaf} is not unique. 
\end{rem}
Returning to the context of the HOMFLY-PT polynomial and an $H$-state $\sigma_{H, P}$ on a link diagram $D$, we denote the Kauffman state associated to $\sigma_{H, P}$ via Lemma \ref{renkaf} by $\sigma^K_{H, P}$. 

\begin{figure}[H]
    \centering
    \def \svgwidth{.7\columnwidth}
    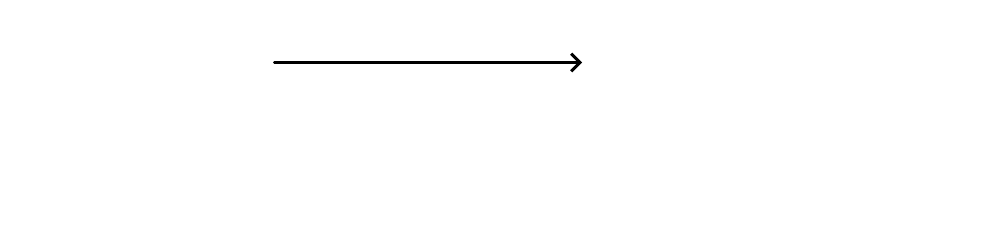
    \caption{Kauffman state associated to a Jaeger state}
    \label{fig:enter-label}
\end{figure}

We may now rewrite the state sum \eqref{e.homflyeqjstate} for the HOMFLY-PT polynomial as follows, replacing $|\sigma_{H, P}|$ by $|\sigma^K_{H, P}|$ using Lemma \ref{renkaf} : 
\begin{align}
P_L(a, z) &= \sum_{\substack{\sigma_{H, P} \text{ a Jaeger state} \\ \text{on the diagram $D$}} } \sigma_{H, P}(a, z) \left(\frac{a-a^{-1}}{z} \right)^{|\sigma_{H, P}|-1} \notag \\  
&= \sum_{\substack{\sigma_{H, P} \text{ a Jaeger state} \\  \text{on the diagram $D$}} } \sigma_{H, P}(a, z) \left(\frac{a-a^{-1}}{z} \right)^{|\sigma^K_{H, P}|-1} \label{e.hksum}
\end{align}

\subsection{Homology classes for the HOMFLY-PT polynomial} \label{ss.tophomflypt}
In this section we construct an intersection model for the HOMFLY-PT polynomial. 

We start from the evaluation of the renormalized Kauffman bracket $$P_{U^{|\sigma_{H, P}|}}(a, z) = \left(\frac{a-a^{-1}}{z} \right)^{|\sigma^K_{H, P}|-1}$$ on an unlink with $|\sigma_{H, P}|$ components, using the state sum formulation \eqref{e.hksum} of the HOMFLY-PT polynomial developed in the previous section. Then, our strategy is to interpret this as a state sum of Lagrangian intersections on our punctured surface. 

In order to achieve this, the idea is to use the above description given by the renormalized Kauffman bracket (as in Lemma \ref{renkaf}) and use its geometric support in order to translate it homologically as graded intersection between homology classes. 

For this we need to understand geometrically the meaning of the renormalization taking the number of closed components minus 1, $|\sigma^K_{H, P}|-1$, rather than $|\sigma^K_{H, P}|$, that appears in the evaluation.

As in the model for the Jones polynomial,  we color both the arcs in an unlink from the oriented/non-oriented resolution chosen by $\sigma^K_{H, P}$ at a crossing blue.  See Figure \ref{f.statecolor}. As in Remark \ref{r.bicolored}, for each state there are twice as many blue arcs as the number of crossings, and twice as many black arcs as the number of crossings. 

\begin{defn}[Cutting point]
    Mark one of our base points $b_{2n}$ and call it the ``cutting base point'' for our diagram.
\end{defn}

This is seen by the fact that once we have a set of state circles, the skein relation of the HOMFLY-PT polynomial evaluates all of them but one, instead of evaluating all components. We interpret this topologically by choosing a base point and cutting the diagram in this point. Then, we proceed with the following homology classes. 
\begin{defn}[Collections of arcs and ovals]
Suppose that we have a fixed state $\sigma^K_{H, P}$. We will follow a similar procedure as the one for the Jones polynomial, the only difference is that we will keep track of our cutting base point. 

More precisely, let us keep the set of blue arcs, but we replace the black arcs by green ovals, constructed in a tubular neighborhood of these arcs, as in Figure \ref{f.arcsovalsoH}. The important change that we make for this case is that we consider the oval that contains the cutting base point $b_{2n}$ and we deform it to a smaller oval, such that it contains just one puncture, instead of two, as in the picture. 

\begin{figure}[H]
\def \svgwidth{.7\columnwidth}
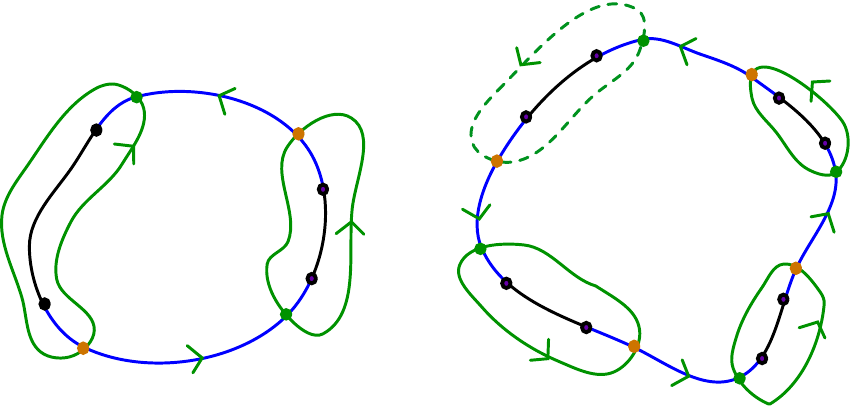
\caption{\label{f.arcsovalsoH} Homology classes for HOMFLY-PT polynomial}
\end{figure}

The effect is that this oval does not intersect the following arc, once we follow the diagram of the link.
Now, we consider the following sets:
\begin{enumerate}
\item[$\bullet$] $F^H(\sigma^K_{H, P})$: given by the blue arcs associated to the state
\item[$\bullet$] $L^H(\sigma^K_{H, P})$: the collection of green ovals constructed from our coloring.
\end{enumerate}
\end{defn}
\begin{defn}[Orientation] We fix an orientation for the arcs and ovals. All the ovals have the counter-clockwise orientation, and for the arcs, we orient them locally as in the case for the Jones polynomial. Each pair that comes from a resolution has the orientation as in Figure \ref{f.arcsovalsoH}.  We orient our 
surface such that the base points all have positive orientation.
\end{defn}

\paragraph{\textbf{Homology classes associated to a state}}
Following the philosophy that we saw for the case of the Jones polynomial, once we have two collections of arcs and ovals associated to a fixed state, we can prescribe two homology classes, using the {\em geometric support}, given by a {\em collection of arcs or ovals in the punctured disk}. 

\begin{defn}
    
With this procedure the collection of curves $F^H(\sigma^K_{H, P})$ leads to a homology class in the  
covering, which we denote as $$\mathscr F^H(\sigma^K_{H, P})\in \mathscr H^{lf}_{2n}.$$

Dually the collection of ovals $L^H(\sigma^K_{H, P})$ leads to a homology class in the covering, which we denote as $$\mathscr L^H(\sigma^K_{H, P})\in \mathscr H_{2n}.$$
\end{defn}

\subsection{Intersection of classes recovers the evaluation of state circles} \label{ss.intersectionhomfly}

In this section, we show that through the topological setup from above we can obtain the evaluation of circles given by the normalized Kauffman bracket for the HOMFLY-PT polynomial. 

The new aspect of this topological viewpoint is that in this case, for a state $\sigma_{H, P}$, the intersection between the homology classes recovers the polynomial $\left( \frac{a-a^{-1}}{z} \right)$ raised to the power given by the {\bf number of circles minus one}.  

\begin{lem}[Recovering the normalized Kauffmann bracket in a general form]\label{intgeneralH}
Let $\sigma_{H, P}$ be a Jaeger state with associated Kauffman state $\sigma^K_{H, P}$. Fix an indeterminate $Q$ and the associated specialization $\alpha^{\sigma_{H, P}}_Q$. Then, the intersection of  $\mathscr 
F^H(\sigma^K_{H, P})$ and $\mathscr 
L^H(\sigma^K_{H, P})$ specialized through $\alpha^{\sigma_{H, P}}_Q$ gives a fixed polynomial in $Q$ to the  number of circles minus one.
\begin{equation}
    \ll \mathscr F^H(\sigma^K_{H, P}), \mathscr L^H(\sigma^K_{H, P}) \gg_{\alpha^{\sigma_{H, P}}_Q}=(1-Q)^{|\sigma_{H, P}|-1}.
\end{equation}    
\end{lem}
\begin{proof}
We will use the formula for computing the intersection form. 

\
The proof of this intersection formula follows a similar strategy as the one presented in Lemmas \ref{intgeneral}. The new part is to explain why the geometry of the intersections of the homology classes gives us the number of circles minus one, even as we use the same local system and specialization of coefficients. 

We explain the main steps below.

\

\noindent {\bf Step 1- Reducing the problem to the case of a state circle}

\

Consider the geometric supports of the homology classes $\mathscr F^H(\sigma^K_{H, P})$ and $\mathscr L^H(\sigma^K_{H, P})$: they are given by the collection of arcs $F(\sigma^K_{H, P})$ and the collection of ovals $ L(\sigma^K_{H, P})$. We remark that these two collections of arcs and ovals on the surface contain our state circles associated to $\sigma^K_{H, P} $. The intersection between $\mathscr F(\sigma^K_{H, P})$ and $\mathscr L(\sigma^K_{H, P})$ is parametrized by: the set of intersection points between their geometric supports, and then graded by the local system. 

Let us look at the set of intersection points. Such a point is encoded by a $(2n)$-tuple of points $z=(z_1,...,z_{2n})$ on the surface $\Sigma'$ such that:
\begin{enumerate}
    \item each arc from $F(\sigma^K_{H, P})$ contains exactly one component from the set $\{z_1,...,z_{2n}\}$
    \item each oval from $L(\sigma^K_{H, P})$ contains exactly one component from the set $\{z_1,...,z_{2n}\}$.
\end{enumerate}

Note that the above requirement adds conditions on the intersection points between arcs and ovals that bound the same state circle. So, in order to have an intersection point $z=(z_1,...,z_{2n})$ we should choose for each state $C_i$ (which is bi-colored into let's say $2m_i$ components) $m_i$ points on the punctured surface $(z^i_1,...,z^i_{m_i})$ with the following requirement:
\begin{enumerate}
    \item each arc from $F(\sigma^K_{H, P})$ bounding $C_i$ contains exactly one component from the set $\{z^i_1,...,z^i_{m_i}\}$
    \item each oval from $L(\sigma^K_{H, P})$ bounding $C_i$ contains exactly one component from the set $\{z^i_1,...,z^i_{m_i}\}$.
\end{enumerate}

Then, our intersection point $z$ will be given by the collection of chosen points $z^i=(z^i_1,...,z^i_{2m_i})$ associated to each state circle: $$\{z^i_1,...,z^i_{2m_i}\mid 1\leq i \leq | \sigma_{H, P}| \  \}.$$
For our intersection pairing, by Proposition \ref{formint} we have to compute two parts:
\begin{enumerate}[i)]
\item  the loop $l_{z}$ (Definition \ref{d.associatedloop}) and its evaluation through the local system
\item  the sign $\alpha_z$.
\end{enumerate}

The sign will be the product of signs associated to each component from the set of state circles. 

For the next steps, we remark that given the property that the intersection point has its components separated into components associated to the state circles, it means that the loop $l_{z}$ and then its evaluation through the local system are both obtained from separate components associated to each state circle separately.
For the result $\Phi(l_{z})$ we can see the contribution of each loop associated to a state circle $l_{z^i}$ and consider the product of all of these when we evaluate the monodromies around the punctures through the local system and then the specialization. 

\

\noindent {\bf Step 2- Intersection points associated to one state circle.}

\

The main idea from the previous step is that our intersection can be obtained using individual computations for each state circle, by looking at its contribution to the monodromy of the local system. 

If we have a state circle that does not contain the cutting base point, then the geometric support of the classes $F(\sigma^K_{H, P})$ and $F^H(\sigma^K_{H, P})$ based on it is the same. Similarly, the geometric support of the classes $L(\sigma^K_{H, P})$ and $L^H(\sigma^K_{H, P})$ based on it the same. So the contribution of this state circle gets evaluated to \begin{equation}
    1+(-1) \cdot Q=1-Q.
\end{equation}
This follows by an analogous computation as the one presented in Steps 2, 3, and 4 from the proof of Lemma \ref{intgeneral}. 

\

\noindent {\bf Step 3- Intersection points associated to the circle containing the cutting point.}

\

It remains to compute the contribution of the state circle that contains the base cutting point. The subtlety is that in this case the geometric supports for $\mathscr L(\sigma^K_{H, P})$ and $\mathscr L^{H}(\sigma^K_{H, P})$ are different (even if the support of $\mathscr F(\sigma^K_{H, P})$ and $\mathscr F^{H}(\sigma^K_{H, P})$ is the same). 
\begin{lem}[Intersection points along the component] Let us fix the above state circle $C_i$. Then, there is precisely one intersection point associated to this component: $B^i=(b^i_1,...,b^i_{m_i})$.
\end{lem}

\begin{proof}
First, we look at the components of the base point $\bf d$ that belong to this state circle and denote them as: 
$(b^i_1,...,b^i_{m_i}).$
Let us suppose that our cutting point $b_m$ is $b^i_{m_i}$ from the above notation. We consider:
\begin{itemize}
    \item $F^i(\sigma^K_{H, P})$  the set of arcs from $F^H(\sigma^K_{H, P})$ bounding $C_i$;
 \item $L^i(\sigma^K_{H, P})$  the set of ovals from $L^H(\sigma^K_{H, P})$ bounding $C_i$.
\end{itemize}

\

\noindent {\bf Requirement for intersection points}\\
An intersection point associated to a state circle $C_i$ is given by a choice of $2m_i$ points on the punctured surface $(z^i_1,...,z^i_{m_i})$ such that:
\begin{enumerate}
    \item each arc from $F^i(\sigma^K_{H, P})$ contains exactly one component $\{z^i_1,...,z^i_{m_i}\}$
    \item each oval from $L^i(\sigma^K_{H, P})$ contains one component from $\{z^i_1,...,z^i_{m_i}\}$.
\end{enumerate}
We have the set of base points, that belong to our collections of arcs and ovals, and let us denote them by:
$$(b^i_1,...,b^i_{m_i}).$$ 
When we following the open contour of our state circle, these points are chosen at intervals at Step $2$, so $$B^i=(b^i_1,...,b^i_{m_i})$$ provides a well defined intersection point that satisfies the above {\bf Requirement for intersection points}. 

Now, we investigate if it is possible to have other intersection points satisfying this requirement.
We look at the arc containing $b^i_1$. This arc intersects {\bf just one oval}: the one  containing $b^i_1$. This is thanks to our choice of making the oval containing the base point $b_m$ smaller, so this does not intersect this arc. So we are obliged to choose for this arc the point $b^i_1$. Now, we look at the next arc. We are obliged to choose on it the point $b^i_2$. Otherwise, we have in our collection of points two components on the same oval, which is not permitted. Inductively, we have to choose the last component to be $b^i_{m_i}$.   
Overall, we see that we have exactly one intersection point:  $B^i=(b^i_1,...,b^i_{m_i})$.

\end{proof}

{\bf Step 4- Computing the contribution of the loop associated to the cut state circle.}

\

We compute the contribution of the intersection point $B^i$ once we evaluate the local system and specialize it through $\alpha^{\sigma_{H, P}}_Q$. For this, we have to see which loop is associated to this intersection point.

If we have chosen the components of $B^i$, then the loop associated to it is the constant loop given by the collection of base points $$b^i_1,\ldots,b^i_{m_i}.$$

We remark that this loop does not wind around any puncture, so the contribution of the intersection point $B^i$ to the local system gets evaluated to $1$.

\

\noindent {\bf Step 5- Sign contribution of intersection point from the cut state circle}

\

For the point $B^i$ we have all intersections between arcs and ovals with positive local orientations and also there is no permutation induced at the level of the set of arcs and ovals. So this point has a sign contribution of $$+1.$$

\

\noindent {\bf Step 6- Contribution of intersection point associated to the cut circle to the intersection form}

\

Overall we see that $B^i$ contributes to the local system evaluation by $1$ and carries the sign $1$.
    
This shows that the component $C_i$ contributes to the intersection by $1$. 
In particular, the {\bf cut circle} does not contribute to the count of the local system, and thus it does not add any contribution to the intersection.  
For all state circles, and we have $| \sigma^K_{H, P}| -1$ of them, we obtain a contribution of $(1-Q)$. So, our intersection is indeed: 
\begin{equation}
    \ll \mathscr F^H(\sigma^K_{H, P}), \mathscr L^H(\sigma^K_{H, P}) \gg_{\alpha^{\sigma_{H, P}}_Q}=(1-Q)^{|\sigma_{H, P}|-1}.
\end{equation}    
which concludes the proof of the Lemma.

\end{proof}

Now we are ready to show that the above topological viewpoint, provided by the set of graded intersections of homology classes which are associated to the cut state circles give the HOMFLY-PT polynomial.

For the case of the HOMFLY-PT polynomial, we have introduced another specialization of coefficients, following Definition \ref{specH}, as below:
 \begin{equation}
      \begin{aligned}
&\alpha^{\sigma_{H, P}}_{H}:  \Z[x_1^{\pm 1},...,x_{8n}^{\pm 1}]  \rightarrow \Z[a^{\pm1},z^{\pm1}](1-\frac{a-a^{-1}}{z})
\end{aligned}
 \end{equation}
This has the associated specialized intersection pairing given by:
\begin{equation}
\ll ~,~ \gg_{\alpha^{\sigma_{H, P}}_{H}}: \mathscr H^{lf}_{2n}\mid_{\alpha^{\sigma_{H, P}}_H} \otimes \mathscr H_{2n}\mid_{\alpha^{\sigma_{H, P}}_H} \rightarrow\Z[a^{\pm1},z^{\pm1}](1-\frac{a-a^{-1}}{z}).
\end{equation}

\begin{lem}[Recovering the HOMFLY-PT polynomial for un-links via arcs and ovals]\label{intJ}
Let us fix a Jaeger state $\sigma_{H, P}$ with the associated Kauffman state $\sigma^K_{H, P}$. Then, the intersection between the classes $\mathscr 
F^H(\sigma^K_{H, P})$ and $\mathscr 
L^H(\sigma^K_{H, P})$ specialized through $\alpha^{\sigma_{H, P}}_{H}$ gives precisely the HOMFLY-PT renormalized Kauffman bracket evaluated on the state circles associated to $\sigma^K_{H, P}$:
\begin{equation}
    \ll \mathscr F^H(\sigma^K_{H, P}),\mathscr L^H(\sigma^K_{H, P}) \gg_{\alpha^{\sigma_{H, P}}_{H}}=\left(\frac{a-a^{-1}}{z}\right)^{|\sigma_{H, P}|-1}.
\end{equation}    
\end{lem}
\begin{proof}
Following Lemma \ref{intgeneral}, we have that:
\begin{equation}
    \ll \mathscr F^H(\sigma^K_{H, P}),\mathscr L^H(\sigma^K_{H, P}) \gg_{\alpha^{\sigma_{H, P}}_{Q}}=(1-Q)^{|\sigma_{H, P}|-1}.
\end{equation}
In our case we have $Q=Q_H=1-\frac{a-a^{-1}}{z}$ and so this leads to our desired formula for the intersection:
\begin{equation}
    \ll \mathscr F^H(\sigma^K_{H, P}),\mathscr L^H(\sigma^K_{H, P}) \gg_{\alpha^{\sigma_{H, P}}_{H}}=\left(\frac{a-a^{-1}}{z}\right)^{|\sigma_{H, P}|-1}.
\end{equation}
\end{proof}

\subsection{A geometrical description of Jaeger's coefficients via linking numbers of twisted curves} \label{ss.Jcoeff}

For a Jaeger state $\sigma_{H, P}$ we also define the monomial coefficients $i^a(\sigma_{H, P})$, $i^z(\sigma_{H, P})$ of $a$ and $z$, respectively, summing over all local contributions. These are obtained via the skein relation \eqref{e.homfly}. Recall $c(D)$ is the set of crossings of $D$, $\chi$ is a crossing of $D$ and $-\chi$ is the crossing of opposite sign to $\chi$. 
\begin{itemize}
\item if $\chi$ is a positive crossing, then 
\[ (i_\chi^a(\sigma_{H, P}), i_\chi^z(\sigma_{H, P})) :=  \begin{cases}  (-1, 1) &\text{ if $\sigma_{H, P}(\chi)=o$} \\
(-2, 0) &\text{ if $\sigma_{H, P}(\chi) = -\chi$} \\ 
(0, 0) & \text{ if $\sigma_{H, P}(\chi) = \chi$} \\ 
\end{cases}  \]
\item if $\chi$ is a negative crossing, then 
\[ (i_\chi^a(\sigma_{H, P}), i_\chi^z(\sigma_{H, P})) := \begin{cases} (1, 1) &\text{ if $\sigma_{H, P}(\chi) = o$} \\
(2, 0) &\text{if $\sigma_{H, P}(\chi) = -\chi$ } \\ 
(0, 0) & \text{ if $\sigma_{H, P}(\chi) = \chi$} \\ 
\end{cases}  \]
\item the \textit{sign} of a Jaeger state is given by: 
\[ \sgn_{\chi}(\sigma_{H, P}) = \begin{cases} 
-1 & \text{if $\chi$ is a negative crossing and $\sigma_{H, P}(\chi) = o$} \\ 
1 & \text{otherwise.} \end{cases} \] 
\end{itemize} 
Now let 
\[ i^a(\sigma_{H, P}):= \sum_{\chi \in c(D)} i_\chi^a(\sigma_{H, P}) \qquad \text{and} \qquad i^z(\sigma_{H, P}):= \sum_{\chi \in c(D)} i_\chi^z(\sigma_{H, P}), \] with
\[ \sgn(\sigma_{H, P}) = \prod_{\chi \in c(D)} \sgn_{\chi}(\sigma_{H, P}). \]

Recall the twisted $\alpha$ curves $\alpha_{o}$ and regular (untwisted) $\alpha$ curves $\alpha_{\hat{o}}$ as in Definition \ref{d.twistedcurvesJ} for recovering a part of the degree of the monomial in $q$ multiplying the Kauffman bracket of state circles in state sum formulation \eqref{e.jfinalss} of the Jones polynomial. We define analogous curves  $\alpha^H_{o}$ and $\alpha^H_{\hat{o}}$ for the HOMFLY-PT polynomial to recover the degree of $a, z$ in the coefficient $\sigma_{H, P}(a, z)$ to each Jaeger state as their oriented intersections. 

\begin{defn}[Twisting curves for Jaeger's coefficients]
\label{dtwist}
The curves $\alpha^H_{o}(+)$ and $\alpha^H_{o}(-)$ locally replace a positive crossing and a negative crossing, respectively, by curves that each winds around the meridian twice in the following figure. 
\begin{figure}[H]
\def\svgwidth{.5\columnwidth}
\begingroup%
  \makeatletter%
  \providecommand\color[2][]{%
    \errmessage{(Inkscape) Color is used for the text in Inkscape, but the package 'color.sty' is not loaded}%
    \renewcommand\color[2][]{}%
  }%
  \providecommand\transparent[1]{%
    \errmessage{(Inkscape) Transparency is used (non-zero) for the text in Inkscape, but the package 'transparent.sty' is not loaded}%
    \renewcommand\transparent[1]{}%
  }%
  \providecommand\rotatebox[2]{#2}%
  \newcommand*\fsize{\dimexpr\f@size pt\relax}%
  \newcommand*\lineheight[1]{\fontsize{\fsize}{#1\fsize}\selectfont}%
  \ifx\svgwidth\undefined%
    \setlength{\unitlength}{495.50922063bp}%
    \ifx\svgscale\undefined%
      \relax%
    \else%
      \setlength{\unitlength}{\unitlength * \real{\svgscale}}%
    \fi%
  \else%
    \setlength{\unitlength}{\svgwidth}%
  \fi%
  \global\let\svgwidth\undefined%
  \global\let\svgscale\undefined%
  \makeatother%
  \begin{picture}(1,0.36669384)%
    \lineheight{1}%
    \setlength\tabcolsep{0pt}%
    \put(0,0){\includegraphics[width=\unitlength,page=1]{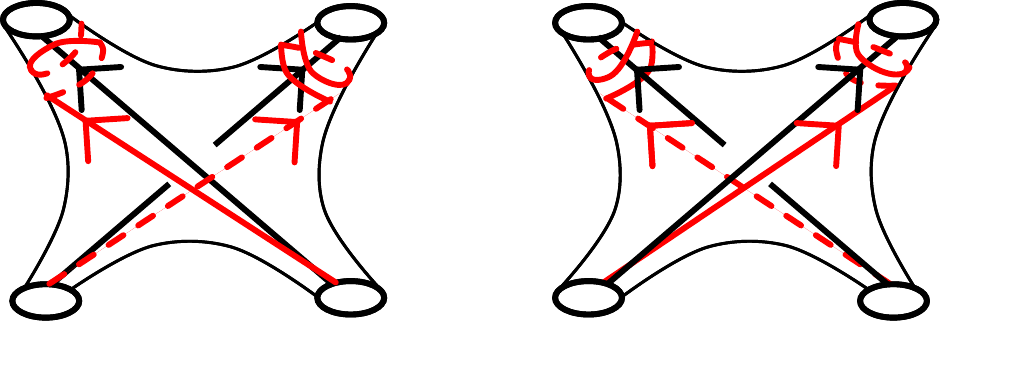}}%
    \put(0.66339145,0.01220747){\color[rgb]{0,0,0}\makebox(0,0)[lt]{\lineheight{1.25}\smash{\begin{tabular}[t]{l}$\alpha^H_{o}(+)$\end{tabular}}}}%
    \put(0.1137691,0.01188005){\color[rgb]{0,0,0}\makebox(0,0)[lt]{\lineheight{1.25}\smash{\begin{tabular}[t]{l}$\alpha^H_{o}(-)$\end{tabular}}}}%
  \end{picture}%
\endgroup%

\caption{\label{f.curves_twistedH} Twisting curves for Jaeger's coefficients. }
\end{figure}
The curves $\alpha^H_{o}(0, +)$ and $\alpha^H_{o}(0, -)$ replaces the oriented resolution with the twisted $\alpha_{o}$ curves as in Figure \ref{f.ct}. 
The curves $\alpha^H_{\hat{o}}$ replaces a crossing by the two $\alpha'$ arcs corresponding to the non-oriented resolution. 
\end{defn}

 Let $D^a_{\sigma_{H, P}}$ be the link diagram $D$ on $\Sigma'$ twisted by $\sigma_{H, P}$ defined locally at each crossing $\chi$ of $D$ as follows. 
\[D^a_{\sigma_{H, P}}(\chi) :=
\begin{cases} 
\text{Replace $\chi$ by $\alpha^H_{o}(0, \sgn(-\chi))$}  &\text{ if $\sigma_{H, P}(\chi) = o$} \\ 
\text{Replace $\chi$ by $\alpha^H_{o}(\sgn(-\chi))$}  &\text{ if $\sigma_{H, P}(\chi) = -\chi$} \\
\text{ Replace $\chi$ by $\alpha^H_{\hat{o}}$}   &\text{ otherwise.} 
\end{cases} \]

\begin{defn}
We define the diagram $D^-$  to be the diagram obtained by switching each crossing $\chi$ of $D$  to $-\chi$. 
\end{defn}

The curves $D^z_{\sigma_{H, P}}$ will be $D^+_{\sigma'_{H, P}}(\alpha)$ twisted by another Kauffman state $\sigma'_{H, P}$ on $D^+$ associated to $\sigma_{H, P}$ on $D$ as in Definition \ref{d.dtbystate}, see Figure \ref{f.ct}, which chooses the non-oriented resolution $\hat{o}$ at every crossing $\chi^+ \in c(D^+)
$ where $\sigma_{H, P}(\chi) = o$ for the corresponding crossing $\chi\in c(D)$ and the oriented resolution otherwise. 

\begin{defn}
We define the diagram $D^+$  to be the diagram obtained by switching each negative crossing $\chi$ of $D$ to a positive crossing. 
\end{defn}
We show the functions
$i^a(\sigma_{H, P}), i^z(\sigma_{H, P})$ can similarly be recovered from the oriented intersections of $D^a_{\sigma_{H, P}}$ and $D^z_{\sigma_{H, P}}$ with the link diagram $D^-$ and link diagram $D^+$. 
\begin{lem}[Geometric interpretation of Jaeger's coefficients] \label{e.giHOMFLY}
We have 
\[ i^a(\sigma_{H, P}) = \frac{i(D^a_{\sigma_{H, P}}(\alpha), D^-)}{2}\]
\[ i^z(\sigma_{H, P}) = \frac{i(D^z_{\sigma_{H, P}}(\alpha), D^+)}{2}.\]
\end{lem}
\begin{proof}
We consider $i^a(\sigma_{H, P})$ first. By construction, the oriented intersection of the two curves $i(D^a_{\sigma_{H, P}}(\alpha), D^-)$ counts $\mp 2$ at a crossing $\chi \in D$ with $\sgn(\chi) = \pm 1$ where $\sigma(\chi) = o$, and it counts $\mp 4$ for a crossing $\chi \in 
D$ with $\sgn(\chi) = \pm 1$ where $\sigma(\chi) = -\chi$, due to the extra twisting for the curves $\alpha^H_{o}(\pm)$. For the rest of the crossings of $D$ or $D^-$, there is no intersection between the two curves. Dividing by $2$ gives the correct local count, which is then added over all the intersections of the curves to give $i^a(\sigma_{H, P})$. 

Next for $i^z(\sigma_{H, P})$, the oriented intersection of the curves $D^z_{\sigma_{H, P}}(\alpha)$ and $D^+$ counts $2$ at every crossing $\chi^+\in c(D^+)$ on which $\sigma_{H, P}(\chi) = o$, and it counts 0 for the rest of the crossings where $D^z_{\sigma_{H, P}}(\alpha)$ has no intersection with $D$.

\end{proof}

Apply Lemma \ref{e.giHOMFLY} to \eqref{e.hksum} to obtain
\begin{equation} \label{e.homflyss}
P_L(a, z) = \sum_{\substack{\sigma_{H, P} \text{ a Jaeger state} \\ \text{ on the diagram $D$} }} \sgn(\sigma_{H, P}) a^{i^a(\sigma_{H, P})} z^{i^z(\sigma_{H, P})} \left(\frac{a-a^{-1}}{z} \right)^{|\sigma^K_{H, P}|-1} 
\end{equation}
where recall $|\sigma^K_{H, P}|$ is the number of circles in the set of circles resulting from applying the Kauffman state $\sigma^K_{H, P}$ associated to $\sigma_{H, P}$ to every crossing in the diagram $D$. 
\subsection{A topological model for the HOMFLY-PT polynomial} \label{ss.tmhomfly}

We use the state sum formulation \eqref{e.hksum} defined in the previous section to define $\Theta_H$ incorporating the intersection pairing $\ll \ , \ \gg_{\alpha^{\sigma_{H, P}}_{H}}$ that recovers the polynomial. 

Now we are ready to prove Theorem \ref{t.main}.

\

{\bf Proof of Theorem \ref{t.main}.} \label{ss.prooftmain}
Let $i^a(\sigma_{H, P}), i^z(\sigma_{H, P})$, and $\sgn(\sigma_{H, P})$ be defined as in the previous section, Section \ref{ss.tmhomfly}, and recall $\ll \ , \ \gg_{\alpha_H^{\sigma_{H, P}}}$ is the intersection pairing defined in Section \ref{ss.intersectionhomfly}. 

Similarly to $\Theta_J$, we define $\Theta_H(D)$. Again for simplicity of notation we will write $\sigma_H = \sigma_{H, P}$, and we will identify $i^a(\sigma_{H, P})$ and $i^z(\sigma_{H, P})$ with oriented intersections of curves via Lemma \ref{e.giHOMFLY}. 
\begin{defn} [Intersection form for the HOMFLY-PT polynomial]\label{intformulaH}
\begin{align*} 
&\Theta_H(D):= \\ 
&:= \sum_{\substack{\sigma_{H} \\  \text{ a Jaeger state} \\ \text{on the diagram $D$}} } \sgn(\sigma_{H}) a^{i^a(\sigma_{H })}z^{i^z(\sigma_{H})} \ll \mathscr{F}^H(\sigma^K_{H}), \mathscr{L}^H(\sigma^K_{H}) \gg_{\alpha_H^{\sigma_{H}}}
\end{align*}
\end{defn}

We restate Theorem \ref{t.main} for the convenience of the reader. 

\begin{tthm}[Theorem \ref{t.main}) (Topological model for the HOMFLY-PT polynomial via ovals]
\hspace{0pt}\\
Let $\Theta_H(D)(a, z) \in \mathbb{Z}[a^{\pm 1}, z^{\pm 1}]$ be the state sum of graded intersection between explicit Lagrangian submanifolds in a fixed configuration space on the Heegaard surface, as in Definition \ref{intHintr}. Then this topological model recovers the HOMFLY-PT invariant: 
 \begin{equation*}
   \Theta_H(D)(a, z) = P_L(a, z)  
 \end{equation*}
\end{tthm}

\begin{proof}
Since $\ll, \gg_{\alpha^{\sigma_{H, P}}_H}$ recovers $\left( \frac{a-a^{-1}}{z}\right)^{|\sigma^K_{H, P}|-1}$ in \eqref{e.homflyss} by Lemma \ref{intJ}, it follows  that $ \Theta_H(D) = P_L(a, z)$. 
\end{proof}

\bibliographystyle{plain}

\bibliography{references.bib}

\end{document}